\documentclass{amsart}

\usepackage{amssymb, amsmath}

\renewcommand{\[}{\begin{equation}\begin{aligned}}
\renewcommand{\]}{\end{aligned} \end{equation}}

\usepackage{verbatim}

\newtheorem{thm}{Theorem}
\newtheorem{prop}[thm]{Proposition}
\newtheorem{lemma}[thm]{Lemma}

\newtheorem{cor}[thm]{Corollary}

\theoremstyle{remark}
\newtheorem{remark}[thm]{Remark}

\theoremstyle{definition}
\newtheorem{definition}[thm]{Definition}

\newtheorem{lem}[thm]{Lemma}

\theoremstyle{definition}

\newtheorem{rmk}{Remark}

\newcommand{\cf}{\emph{cf.} }

\newcommand{\R}{\mathbb{R}}
\newcommand{\C}{\mathbb{C}}
\newcommand{\Z}{\mathbb{Z}}

\newcommand{\norm}[1]{\left\lVert#1\right\rVert}
\newcommand{\Lap}{\Delta}

\title{Singularity formations in Lagrangian mean curvature flow}

\author{Yang Li}
\address{Department of Pure Mathematics and Mathematical Statistics, Cambridge University, Cambridge CB3 0WA, UK }
\email{yl454@cam.ac.uk}

\author{G\'abor Sz\'ekelyhidi}
\address{Department of Mathematics, Northwestern University, Evanston, IL 60208, USA}
\email{gaborsz@northwestern.edu}

\begin{document}

\begin{abstract}
    We study singularities along the Lagrangian mean curvature flow with tangent flows given by multiplicity one special Lagrangian cones that are smooth away from the origin. Some results are: uniqueness of all such tangent flows in dimension two; uniqueness in any dimension when the link of the cone is connected; the existence of nontrivial special Lagrangian blowup limits. We also prove a singular version of Imagi-Joyce-dos Santos's uniqueness result of the Lawlor neck. As an application we prove that in any dimension, singularities that admit a tangent flow given by the union of two transverse planes is modeled on shrinking Lawlor necks at suitable scales. 
\end{abstract}

\maketitle

\section{Introduction}

The main theme of this paper is to study singularity formation mechanisms for the Lagrangian mean curvature flow (LMCF) in general dimensions. There are two principal sources of motivations:

\begin{enumerate}
    \item

    An influential conjectural programme of Thomas-Yau \cite{ThomasYau}, with major updates by Joyce \cite{JoyceEMS}, aims to relate the \emph{existence questions of special Lagrangians} inside a given Calabi-Yau manifold, to \emph{stability conditions} on the derived Fukaya category. The heuristic picture is that if we start with a Lagrangian brane which is well behaved from a Floer theoretic viewpoint, then one would hope that there is some mean curvature flow with surgery, which preserves the Lagrangian condition, and exists for long time, such that the infinite time limit decomposes the Lagrangian into a union of special Lagrangian integral currents, which morally speaking support the stable objects of the derived Fukaya category.

    Previous works of Neves \cite{Neves1}\cite{Neves2} have demonstrated the inevitability of singularity in the flow. To push the programme further, it is essential to understand better the singularity formation mechanisms.

    \item Lagrangian mean curvature flows are one of the most interesting classes of mean curvature flows \emph{in high codimensions}. One salient feature of LMCF, which is markedly in contrast with the (much better understood) mean curvature flow of hypersurfaces, is that under some natural conditions there is \emph{no type I singularity}, namely a singularity modelled on a smooth shrinker~\cite{Neves1}. In this sense the singularities of the LMCF cannot be very elementary.

\end{enumerate}

\subsection{Previous works}

The studies of LMCF typically fall into one of the following categories:

\begin{enumerate}
    \item  Under highly symmetric settings, the Lagrangian at any fixed time can be represented by some profile curve, and the LMCF becomes a variant of the curve shortening flow, with possibly extra singular forcing terms. 

\item  Under certain assumptions, the LMCF can be reduced to the evolution of a scalar potential function, which may be viewed as the parabolic version of the special Lagrangian graph equation. In favourable cases, one can rule out singularity in these cases.

\item There are a number of constructions for LMCF local models such as expanders and translating solitons, via some special ansatz.

\item  One can study the LMCF in some perturbative regime, including gluing style constructions and short time existence results.

\end{enumerate}

Neves \cite{Neves1}\cite{Neves2} went beyond the codimension one or perturbative cases, to establish a structural theory for the tangent flows of the LMCF, and applied it to show that LMCF singularity must form within any given Hamiltonian isotopy class of graded Lagrangians inside any Calabi-Yau manifold, which is compact or controlled near infinity. Neves's theory has been influential, for instance on later works about ancient flows \cite{LambertLotaySchulze}. The nature of the singularity is not explicit from Neves's arguments, which are based on contradictions.

The recent work of Lotay-Schulze-Sz\'ekelyhidi \cite{LSSz2} was able to extract more refined information on the singularity, without symmetry or perturbative assumptions, but relying crucially on the \emph{complex dimension two} assumption. The nature of this result is that under mild ambient assumptions, a priori knowledge of some tangent flow (assuming it is the union of two transverse special Lagrangian planes in $\mathbb{C}^2$), implies the uniqueness of the tangent flow, and determines the type II blow up model (it has to be the Lawlor neck inside $\mathbb{C}^2)$. The main challenge can be explained as follows. Suppose the LMCF $(L_t)$ develops a singularity at the origin at time zero. A \emph{tangent flow} is by definition some geometric measure theoretic subsequential limit for the parabolically rescaled flows $L^i_t= \sigma_i^{-1} L_{\sigma_i^2 t}$ for $\sigma_i\to 0$, which detects information of the original flow $L_t$ at length scale comparable to $\sqrt{-t}$. On the other hand, a \emph{type II blow up limit} is by definition a geometric measure theoretic subsequential limit for the parabolically rescaled flows $L^i_t= \sigma_i^{-1} (L_{\sigma_i^2 (t-t_i)}-x_i)$ where the parabolic centres $(x_i,t_i)\to (0,0)$, and the scaling parameter $\sigma_i\to 0$. \footnote{Our notion of type II blowup allows a priori singularities.} Since LMCF cannot have type I singularity, we need to zoom into a much smaller length scale $o( \sqrt{-t} )$ to detect a nontrivial type II blow up limit.
Usual parabolic tools, such as Huisken's monotonicity formula, are only able to access information at length scale at least $\sqrt{-t}$, and new ideas were required to pass information from the larger length scale to the smaller.

One main reason why complex dimension two was required, is that special Lagrangians in $\mathbb{C}^2$ can be viewed as algebraic curves in $\mathbb{C}^2$ after hyperk\"ahler rotation. This fact allows one to design a quantity whose vanishing characterises the union of two planes, 
and this quantity turns out to satisfy a useful differential inequality under the LMCF in complex dimension two. The principal goal of the present paper is to generalize some aspects of this 2D result to dimensions $n\geq 3$, and our methods are flexible enough to deal with certain other tangent flows which are not planar.

\subsection{Main results}

For simplicity, throughout the paper we assume the following conditions:

\begin{enumerate}
 \item Let $L^0_t$ be an \emph{exact immersed
LMCF} in $\mathbb{C}^n$ for
  $t\in [-1,0)$ with
 \emph{uniformly bounded Lagrangian angle and area ratios}, such that a singularity forms at time zero at the origin. 

 \item Some tangent flow\footnote{We do not assume the uniqueness of the tangent flow, but prove it in some special cases.} at $(x_0,t_0)=(0,0)$ is given by an \emph{integrable special Lagrangian cone}
  $W_0$ with \emph{multiplicity one}, that has a \emph{smooth link}. If $W_0 =
  \bigcup V_i$ is a union of 
  cones with connected links, then integrability means the
  integrability of each $V_i$.
\end{enumerate}

\begin{remark}
More generally  we expect that the methods here could be applied to the setting of a rational Lagrangian mean curvature  flow $L^0_t$ in a compact ambient Calabi-Yau
manifold, using ideas similar to \cite[Section 9]{LSSz2}. 
\end{remark}

\begin{thm}\label{thm:goodblowup}  
  There exists some type II blowup of the flow given by an exact special Lagrangian
  current $L_\infty$ which has tangent cone $W_0$ at infinity, but
  which is not equal to a translate of $W_0$.
\end{thm}

We emphasize that it is assumed that a singularity forms at $(0,0)$, otherwise we could also have a nontrivial tangent flow as a result of several immersed sheets intersecting, in which case there would not be a type II blowup as claimed. 
This result is particularly useful if we have a classification of the exact
special Lagrangian currents asymptotic to $W_0$. When $W_0$ is the
transverse union of two planes,
Joyce-Imagi-Santos \cite{JoyceImagi} have shown that a  \emph{smooth embedded} exact special Lagrangians asymptotic to $W_0$ must be a member of an explicit family known as \emph{Lawlor neck}. In our setting, an important caveat is that the \emph{smoothness} of the type II blowup is not known a priori. This is because our type II blow up detects the length scale at which the Gaussian area of the flowing Lagrangians starts to drop  perceptibly below the density of the tangent flow, which may a priori be much larger than the length scale determined from the magnitude of the second fundamental form.

Instead, we prove a more robust version of the Joyce-Imagi-Santos result allowing for a priori singular exact special Lagrangians admitting smooth approximations, and the smoothness becomes a nontrivial \emph{a posteriori} consequence once we know the objects are in fact Lawlor necks. Combining this ingredient with Theorem \ref{thm:goodblowup}, we shall deduce the following:

\begin{thm}\label{thm:Lawlorblowup}
  Suppose that the tangent flow $W_0$ is the union of two transverse
  planes in $\mathbb{C}^n$. Then the flow admits some type II blowup sequence converging to an exact Lawlor
  neck.
\end{thm}

If instead $W_0$ has
connected link, and is stable in the sense of Joyce~\cite{JoyceII}, then there are no nontrivial exact special Lagrangians asymptotic
to $W_0$  (see
Proposition~\ref{prop:noexactAC}), so Theorem \ref{thm:goodblowup} implies the following (see also Lambert-Lotay-Schulze~\cite[Theorem 1.2]{LambertLotaySchulze} for a related result). 

\begin{cor}\label{cor:nostable}
  If the tangent flow $W_0$ has connected link, then it cannot be stable in the sense
  of Joyce. In particular the $T^2$-cone cannot arise as a tangent
  flow at a (first) singularity.
\end{cor}

\begin{remark}
    In \cite[Section 3.6]{JoyceEMS}, Joyce suggests that in 3D
the Lagrangian mean curvature flow could develop a 
singularity at some time $t_0$ and be continued for later time, such that $L_t$ for $t > t_0$ would
have two singular points modeled on the $T^2$-cone. Note that this
does not contradict the Corollary, since we are not
ruling out the possibility that a tangent flow at $t=t_0$ could be the union of
two $T^2$-cones, for instance. This could be thought of as a generalized neck
pinch.    
\end{remark}

In general, for cones with connected links, by a variant of the proof of Theorem \ref{thm:goodblowup}, we obtain uniqueness of
the tangent flow.

\begin{thm}\label{thm:connected2}
  If $W_0$ has connected link, then $W_0$ is the unique tangent flow
  at the given singular point $(x_0,t_0)=(0,0)$. In particular at the singular time
  $L_{t_0}$ is a Lipschitz graph over $W_0$ in a neighborhood of
  $x_0$.
\end{thm}

\begin{remark}
 It is an interesting question in what way the flow could be continued
if such a connected link singularity forms. Behrndt~\cite{Behrndt} showed that if a
Lagrangian $L$ has isolated conical singularities modelled on \emph{stable}
special Lagrangian cones, then the flow can be defined with initial
condition $L$, allowing for the singularities to move. However,
Corollary~\ref{cor:nostable} implies that the connected link singularity cannot be stable. It seems possible that in the setting of
Theorem~\ref{thm:connected2} the flow typically immediately smoothes
out the singularity, and generically such singularities do not form in
the first place.

\end{remark}

Finally we have the following improvement of the main result of
\cite{LSSz2} to allow any special Lagrangian union of planes in
$\mathbb{C}^2$.
\begin{thm}\label{thm:2dim}
  Suppose that the tangent flow $W_0$ is any (special Lagrangian) union of pairwise
  transverse planes in $\mathbb{C}^2$. Then $W_0$ is the unique
  tangent flow. Moreover at the singular time $t_0$ the Lagrangian $L_{t_0}$
  can be written as a Lipschitz graph over $W_0$ near the singularity.
\end{thm}
In particular this means that in two dimensions all tangent flows
given by multiplicity one special Lagrangians are unique.

\subsection{Outline of the proofs: LMCF ingredients}

We now give a somewhat oversimplified account of some ideas in Theorems \ref{thm:goodblowup}, \ref{thm:connected2} and \ref{thm:2dim}. The key difficulty is to pass quantitative information from a definite scale down to very small length and time scales around parabolic centres near the spacetime singularity.
As in \cite{LSSz2}, the broad strategy of the proof, modulo all technical subtleties, may be viewed as a \emph{parabolic analogue of the Allard-Almgren proof} for the uniqueness of the tangent cone. In this strategy, one considers the rescaled LMCF, and extracts a solution of the drift heat equation (essentially a Jacobi field) by a blow up procedure, with suitable weighted Sobolev norm bounds. The drift heat solution describes the first order deviation of the flow from the tangent flow $W_0$.  The spectral decomposition on the drift heat solution leads to a \emph{growth/decay dichotomy} in rescaled time. In the strict growth case, the flow gradually deviates from the tangent flow, which would cost a definite amount of entropy type quantity, so eventually cannot happen. The strict decay setting can be iterated in many time scales, so that we get some quantitative control down to the length scale associated to the type II blow up limit. The static component is more delicate, and makes essential use of the integrability assumption on $W_0$, morally similar to what happens in the Allard-Almgren argument.

A major technical difficulty is that because the tangent flow is not smooth, the time slice of the flow cannot be globally graphical over $W_0$, and the best hope is \emph{graphicality over a sufficiently long annulus}, where the dirft heat solution is a good approximation to the flow. As in \cite{LSSz2}
the main tool to overcome this issue is a \emph{non-concentration estimate},
which controls the \emph{quantitative rate} at which $L_t$ can diverge from $W_0$ as we approach the origin, and is measured in terms of some variant of the $L^2$ distance to $W_0$. In \cite{LSSz2} such a non-concentration estimate
was proved in the case when $W_0$ was the union of two plane. Writing
$W_0 = \{zw=0\}$ for suitable complex valued functions $z,w$ the
non-concentration estimate relies on estimating the quantity $|zw|$
along the flow. This method could not be generalized to higher
dimensions because it heavily depended on the connection between
special Lagrangians and complex submanifolds in complex dimension
two. Even in complex 2D, if $W_0$ were the union of more than two
planes, then the same ideas would give rise to a non-concentration
estimate, but it would not be strong enough to deduce uniqueness of
the tangent flow.

In this paper we use a different approach, exploiting the fact
that we have good pointwise control of the Lagrangian angle $\theta$
along the flow, because of the heat equation satisfied by $\theta$. This means that if $L_t$ is very close to $W_0$ in
some sense, then at a slightly later time $t'$ the oscillation of $\theta$
will be very small on $L_{t'}$. We can then treat the time slice $L_{t'}$ as being
almost a minimal surface, and morally we run an \emph{effective elliptic version of the Allard-Almgren proof for the uniqueness of tangent cones}, to show that $L_{t'}$ persists to be modelled on a small deformation of $W_0$ down to the length scale where the Gaussian area has dropped by a small definite amount. The upshot is a basic \emph{dichotomy} (see for instance
Proposition~\ref{prop:mainelliptic}): either
$L_{t'}$ remains close to $W_0$ on many scales, giving rise to a good
\emph{non-concentration estimate}, or some rescaling of $L_{t'}$ around some carefully chosen centre must be
close to an exact special Lagrangian $L_\infty$ in $\mathbb{C}^n$ that is asymptotic to
$W_0$ at infinity, but is not (the translation of) a cone. The amount of rescaling here is dictated by the length scale where the Gaussian area drops.  Moreover, by analysing the Laplacian on asymptotically conical special Lagrangians, we note that $L_\infty$  cannot approach $W_0$ with too fast decay rate at infinity, and
this leads to some quantitative estimate for how quickly $L_{t'}$ could have diverged
from $W_0$.

Theorem~\ref{thm:goodblowup} essentially exploits this dichotomy. The second alternative would lead to the desired type II blow up limit, while the first alternative provides a good
non-concentration estimate, which would lead to a 3-annulus type inequality for the $L^2$-distance in the rescaled time, which would then feed into the previous growth/decay dichotomy for the drift heat solution extracted from the rescaled LMCF. In the idealised case of strict decay, then upon many iterations, we would deduce that $L_t$ is close to $W_0$ on many spacetime scales, and then one can extract an ancient flow from a suitably rescaled sequence, which would again provide the desired type II blow up limit by appealing to \cite{LambertLotaySchulze}.

We now comment on the subtle issues arising from the static component in the spectral decomposition of the drift heat solution, which also appeared in \cite{LSSz2}. The Allard-Almgren strategy would be to kill the static component by deforming the model special Lagrangian $W_0$, to reduce to the strictly decaying case, which could then be iterated. This strategy works if the link of $W_0$ is \emph{connected}, and Theorem \ref{thm:connected2} would then follow through. On the other hand, if $W_0$ has several connected components, then each component can be rotated in $\mathbb{C}^n$ independently to special Lagrangians with \emph{different phase angles}, so their union is no longer a special Lagrangian, but nevertheless at the linearised level we still get some nontrivial static solution to the drift heat equation.  To deal with this issue, we use a quantitative variant of some earlier idea in Neves's structure theory \cite{Neves1}, which exploits a differential inequality combining the Lagrangian angle and the Lagrangian potential, and some local connectedness assumption on the Lagrangian. This leads to the technical sub-problem of estimating the Lagrangian potential. To do this efficiently, one needs to go beyond the Neves argument, and show that  the graphical region extends to a sufficiently small length scale at a certain \emph{quantitative rate}.
This step depends again strongly on the non-concentration estimate, which is another crucial ingredient of Theorem \ref{thm:goodblowup}.

Finally, the proof mechanism of Theorem  ~\ref{thm:2dim} depends on the following dichotomy, which is a byproduct of our non-concentration estimate proof (see Prop. \ref{prop:mainelliptic}). Either the non-concentration estimate holds down to sufficiently small length scale in order to run the quantitative version of the Neves argument as before, or some entropy type functional needs to drop by a perceptible amount. The second alternative leads to a
\emph{discrete version of the Lojasiewicz inequality}, which again controls the drift of $L_t$ over many spacetime scales.

\subsection{Outline of the proof: Lawlor neck}

The proof of Theorem \ref{thm:Lawlorblowup} relies on the singular version of Joyce-Imagi-Santos  characterisation of the Lawlor neck \cite{JoyceImagi}, see Theorem \ref{singularuniqueness} for the detailed statement, whose proof occupies section \ref{section:characterisationLawlor}, which can be read independent of the rest of the paper. We now give a slightly oversimplified outline.

Joyce-Imagi-Santos \cite{JoyceImagi} proved that in dimension $n\geq 3$, any exact and smoothly embedded special Lagrangian $L$, asymptotic to the union of two transverse special Lagrangian planes, must be a Lawlor neck. Their original argument relies on \emph{Floer theory}, and we provide a detailed sketch in 
section \ref{LawlorneckJoyceImagi}. In very rough terms, they first partially compactify $\mathbb{C}^n$ into a Liouville manifold obtained as the plumbing of two copies of $T^*S^n$, and the special Lagrangian can be compactified to $\bar{L}$ by adding some points at infinity. Using the classification of objects in the derived Fukaya category, and by studying the intersections with some distinguished Lagrangians, they show that $\bar{L}$ must be isomorphic to the compactification of the Lawlor neck as an object of the derived Fukaya category. The rest of their argument is modelled on the uniqueness theorem of Thomas-Yau \cite{ThomasYau}, whose main idea is to use judicious Hamiltonian perturbations to  eliminate the Floer degree $0,n$ intersection points between $\bar{L}$ and the Lawlor neck. The  caveat is that due to the subtlety of compactification, the degree $0,n$ intersection points within $\mathbb{C}^n$ can be eliminated, but at spatial infinity they cannot be eliminated. The main gist in \cite{JoyceImagi} is that if $\bar{L}$ does not coincide with the Lawlor neck, then one can get a contradiction by producing a holomorphic curve with zero area. The more technical part of \cite{JoyceImagi} is to control the intersection points, and relies crucially on the real analyticity of the Lagrangians.

In our case, the special Lagrangian $L$ arises as the geometric measure theoretic limit of a sequence of exact and  smoothly embedded Lagrangians $L_i$ whose Lagrangian angles are $C^0$-small, which reflects the fact that the type II limit is a rescaling limit of smooth Lagrangians in the LMCF. Since we do not know the a priori smoothness of the limiting special Lagrangian $L$, we cannot assume that the convergence is in the $C^1$-topology. The Floer theoretic argument can only be applied directly to the smooth approximations instead of $L$, and the difficulty is that the sequence does not have strong uniform control. In particular, we cannot eliminate the Lagrangian intersections inside $\mathbb{C}^n$, and all the arguments in \cite{JoyceImagi} relying on real analyticity break down.

We use a more subtle holomorphic curve argument, inspired by our previous work on quantitative Thomas-Yau uniqueness \cite{LiquantitativeTY}.

We recall the well known fact that the \emph{area of the holomorphic curve} with boundary on exact Lagrangians, can be computed in terms of the \emph{difference of the Lagrangian potentials} evaluated at the Lagrangian intersections. After some technical preliminary estimates on Lagrangian potentials, we choose a small Hamiltonian perturbation of the  Lawlor neck, by solving an auxiliary \emph{Poisson equation}, with forcing term prescribed by some carefully chosen function depending on the Lagrangian potentials. Consequently, at the degree $0,n$ Lagrangian intersections between this perturbed Lawlor neck and $L_i$, the difference of the Lagrangian potentials now have \emph{preferred signs}, modulo small estimable error term. Using Fukaya categorical ingredients similar to \cite{JoyceImagi}, we obtain holomorphic curves with boundary on these two Lagrangians. The main gist is that \emph{the area of this holomorphic curve will be small}, and the holomorphic curve can be required to pass through any prescribed point on the perturbed Lawlor neck. Using an area monotonicity formula for holomorphic curves with Lagrangian boundary, we deduce that the Lawlor neck lies inside a $C^0$-small neighbourhood of $L_i$. Passing the estimates to the limit, we eventually show that the Lawlor neck coincides with the special Lagrangian $L$.

\subsection{Organisation}
We begin with some background on special Lagrangians and LMCF in section \ref{section:Preliminaries}, such as analysis on the linearised equation, and some consequences of the monotonicity formula. Section \ref{section:timeslice} treats the fixed time slice as an elliptic problem, and proves the dichotomy that either there is a good blow up sequence, or there is some non-concentration estimate. Section \ref{section:Decayalongflow} builds upon this to prove decay estimates on the rescaled LMCF in time, and essentially shows that the rescaled LMCF remains graphical over some deformation of the tangent flow, as long as the Gaussian area does not drop significantly. Section \ref{section:characterisationLawlor} concerns the singular version of the Joyce-Imagi-Santos characterisation of the Lawlor neck, with a more Floer theoretic flavour, and can be read independent of the rest of the paper. Section \ref{section:main} combines the above to prove the main theorems.

\subsection*{Acknowledgements}
We thank Nick Edelen, Jason Lotay and Felix Schulze for helpful discussions. 
Y.L. is a Royal Society University Research Fellow at Cambridge University, and this work started while Y.L. was a Clay Research Fellow based at MIT. G. Sz. 
was supported in part by NSF grant DMS-2203218.

\section{Preliminaries}\label{section:Preliminaries}
\subsection{Special Lagrangian cones}
Let $W_0\subset \mathbb{C}^n$ be a special Lagrangian cone, smooth
away from the origin. We can write
\[ W_0 = \bigcup_{i=1}^l V_i, \]
where $V_i$ are special Lagrangian cones with connected
links. Throughout we will assume that the $V_i$ are all integrable
cones, in the sense that all quadratic growth harmonic functions on
$V_i$ correspond to a deformation of $V_i$ in a family of special
Lagrangian cones (with the same Lagrangian angle as $V_i$).

To be more precise we first have the following definition.
\begin{definition}
  Suppose that $V_s$ is a family of Lagrangian cones for $s\in
  (-1,1)$, which can be written as the (normal) graph of vector fields
  $v_s$ over the special Lagrangian cone $V$. Define the 1-forms
  $\eta_s = \iota_{v_s} \omega$ on $V$ using the symplectic form
  $\omega$. We say that the family $V_s$ is generated by the 1-form
  $\eta$ over $V$ if $s^{-1}\eta_s \to \eta$ on $V$ as $s\to 0$. 
\end{definition}

\begin{lemma}
  Suppose that we have a family $V_s$ of special Lagrangian cones, possibly with
  different Lagrangian angles than $V$, generated by the 1-form
  $\eta$ over $V$. Then $\eta = df$ for a homogeneous quadratic function $f$ on
  $V$ satisfying that $\Delta f$ is constant. 
\end{lemma}

Given this, the integrability of $V$ means that for every quadratic
growth harmonic function $f$ on $V$ there is a family of special
Lagrangian cones $V_s$ generated by $df$. Note that we always have the
rotations $V_s = e^{is} V$, where $e^{is}$ acts by multiplication on each
factor of $\mathbb{C}^n$, and the $V_s$ are special Lagrangians, with
different angles from $V$. This family is generated by the 1-form $d |x|^2$
(up to scalar multiple), and $\Delta |x|^2 =
2n$ on $V$. It follows that if the special Lagrangian cone $V$ is integrable, then for any
quadratic function $f$ on $V$ with constant Laplacian, there is a
family of special Lagrangian cones $V_s$ generated by $df$.

\subsection{The drift heat equation on cones}
We will need certain results about the drift heat equation for
functions and 1-forms on special Lagrangian cones
$V$. Let us denote by
$W^{1,2}_0(V)$ the closure of the space of smooth functions of compact
support in $V\setminus\{0\}$, under the norm
\[ \begin{aligned}
    \Vert f \Vert_{W^{1,2}}^2 &= \Vert f\Vert_{L^2}^2 + \Vert \nabla
  f\Vert_{L^2}^2, \\
  &= \int_V (|f|^2 + |\nabla f|^2)\,
  e^{-\frac{|x|^2}{4}}\, d\mathcal{H}^n. \end{aligned} \]
Let us also denote by $L^2(V)$ the corresponding Gaussian $L^2$
space. Using the logarithmic Sobolev inequality~\cite{Ecker} on $V$, it follows
that $W^{1,2}_0(V)$ is compactly embedded in $L^2(V)$. As a
consequence the drift Laplacian defined by 
\[ \mathcal{L}_0 f = \Delta - \frac{1}{2} x. \nabla f \]
has discrete spectrum, and we can decompose $L^2(V)$ into an
orthogonal sum of eigenspaces of $\mathcal{L}_0$. The main application
of this for us is
the following, regarding solutions of the drift heat equation on
1-forms. 

\begin{lemma}\label{lem:decomp1}
  Suppose that $\eta(x,t)$ is a smooth family of exact 1-forms on
  $V\setminus\{0\}$ for $t\in (0,1)$, satisfying the drift heat
  equation
  \[ \frac{\partial \eta}{\partial t} = \Delta \eta + \frac{1}{2}
    (\eta - x\cdot \nabla \eta), \]
  where $\Delta$ is the Hodge Laplacian. 
  Assume that $\Vert \eta(\cdot, t)\Vert_{L^2(V\setminus B_1)}$
  is uniformly bounded for all $t$ and in addition that one of the following holds:
  \begin{itemize}
    \item[(i)] $\eta$ is uniformly bounded on $B_2\cap V$
    \item[(ii)] $n \geq 3$ and $|x|^{1.1} \eta$ is uniformly bounded
      on $B_2\cap V$
    \item[(iii)] $n=2$, so $V=\mathbb{R}^2$, and $|x|^{1.1} \eta$ as
      well as $d^*\eta$ are uniformly bounded on $B_2\cap V$. 
    \end{itemize}
    Then in cases (i) and (ii) we can write
  \[ \label{eq:etadecomp1} \eta = \sum_{j=1}^\infty a_j e^{(\lambda_j + 1) t } d\phi_j, \]
  where the $\phi_j$ form an orthonormal basis of eigenfunctions of 
  $\mathcal{L}_0$ acting on functions with $\mathcal{L}_0\phi_j = \lambda_j \phi_j$.
  In case (iii) we can write $\eta = a_0 e^{t/2} d\ln |x| + \eta_1$,
  where $\eta_1$ has an expansion as in \eqref{eq:etadecomp1}. 
\end{lemma}
\begin{proof}
  The case (iii) is contained in \cite[Proposition 5.2]{LSSz2}, so we
  will focus on (i) and (ii). 
  We can write $\eta = df$, and modifying the $f$ by a suitable
  time dependent constant we can arrange that
  \[ \frac{\partial f}{\partial t} = \Delta f + \frac{1}{2}( 2f -
    x\cdot \nabla f). \]
  Using that the spectrum of $\mathcal{L}_0$ is discrete, and the kernel consists of the constants, we have the Poincar\'e inequality
  \[ \int_V (f- \underline{f})^2 \, e^{-|x|^2/4}\, d\mathcal{H}^n \leq C\int_V |df|^2\,
    e^{-|x|^2/4}\, d\mathcal{H}^n, \]
  where $\underline{f}$ is the average of $f$ with respect to
  $e^{-|x|^2/4}d\mathcal{H}^n$. Using this we see that $f, df
  \in L^2(V)$, while at the same time in case (i) $f, df$
  are uniformly bounded near 0, and in case (ii) $|x|^{0.1}f$ and
  $|x|^{1.1}df$ are bounded near 0. We can write
  \[ f = \sum_{j=1}^\infty d_j(t) \phi_j, \]
  and using the bounds on $f$ show that $d_j'(t) = (\lambda_j + 1) d_j(t)$. The result
  follows.  
\end{proof}

The linearization of the
Lagrangian mean curvature 
flow near a special Lagrangian cone $V$ is given by the drift heat
equation
\[\label{eq:f1} \frac{\partial\eta}{\partial t} = \Delta\eta +
  \frac{1}{2}(\eta - x\cdot \nabla\eta).
  \]
for exact 1-forms $\eta$. Of particular importance are the
static solutions. 
\begin{lemma}\label{lem:static1}
  Suppose that $\eta(x,t)$ is a smooth, exact, time independent solution of
  \eqref{eq:f1} on $(V\setminus\{0\}) \times (0,1)$. As in
  Lemma~\ref{lem:decomp1} assume that $\eta$ satisfies the bound (i)
  or (ii) near 0, and is in $L^2$ on $V$. Then we can write $\eta =
  d(f + c|x|^2)$ for a constant $c$ and a 
  homogeneous harmonic function of degree 2 on $V$.
\end{lemma}
\begin{proof}
  From Lemma~\ref{lem:decomp1} we find that $\eta = df$, where
  $\mathcal{L}_0 f = - f$, that is
  \[ \label{eq:Df1} \Delta f + \frac{1}{2}( 2f - x\cdot \nabla f) = 0. \]
  By direct calculation we find that $h = \Delta f$ satisfies
  \[ \label{eq:heq} \Delta h- \frac{1}{2} x\cdot \nabla h = 0. \]
  For the calculation note that on a cone the contraction of the Ricci
  tensor with a radial vector field vanishes. 

  We claim that the following integration by parts is justified: 
  \[\label{eq:intparts} \int_V |\nabla h|^2 e^{-|x|^2/4}\,
    d\mathcal{H}^n = -\int_V h(\Delta h - \frac{1}{2} x\cdot \nabla
    h)\, e^{-|x|^2/4}\,d\mathcal{H}^n = 0. \]
  To see this we need to show that the boundary term
  \[ \label{eq:bdterm} \int_{\partial B_r(0)\cap V} h\partial_n h\,
    e^{-|x|^2/4}\, d\mathcal{H}^{n-1} \]
  converges to zero as $r\to 0$ or $r\to \infty$.

  Let us first consider $r\to 0$ in \eqref{eq:bdterm}. First we assume
  $n\geq 3$. Using that
  $\eta = df$ and that $|x|^{1.1}|\eta|$ is bounded near 0, we have
  that $|x|^{0.1}|f|$ and $|x|^{1.1}|\nabla f|$ are bounded near
  0. Equation~\eqref{eq:Df1} implies that $|x|^{0.1}\Delta f = |x|^{0.1}h$ is also
  bounded near 0. Applying Schauder estimates on annuli $A_{2r, 2}\cap
  V$ to \eqref{eq:heq} we find that $|x|^{1.1}|\nabla h|$ is bounded
  near 0. In particular as $r\to 0$ the boundary integral
  \eqref{eq:bdterm} is of order $r^{n-1}r^{-0.1}r^{-1.1} = r^{n-2.2}$,
  which tends to zero as $r\to 0$, since $n \geq 3$. If $n=2$, then
  $f, \nabla f$ are bounded near 0. Since   $V$ is a plane in this
  case, we find that $f$ is a distributional solution of
  \eqref{eq:Df1} across the origin, and so $f$ is actually smooth
  across 0. It follows that the boundary term \eqref{eq:bdterm}
  converges to zero as $r\to 0$.

  We now consider $r \to \infty$. For this we use the growth estimate
  Colding-Minicozzi~\cite[Theorem 0.7]{CMgrowth}. Although $V$ is
  possibly singular at the origin, as stated in \cite{CMgrowth} below
  Theorem 0.7, the result holds even if the conditions are only
  satisfied outside of a compact set. We can apply the theorem to both
  $u=f$ and $u=\eta$. The conclusion is that the integrals
  \[ \int_{\partial B_r(0)\cap V} |f|^2\, d\mathcal{H}^{n-1}, \quad
    \int_{\partial B_r(0)\cap V} |\nabla f|^2\, d\mathcal{H}^{n-1} \]
  have polynomial growth in $r$ as $r\to\infty$. Using \eqref{eq:Df1},
  the analogous integral of $|\Delta f|^2$ also has polynomial growth,
  and so by elliplic $L^p$ theory the integral of $|\nabla^2 f|^2$
  also has polynomial growth. Differentiating \eqref{eq:Df1} we
  conclude that
  \[ \int_{\partial B_r(0)\cap V} |\nabla h|^2\, d\mathcal{H}^{n-1}\]
  also has polynomial growth, and then using the Cauchy-Schwarz
  inequality we find that \eqref{eq:bdterm} converges to zero as $r\to
  \infty$. This justifies the integration by parts formula
  \eqref{eq:intparts}. 
  
    From \eqref{eq:intparts} it follows that
    $h = \Delta f$ is constant. For a suitable $c$ it follows that
  $\Delta( f - c |x|^2) = 0$, and from \eqref{eq:Df1} we then find
  that $f$ is homogeneous of degree 2. This implies the result. 
\end{proof}

We will need the following $L^2$ 3-annulus lemma, similar to
\cite[Lemma 6.1]{LSSz2} (see also Colding-Minicozzi~\cite{CM} as
well as
Simon~\cite[Lemma 3.3]{SimonIsolated}). When $n=2$ then, as in
\cite{LSSz2}, we need to account for the fact that the $d\ln |x|$
terms are not in $L^2$ near the origin. Because of this, for a
solution
\[\label{eq:etadecomp2} \eta &= a_0 e^t d\ln |x| + \eta_1 \\
  &= a_0 e^t d\ln |x| +\sum_{j=1}^\infty a_j e^{(\lambda_j +
    1) t } d\phi_j, \]
of \eqref{eq:f1} (where $a_0=0$ unless we are in the case $n=2$ so
$W_0$ is a union of planes), we define
\[ \label{eq:normdefn10} \Vert \eta(t)\Vert^2 = |a_0 e^t|^2 + \Vert
  \eta_1(t)\Vert_{L^2(W_0)}^2. \]

\begin{lemma}\label{lem:L23ann2}
 For any $d \in \mathbb{R}$ there are $0 < \lambda_1 < \lambda_2$
 depending on $W_0, d$ with the
  following property. 
  Let $\eta(t)$ be a smooth exact solution of \eqref{eq:f1} on
  $W_0\setminus\{0\}$ for $t\in [0,2]$, satisfying the bound (i), (ii)
  or
  (iii) in Lemma~\ref{lem:decomp1} in $B_2$, 
  and in $L^2$ on $W_0\setminus B_1$. We then have
  \begin{enumerate}
    \item If $\Vert \eta(1)\Vert \geq e^{d+\lambda_1} \Vert \eta(0)\Vert$,
      then $\Vert \eta(2)\Vert \geq e^{d+\lambda_2} \Vert
      \eta(1)\Vert$,
    \item If $\Vert \eta(1)\Vert \geq e^{-d+\lambda_1} \Vert \eta(2)\Vert$,
      then $\Vert \eta(0)\Vert \geq e^{-d+\lambda_2} \Vert
      \eta(1)\Vert$,
    \item If $\eta$ contains no term in the decomposition
      \eqref{eq:etadecomp2} with degree $\lambda_j+1 = d$,
      then the conclusion of either (1) or (2) above must hold.
   \end{enumerate}
 \end{lemma}

 Finally we have the following estimate for derivatives of a solution
 of the drift heat equation, similar to Proposition 5.3 in
 \cite{LSSz2}. 
 \begin{lemma}\label{lem:pointwise2}
Let $\eta(t)$ be a smooth exact solution of \eqref{eq:f1} on 
  $W_0\setminus\{0\}$ for $t\in [0,1]$, satisfying the bound (i) or
  (ii)  in Lemma~\ref{lem:decomp1} in $B_2$, 
  and in $L^2$ on $W_0\setminus B_1$. 
   For some $p_k > 1$ and $C_k > 0$ depending only on the cone
   $W_0$, for $k\geq 0$ we have
   \[ |\nabla^k \eta(1)|^2(x) \leq C_k \max\{|x|^{-k},1\} e^{\frac{|x|^2}{4p_k}}
     \int_{W_0} |\eta(0)|^2\, e^{-|x|^2/4}. \]
  If $\eta$ satisfies the bound (iii) in Lemma~\ref{lem:decomp1}, then
  $\eta$ can be decomposed as in \eqref{eq:etadecomp2} on each plane
  component of $W_0$. For the corresponding smooth pieces $\eta_1$ we
  have
  \[ |\nabla^k \eta_1(1)|^2(x) \leq C_k e^{\frac{|x|^2}{4p_k}}
    \int_{W_0} |\eta_1(0)|^2\, e^{-|x|^2/4}.\]
\end{lemma}

It is worth observing also that the logarithmic term in the
decomposition \eqref{eq:etadecomp2} satisfies $d^*d\ln |x| = 0$, at
least away from the origin, and so under the setting (iii) in
Lemma~\ref{lem:decomp1} we still have the estimate
\[ |d^*\eta(1)|^2(x) \leq C e^{\frac{|x|^2}{4p_k}}
  \int_{W_0} |\eta(0)|^2\, e^{-|x|^2/4}, \]
even though $\nabla\eta(1)$ may blow up at rate $|x|^{-2}$ at the
origin.

 \subsection{Special Lagrangian currents asymptotic to cones}
In this section we prove some results about the decay rate of exact special
Lagrangian currents to their asymptotic cones.

\begin{definition}
  Following \cite[Definition 2.10]{CollinsLi}, we define an exact special
  Lagrangian current $L\subset\mathbb{C}^n$ to be a multiplicity one, closed
  integral current, that is special Lagrangian, and
  such that $\frac{1}{2}\lambda|_{L_{reg}} = df$
  for some $f: L_{reg} \to \mathbb{R}$.

  Let $W$ be a special Lagrangian cone, with smooth link, as
  before. We say that a special Lagrangian current $L$ is \emph{asymptotic
  to $W$}, if the tangent cone at infinity of $L$ is given by
$W$. Furthermore we say that $L$ is asymptotic to $W$ \emph{at rate
  $\delta< 2$}, if outside of a ball $B_{R_0}(0)$, $L$ can be written as the
  graph of a normal vector field $v$ over $W$ satisfying $|\partial^k
  v| = O(r^{\delta-k-1})$ for all $k\geq 0$. Note that by using the Schauder
  theory and rescalings, the estimate for $k=0$ implies the ones for $k > 0$, as long
  as $\delta < 2$ so that $|v|/r \to 0$. 
\end{definition}

If $W$ is an integrable special Lagrangian cone, as in the
setting of this paper, then it follows from Allard-Almgren~\cite{AA}
that whenever $L$ is a special Lagrangian with a tangent cone at
infinity given by $W$ (with multiplicity one), then $L$ is in fact
asymptotic to $W$ at some rate $\delta < 2$. 

We have the following result stating that exact special Lagrangian
currents cannot converge to a cone $W$ too quickly at infinity if they
differ from $W$.

\begin{prop}\label{prop:slowdecay}
  Let $W\subset\mathbb{C}^n$ be a special Lagrangian cone with smooth
  link, and $L$ an exact special Lagrangian
  current asymptotic to $W$ at rate $\delta$.
  \begin{itemize}
  \item[(a)] If $\delta < 2-n$, then $L=W$. 
  \item[(b)] If the link of $W$ is connected, and $\delta < 0$, then
    $L=W$. 
  \end{itemize}
\end{prop}
\begin{proof}
  Outside of a large ball $B_R(0)$ we write $L$ as the graph of $v$
  with $|\partial^k v| = O(r^{\delta-k-1})$. This implies that on $L$
  we have  $|\nabla f| = |x^\perp| = O(r^{\delta -1})$. Since in either case
  above we have $\delta < 0$, we can assume that $|f| = O(1)$.
  Note also that $f^2$ is weakly subharmonic on $L$ (see
  \cite[Lemma 2.13]{CollinsLi}), while $f$ is harmonic on the
  regular part of $L$. 
  \begin{itemize}
  \item[(a)]
    We use that $f^2$ is subharmonic, and on the regular part
    $L_{reg}$ we have $\Delta f^2 = 2|\nabla f|^2$. This
    implies that for large balls $B_r(0)$
    \[ \label{eq:fint10} 2\int_{L_{reg}\cap B_r(0)} |\nabla f|^2 \leq \int_{L \cap
        B_r(0)} \Delta f^2 = \int_{\partial B_r(0) \cap L} 2f
      \nabla_n f, \]
    where $n$ is the unit outward pointing normal to $\partial
    B_r(0)$. Using the bound for $f, \nabla f$ we have
    \[ \int_{L_{reg}\cap B_r(0)} |\nabla f|^2 \leq C r^{\delta-1}r^{n-1}, \]
    for large $r$. If $\delta < 2-n$, this implies that $f$ is
    constant. 
  \item[(b)]
    The estimate $|\nabla f| = O(r^{\delta-1})$, and the fact that the
    link of $W$ is connected,  implies that the
    oscillation $\mathrm{osc}_{\partial B_r(0)\cap L} f= O(r^{\delta})$ for large
    $r$. Modifying $f$ by a constant, we can assume that $|f| =
    O(r^{\delta})$ on $\partial B_r(0)$. Using that $f^2$ is
    subharmonic, we can apply the maximum principle, and let $r \to
    \infty$ to deduce that $f$ is constant. 
  \end{itemize}
    In either case $L$ must be a cone if $f$ is constant, and the fact
    that $L$ is asymptotic to $W$ implies that then $L=W$. 
  \end{proof}

  Note that the condition in (a) is sharp, since the Lawlor neck is asymptotic
  to the union of two planes at rate $2-n$. We expect that the
  condition in (b) is not sharp, but it is sufficient for our
  purposes. 

  Let us briefly recall the notion of a stable special Lagrangian cone
  from Joyce~\cite[Definition 3.6]{JoyceII}. We will only be concerned
  with the case when the link of the cone is connected. 
  \begin{definition}\label{defn:stableSL}
    Let $W$ be a special Lagrangian cone with connected link. Then $W$
    is stable if the space of harmonic functions on $W$ with at most
    quadratic growth are given by the span of the constant functions,
    the coordinate functions, and the degree 2 functions induced by
    the action of $SU(n)$.  
  \end{definition}
  There are only a few known examples of stable special Lagrangian
  cones, such as the $T^2$-cone~\cite{JoyceII} and a few
  others~\cite{Ohnita}. The following result shows that there are no
  non-trivial exact special Lagrangians asymptotic to a stable special
  Lagrangian cone. For the $T^2$-cone $C$ this also follows from the
  classification of all special Lagrangians asymptotic to $C$ by
  Imagi~\cite[Theorem 1.1]{Imagi}. 

  \begin{prop}\label{prop:noexactAC}
    Let $W$ be a stable special Lagrangian cone (with smooth connected
    link). Then the only exact special Lagrangian currents asymptotic
    to $W$ are the translates of $W$.
  \end{prop}
  \begin{proof}
    We can assume that $n > 2$ since the only option for $W$ in
    $\mathbb{C}^2$ is a plane, in which case the result follows from
    the monotonicity formula for minimal surfaces. 
    Suppose that $L$ is an exact special Lagrangian current asymptotic to $W$
    at infinity. By \cite[Theorem 4.7]{Imagi} we have that a
    translate $L-x$ is
    asymptotic to $W$ at rate 0. This means that near infinity $L-x$
    is the graph of a one form $\eta=du$, where $u = u_0 + \tilde{u}$,
    $u_0$ is a bounded harmonic function, and $\tilde{u}$ decays at infinity. Modifying
    $u_0$ by a constant, and using that there are no homogeneous
    harmonic functions on $W$ with degrees in $(2-n,0)$, we can ensure
    (see \cite[Theorem 7.11(a)]{JoyceI}) that $L-x$ is an exact
    special Lagrangian, asymptotic
    to $W$ at rate $\delta$ for any $\delta > 2-n$. It follows from
    Proposition~\ref{prop:slowdecay}(b) that $L-x = W$. 
  \end{proof}

\subsection{Preliminaries for the Lagrangian mean curvature flow}
Throughout the paper we will consider a solution $L_t$ of the
Lagrangian mean curvature 
flow in $\mathbb{C}^n$ for $t\in [-1,0)$. We
assume that we have uniformly bounded Lagrangian angle
and  area ratios:
\[ |\theta| \leq C_0, \qquad r^{-n}|L_t \cap B(p, r)| \leq
    C_0\text{ for all }p\in \mathbb{C}^n, r > 0, \]
  where $|A|$ denotes the $n$-dimensional Hausdorff measure.
These conditions are preserved along the flow, so it is enough to
assume them for the initial condition $L_{-1}$ (see
Neves~\cite{Neves1}). In addition the conditions will hold with the same
constant for any rescalings of the $L_t$, in particular along any
parabolic rescaling of the flow. 

We define the Gaussian area of $L$ (centered at the origin, at scale
1) to be
\[ \Theta(L) = \frac{1}{(4\pi)^{n/2}} \int_L e^{-|x|^2/4}\,
  d\mathcal{H}^n. \]
More generally we define the Gaussian area centered at $x$ at scale
$r$ by
\[ \Theta(L, x, r) = \Theta(r^{-1}(L-x)). \]
In terms of the backwards heat kernel
\[ \rho_{x_0, t_0} (x,t) = \big( 4\pi (t_0 - t)\big)^{n/2} exp\left(
    -\frac{|x-x_0|^2}{4(t_0-t)}\right), \]
for any $t$ we have 
\[ \Theta(L,x_0,r_0) = \int_L \rho_{x_0, t+r_0^2}(x,t)\,
  d\mathcal{H}^n(x). \]
Huisken's monotonicity formula states that for any smooth family of
functions $f_t$ on $L_t$ with polynomial growth at infinity, we have
\[ \frac{d}{dt} \int_{L_t} f_t \rho_{x_0, t_0}\, d\mathcal{H}^n(x) =
  \int_{L_t} \left( \frac{d}{dt} f_t - \Delta f_t - \left| H +
      \frac{(x-x_0)^\perp}{2(t_0-t)}\right|^2 f_t \right)\, \rho_{x_0,
    t_0}\, d\mathcal{H}^n(x). \]
Since $\theta$ satisfies the heat equation along the Lagrangian mean
curvature flow, and $|\nabla \theta| = |H|$, for any constant
$\theta_0$ we have
\[ \label{eq:thetamon} \frac{d}{dt}\int_{L_t} |\theta - \theta_0|^2 \rho_{x_0, t_0}\,
  d\mathcal{H}^n(x) \leq -\int_{L_t} 2|H|^2 \rho_{x_0, t_0}\,
  d\mathcal{H}^n(x). \]
Using this, we have the following, whose proof is essentially
identical to Neves~\cite[Theorem A]{Neves1} (see also \cite[Proposition 4.1]{LambertLotaySchulze}).
\begin{prop}\label{prop:N1}
  Suppose that we have a sequence of flows $L^i_t$ for $t\in [-1,0)$ satisfying 
  uniform angle and area ratio bounds. In addition suppose that for
  some $\theta_0$ we have
  \[ \label{eq:theta1} \lim_{i\to \infty} \int_{L^i_{-1}} |\theta-\theta_0|^2 e^{-|x|^2 /
      4}\, d\mathcal{H}^n(x) = 0. \]
  Then there is an integral special
  Lagrangian current $L_\infty$ such that after passing to a
  subsequence, for every compactly supported smooth function
  $\phi$ and all $t\in [-1, 0)$ we have
  \[ \label{eq:limL1} \lim_{i\to \infty} \int_{L^i_t} \phi\, d\mathcal{H}^n =
    \mu_\infty(\phi), \]
  where $\mu_\infty$ is the Radon measure of $L_\infty$. Moreover this
  convergence is uniform in $t$ over $[-1,0)$. 
\end{prop}
\begin{proof}
  We sketch the main steps for the readers' convenience. From
  \eqref{eq:theta1}, and the uniform bound on area ratios and the angle,
  we have
  \[ \lim_{i\to\infty} \int_{L^i_{-1}} |\theta-\theta_0|^2 \rho_{0, 1}\,
    d\mathcal{H}^n = 0, \]
  i.e. we can change the center of the backward heat kernel to the
  later time $t=1$. 
  By the
  monotonicity formula applied to $|\theta - \theta_0|^2$ we then have
  \[ \label{eq:Hint2} \lim_{i\to \infty} \int_{-1}^0 \int_{L^i_t} |H|^2\, \rho_{0,1}\,
    d\mathcal{H}^n\, dt = 0. \]
  This means on the one hand that for any compact set $K$ we have
  \[  \label{eq:Hint21} \lim_{i\to \infty} \int_{-1}^0 \int_{L^i_t\cap K} |H|^2\,
    d\mathcal{H}^n\, dt = 0, \]
  and on the other hand we can find an $a_0\in [-1/2,-1/4]$ such that for all $R > 0$ we
  have 
  \[ \lim_{i\to \infty} \int_{L^i_{a_0} \cap B(0,R)} |H|^2\,
    d\mathcal{H}^n = 0. \]
  Using Neves~\cite[Proposition 5.1]{Neves1} we obtain an integral
  special Lagrangian $L_\infty$ such that \eqref{eq:limL1} holds for
  $t=a_0$. Unlike in \cite{Neves1} we only need one special
  Lagrangian $L_\infty$ because of our assumption \eqref{eq:theta1}.

  To see that \eqref{eq:limL1} holds for all $t\in [-1,0)$, we can
  argue exactly as Neves, using that
  \[ \frac{d}{dt} \int_{L^i_t} \phi\, d\mathcal{H}^n = \int_{L^i_t}
    \langle H, D\phi\rangle\, d\mathcal{H}^n - \int_{L^i_t} |H|^2
    \phi\, d\mathcal{H}^n. \]
  The bound \eqref{eq:Hint21} can then be used to bound the difference
  between the integrals of $\phi$ on $L^i_t$ and $L^i_{a_0}$. 
\end{proof}

We suppose that the flow $L_t$ has a tangent flow given by the
(special Lagrangian) static
flow $W_0$ at $(0,0)$. This means that for a suitable sequence of
scales $\sigma_i \to 0$, the sequence of rescaled flows defined by
\[ L^i_t = \sigma_i^{-1}L_{\sigma_i^2t} \]
converge to $W_0$ in the sense of the convergence \eqref{eq:limL1}
in Proposition~\ref{prop:N1}. The following is a
consequence of the monotonicity formula \eqref{eq:thetamon}. 

\begin{lemma} Suppose that $L_t$ has a tangent flow $W_0$ at $(0,0)$.
  Let $r_i \to 0$ and
$(x_i, t_i) \to (0,0)$ satisfying $r_i \leq \sqrt{-t_i}$. Then
the rescaled flows $\tilde{L}^i_t = r_i^{-1} (L_{t_i + r_i^2 + r_i^2t} - x_i)$ satisfy
\[ \lim_{i\to \infty} \int_{\tilde{L}^i_{-1}} |\theta - \theta_{W_0}|^2
  \rho_{0,0}\, d\mathcal{H}^n = 0. \]
\end{lemma}
\begin{proof}
The assumption that $L_t$ has a special Lagrangian tangent flow $W_0$ at
$(0,0)$ implies that given any $\epsilon > 0$ we can choose a small
$\sigma > 0$ such that the
rescaling $L^\sigma_t = \sigma^{-1} L_{\sigma^{2}t}$ satisfies
\[ \label{eq:thetaint20} \int_{L^\sigma_{-1}} |\theta - \theta_{W_0}|^2\,
  \rho_{0,0}\, d\mathcal{H}^n < \epsilon. \]
Using the monotonicity formula applied to $|\theta -\theta_{W_0}|^2$
centered at points $(x,t)$ with
$|x|\leq 1$ and $t\in [-1/2,0)$ this implies that we have the
pointwise bound
\[ \label{eq:thetap10} |\theta - \theta_{W_0}| \leq \Psi(\epsilon) \text{ on } (B_1(0)\cap
  [-1/2, 0) )\cap L^\sigma_t, \]
where $\Psi(\epsilon)$ denotes a function converging to 0 as
$\epsilon\to 0$. 

Note that 
\[ \tilde{L}^i_t = \tilde{r}_i^{-1}(L^\sigma_{\tilde{t}_i + \tilde{r}_i^2+
    \tilde{r}_i^2t} - \tilde{x}_i), \]
where
\[ \tilde{t}_i = \sigma^{-2} t_i, \quad \tilde{x}_i = \sigma^{-1}x_i,
  \quad \tilde{r}_i = \sigma^{-1} r_i. \]
In particular $(\tilde{x}_i, \tilde{t}_i) \to (0,0)$ and $\tilde{r}_i
\leq \sqrt{-\tilde{t}_i}$, and we have
\[ \int_{\tilde{L}^i_{-1}} |\theta - \theta_{W_0}|^2\, \rho_{0,0}\,
    d\mathcal{H}^n = \int_{L^\sigma_{\tilde{t}_i}}
   |\theta - \theta_{W_0}|^2 \rho_{\tilde{x}_i, \tilde{t}_i}\, d\mathcal{H}^n. \] 
Using the pointwise bound \eqref{eq:thetap10},  the fact that
$\tilde{t}_i \to 0$ and $\tilde{r}_i\to 0$, and the
global uniform bounds for  $\theta$ and the area ratios, we find that
for sufficiently large $i$ 
\[ \int_{\tilde{L}^i_{-1}} |\theta - \theta_{W_0}|^2\, \rho_{0,0}\,
  d\mathcal{H}^n < \Psi_1(\epsilon), \]
for another function $\Psi_1$ converging to zero as $\epsilon \to
0$. This completes the proof. 
\end{proof}

This result, together with Proposition~\ref{prop:N1}, implies that a
rescaled sequence $\tilde{L}^i_t$ as in the Lemma converges, up to
choosing a subsequence, to a
static flow given by an integral special Lagrangian current
$L_\infty$.

\begin{definition} \label{defn:goodblowup}
  Let us call a sequence of scales $r_i \to 0$ and points $(x_i,
  t_i)\to (0,0)$ satisfying $r_i < \sqrt{-t_i}$ a good blowup sequence
  for the flow $L_t$, if along a subsequence the corresponding rescaled Lagrangians
  $\tilde{L}^i_{-1} = r_i^{-1}(L_{t_i}-x_i)$ converge to a special Lagrangian current $L_\infty$
  that is asymptotic to $W_0$ at infinity, but such that
  $L_\infty$ is not a translation of $W_0$.
\end{definition}
Our goal in Theorem~\ref{thm:goodblowup} is to show that if a tangent
flow is given by $W_0$, then there is a good blowup sequence.

\subsection{Graphicality and linearization}\label{sec:gl}
Let us define the following notion of graphicality.
\begin{definition}\label{defn:cgraph}
  Let us fix a large integer $k_0 > 0$. 
  We say that $L$ is a $c$-graph over $W$ on the region $U\subset
  \mathbb{C}^n$, if $L\cap U$ can be written as the graph of a (not
  necessarily normal) vector field $v$ over $W\cap U$, with
  $|v|_{C^{k_0}} < c$.
\end{definition}

Note that if we have a sequence $L_i$ of $c_i$-graphs over $W$ on an open set $U\subset W$,
and $c_i\to 0$, then we can also view the $L_i$ as normal $c_i'$-graphs
over slightly perturbed open sets $U_i$, with $c_i' / c_i$ bounded. If
$W$ is Lagrangian, and $L$ is the normal graph of a vector field $v$
over $W$, then we also say that $L$ is the graph of the 1-form $\eta =
\iota_v\omega$. Equivalently this is the 1-form on $W$ dual to the
vector field $Jv$ along $W$. If both $W$ and $L$ are exact Lagrangian, then
to leading order $\eta$ is exact, and to leading order the difference
in Lagrangian angles $|\theta_L - \theta_W|$ is given by
$d^*\eta$. This is made more precise in the following.

\begin{lemma}\label{lem:graphlimit}
    Suppose that we choose $k_0$ in Definition~\ref{defn:cgraph} 
  large enough depending on $n$, and $c_0$ sufficiently small,
  depending on bounds for the geometry of $W\cap U$.
  Let $L_i$ be a sequence of exact Lagrangians that
  are $c_0$-graphs over the exact special Lagrangian $W$ on $U$ and satisfy
  \[  \int_{L_i\cap U} d_{W}^2\,
        d\mathcal{H}^n \leq d_i^2 \to 0, \]
      for the distance function $d_{W}$ from $W$. Suppose in addition
      that we have the bound
  \[  \sup_{L_i\cap U} d_i^{-1} |\theta_{L_i}| \to 0. \]
  Writing $L_i$ as the graphs of the 1-forms $\eta_i$ over open sets
  $U_i$ as above, we then have that along a subsequence the 1-forms
  $d_i^{-1}\eta_i$ converge locally uniformly to an exact harmonic 1-form
  $\eta_\infty$ on $W\cap U$. 
\end{lemma}
\begin{proof}
  We can use the description of $L_i$ as the graph of $\eta_i$ over
  $W$ to pull back geometric quantities on $L_i$ to $W$. In this way
  we have
  \[ \begin{aligned}
      \omega|_{L_i} &= \omega|_W + d\eta_i + Q_1(x,\eta_i, \nabla
      \eta_i), \\
        \theta_{L_i} &=  d^*\eta_i + Q_2(x,\eta_i, \nabla\eta_i),
      \end{aligned}\]
   where $Q_1$ and $Q_2$ are power series depending at least
   quadratically on the entries $\eta_i, \nabla\eta_i$. If
   \[ |\eta_i| + |\nabla \eta_i| \leq \epsilon, \]
   for sufficiently small $\epsilon$, then we can apply elliptic
   estimates to a small perturbation of the operator $d+d^*$ to deduce
   \[ \Vert \eta_i\Vert_{W^{1,2}(U')} \leq C_{U'} d_i ,\]
   for any relatively compact $U'\subset U$. It follows that up to
   choosing a subsequence the $d_i^{-1}\eta_i$ converge locally in
   $L^2$ to a limit $\eta_\infty$ on $W\cap U$. The assumption that
   $d_i^{-1}|\theta_{L_i}| \to 0$ uniformly implies that $\eta_\infty$
   satisfies the system $d\eta_\infty = 0$ and $d^*\eta_\infty = 0$
   weakly, and so $\eta_\infty$ is a harmonic one-form.

   To see that $\eta_\infty$ is exact, note that any loop $\gamma$
   on $W\cap U$ can be viewed as a loop $\gamma_i$ on $L_i$ by the
   graphicality condition. If we denote by $\lambda$ the Liouville one-form,
   then we have
   \[ \int_{\gamma_i} \lambda = 2\int_\gamma \eta_i + O(d_i^2). \]
   It follows that
   \[ \lim_{i\to\infty} \int_\gamma d_i^{-1} \eta_i = 0, \]
   and so the limiting form $\eta_\infty$ is exact. 
  \end{proof}

\section{Estimates for a given time slice in a ball}\label{section:timeslice}
In this section we study a fixed time slice $L$ of the (rescaled) mean
curvature flow, inside the ball $B(0,2)$. Our goal is to show that in
several different settings we
can obtain bounds for how much a time slice $L$ can deviate from
$W_0$ or from a perturbation of $W_0$.
This exploits the fact that when the Lagrangian angle $\theta$ is small
pointwise, then the behavior of $L$ is close to that of a minimal
submanifold. The basic idea is to use the three annulus lemma to show
that either the time slice $L$ remains close to a cone $W$ at many
scales, or at some scale $L$ is close to a special Lagrangian
asymptotic to $W_0$. For instance if no good blowup sequence exists, in the sense of
Definition~\ref{defn:goodblowup}, then we find that $L$ must remain
close to $W$. At the same time if the link of $W_0$ is connected, or
in the two dimensional case, we can obtain a
slightly worse estimate, even without assuming that good blowup
sequences do not exist.

In our applications we will fix a Lagrangian mean curvature
flow in $\mathbb{C}^n$ that we denote by $L^0_t$ for $t\in [-1,0)$, and which has a
singularity at $(0,0)$. This flow is assumed to satisfy uniform bounds
for the area ratios and the Lagrangian angle. Throughout this section
we consider Lagrangian mean curvature flows $L_t$ for $t\in [-1,0)$
which are obtained as suitable rescalings of $L^0_t$. More precisely
we assume the following:

\begin{itemize}
\item[(\dag)] For some $(X,T)$ and $r > 0$ such that $T\in [-1,0)$ and
  $|X|, r < \sqrt{-T}$, the flow $L_t$ is defined by
  \[ r^{-1}\Big( L^0_{T+ r^2(t+1)} - X\Big). \]
  for $t\in [-1,0)$. 
\end{itemize}

We will need to allow perturbations of the cone $W_0$ to certain unions of special Lagrangian cones that do not all have the same Lagrangian angle. For this we have the following definition.
\begin{definition}\label{defn:components1}
   Suppose that $L$ is an immersed Lagrangian in $B(0,2)$ that is a $c_0$-graph over $W_0$ on the annulus $A_{2,1/2}$ for sufficiently small $c_0 > 0$ depending on $W_0$. We can then decompose $L\cap B(0,2)$ as a union of connected components 
   \[ L\cap B(0,2) = \bigcup_{j=1}^m L_j, \]
   where each $L_j$ is an immersed Lagrangian whose domain of parametrization is connected. 

   Corresponding to the decomposition of $L$ we also have a decomposition of the cone $W_0$:
   \[ W_0 = \bigcup_{j=1}^m W_{0,j},\]
   such that on $A_{2,1/2}$ the component $L_j$ is $\delta$-graphical over $W_{0,j}$. 

    We will denote by $\mathcal{W}_L$ a small neighborhood of $W$ in the moduli
space of minimal Lagrangian cones obtained by perturbing each $W_{0,j}$ to other special Lagrangian cones. Thus, each element $W\in
\mathcal{W}$ is of the form
\[ W = \bigcup_{j=1}^m W_j, \]
where the cone $W_j$ is a small special Lagrangian perturbation of $W_{0,j}$. 

    Finally, for $W\in \mathcal{W}_L$ we define the angle function $\theta_W: L\cap B(0,2)\to \mathbb{R}$ to be the (locally constant) function equal to $\theta_{W_j}$ on the component $L_j$.
\end{definition}

The main result is the following.
\begin{prop}\label{prop:mainelliptic}
  Suppose that $L_t$ is a Lagrangian mean curvature flow as in (\dag)
  for $t\in [-1,0)$. 
  There are constants $\delta, c_1,  \rho_0, C > 0$ depending on the bounds
  on the angle and area ratios, as well as on the cone $W_0$,
  satisfying the following. Suppose that $L_{-1}$ is sufficiently
  close to $W_0$ in the sense that $L_{-1}$ is a $\delta^2$-graph over
  $W_0$ on the annulus $A_{\delta^{-1}, \delta}$.
  In addition suppose
  that for some $d < \delta$ the time slice $L := L_{-1/2}$ satisfies
  \begin{itemize}
  \item[(i)] $|\theta - \theta_{W}| < d$ on $B_2\cap L$ for some
    $W\in \mathcal{W}_L$, 
    \item[(ii)] $L$ is a $d$-graph over $W$ on $A_{2, \rho_0/2}$.
    \end{itemize}
    Then
    \begin{itemize}
    \item[(a)]
        If the link of $W_0$ is connected, or $n=2$, then for all $r\in
    (\delta^{-1}d^{(2+\lambda_2)^{-1}}, 1)$,  $r^{-1}L$ is a
    $Cdr^{-2-\lambda_2}$-graph 
    over $W$ on $A_{1,1/2}$. Here $\lambda_2 > 0$ appears in
    Lemma~\ref{lem:L23ann1}, and can be chosen to be 
    arbitrarily close to 0. 
  \item[(b)]
    If the flow $L^0_t$ admits no good blowup sequence in the
  sense of Definition~\ref{defn:goodblowup}, then for all $r \in
  (\delta^{-1}d, 1)$, $r^{-1}L$ is a $Cdr^{-1}$-graph over
  $W$ on the annulus $A_{1, 1/2}$.
\item[(c)]
  At least one of the following two alternatives holds:
  \begin{itemize}
  \item Either the same conclusion as in (b) holds (without any
    assumption on good blowup sequences),
    \item or there is an $r_0 > 0$ (depending on $L$) such that for all $r\in (\delta^{-1}r_0,
      1)$ the rescaling $r^{-1}L$ is a $\max\{c_0,
      Cdr^{-(n+\lambda_2)}\}$-graph over $W$ on $A_{1, 1/2}$ and in
      addition
      \[ \label{eq:xperplower} \int_{r_0^{-1}L\cap (B_2\setminus B_1)} |x^\perp|^2\,
        d\mathcal{H}^n > c_1. \]
      Here $\lambda_2, c_0 > 0$ can be arranged to be arbitrarily small, by choosing
      the other constants suitably. 
    \end{itemize}
    \end{itemize}
  \end{prop}

The proof will take up the rest of this section. We first give a rough
outline of the argument for part (a). Part (b) and (c) follow a similar
approach, using also Proposition~\ref{prop:slowdecay}. 
Initially let us assume that actually $\theta = \theta_W$ on
$L$, so that $L$  is a minimal submanifold. Then, similarly to
Edelen~\cite[Theorem 13.1]{Edelen}, we can follow
the proof of the uniqueness of tangent cones to show that
if $L$ is close to $W$ at the unit scale (i.e. on the annulus $A_{1,
  \rho_0}$), then $L$ remains close to $W$ on smaller annuli $A_{r,
  r\rho_0}$, as long as it remains ``sufficiently conical''. Whether
$L$ remains sufficiently conical at scale $r$ can in turn be
controlled by requiring that the
Gaussian area $\Theta(L, 0, r)$ remains close to $\Theta(W)$. Applying
similar considerations to translates of $L$ as well leads to the
statement that either a translate $L-x_0$ of $L$ is close to $W$ at all scales, or
we can find a rescaling which leads to a good blowup sequence. This is
similar to the argument in \cite[Theorem 2.8]{Edelen}.

In our setting instead of $\theta=\theta_W$ we only know
that $|\theta - \theta_W| < d$. The arguments used in the proof of the
uniqueness of tangent cones -- in our case the 3-annulus lemma -- will
then only work at scales at which the distance from $W$ to $L$ is much
greater than $d$. As we will see more precisely below, the effect of
this will be that either we can find a rescaling of $L$ which leads to
a good blowup sequence, or $L$ remains close to some cone $W_r$ on the
annulus $A_{r, r\rho_0}$ for $r < 1$, however the cones $W_r$ can slowly drift
away from $W$: we have an estimate of the form $d(W, W_r) < Cd|\log
r|$ for the distance between the cones $W,W_r$ inside
$\mathcal{W}_L$. 
We are only interested in $r > \delta^{-1}d$ in our application, so
$W_r$ still remains very close to $W$ and is in particular a
$Cd$-graph over $W$ on the annulus $A_{r, r\rho_0}$. 

Before we give the proof of Proposition~\ref{prop:mainelliptic}, we
need some preliminary results. 
Let us begin by stating the linear 3-annulus lemma that we will
use, which is the elliptic version of Lemma~\ref{lem:L23ann2}. For a
1-form $\eta$ on the cone $W_0$, we define the $L^2$ norm 
\[ \Vert \eta\Vert^2_{U} = \int_U |\eta|^2\,
  d\mathcal{H}^n, \]
where $U\subset W_0$ is any subset. 

\begin{lemma}[See Lemma 3.3 in Simon~\cite{SimonIsolated}] \label{lem:L23ann1}
  Let $d\in \mathbb{R}$. For all sufficiently small $0 < \lambda_1 < \lambda_2$
  depending on $W_0, d$, we can choose 
  sufficiently small $\rho_0 > 0$ (depending on
  $W_0, d, \lambda_i$) satisfying the following. Suppose that $\eta$ is a
  harmonic 1-form on the annulus $A_{1, \rho_0^3}$ with finite
  $L^2$-norm. Then
  \begin{itemize}
    \item[(i)] If $\Vert \eta\Vert_{A_{\rho_0, \rho_0^2}} \geq
      \rho_0^{d-\lambda_1+\frac{n}{2}} \Vert \eta\Vert_{A_{1, \rho_0}}$ then $\Vert
      \eta\Vert_{A_{\rho_0^2, \rho_0^3}} \geq \rho_0^{d-\lambda_2 + \frac{n}{2}} \Vert
      \eta\Vert_{A_{\rho_0, \rho_0^2}}$. 
    \item[(ii)] If $\Vert \eta\Vert_{A_{\rho_0, \rho_0^2}} \geq
      \rho_0^{-d-\lambda_1- \frac{n}{2}} \Vert \eta\Vert_{A_{\rho_0^2, \rho_0^3}}$ then $\Vert
      \eta\Vert_{A_{1, \rho_0}} \geq \rho_0^{-d-\lambda_2 - \frac{n}{2}} \Vert
      \eta\Vert_{A_{\rho_0, \rho_0^2}}$. 
    \end{itemize}
    In addition if under the decomposition of $\eta$ into homogeneous
    components there is no homogeneous degree $d$ term, then one of
    the conclusions in (i) or (ii) must hold. 
  \end{lemma}

Here we define the degree of a homogeneous one-form $\eta$ in terms of
the degree of its norm. That is, a homogeneous one-form $\eta$ of
degree $d$ satisfies $|\eta|(rz) = r^d |\eta|(z)$ where $r >
0$ and $z\in W_0$. Thus the graphs of homogeneous $\eta$ with degree
1 over $W_0$ are cones. 

Using Lemma~\ref{lem:graphlimit} we can obtain an analogous result for
Lagrangians $L$ close to $W_0$, as in \cite[Lemma 3.3]{SimonIsolated},
as long as the Lagrangian angle of $L$ is sufficiently close to that
of $W_0$. Let us first define the following distance function for a
Lagrangian $L$ from a cone $W\in \mathcal{W}_L$ on an open set $U$:
\[ D_W(L, U) = \left(\int_{L\cap U} d_W^2\, d\mathcal{H}^n\right)^{1/2}, \]
where $d_W(x)$ is the distance of $x$ from $W$. 

\begin{lemma}\label{lem:3ann2}
  We choose $0 < \lambda_1 < \lambda_2$ small, so that
  Lemma~\ref{lem:L23ann1} applies with $d=1,-1$ and $1-n$. We choose $\rho_0 > 0$ small enough for Lemma~\ref{lem:L23ann1}
  to apply with these values of $d$ and $\lambda_i$. Let $\lambda\in
  (\lambda_1, \lambda_2)$. 
 If $c_0, \kappa_0 > 0$ are chosen sufficiently small,
  depending on $W_0, \rho_0, \lambda$, then there is a constant $C > 0$
  depending on $W_0$ such that we have the following. 

  Let $L$ be a Lagrangian $c_0$-graph over $W_0$ on the annulus $A_{1,
    \rho_0^{3}}$.
  Suppose that $d(W, W_0) < \kappa_0$ in $\mathcal{W}_L$, and $D_{W}(L,
  A_{1, \rho_0}) < \kappa_0$. In addition suppose that
  \[ |\theta_L - \theta_W| \leq \kappa_0 D_W(L, A_{1, \rho_0})
    \text{ on the ball $B_1(0)$.} \]
  We then have the following.
  \begin{itemize}
    \item[(a)] There exists a $W' \in \mathcal{W}_L$ with $d(W, W') \leq
      CD_W(L , A_{1, \rho_0})$, and $\theta_{W'} = \theta_W$, such
      that at least one of the following possibilities holds:
      \begin{itemize}
      \item Growth setting:
        \[D_{W'}(\rho_0^{-2}L, A_{1, \rho_0}) \geq \rho_0^{-\lambda}
          D_{W'}(\rho_0^{-1}L, A_{1, \rho_0}), \]
        \item Decay setting (note on the right the cone is $W$ not $W'$):
           \[D_{W'}(\rho_0^{-1}L, A_{1, \rho_0}) \leq \rho_0^{\lambda}
          D_{W}(L, A_{1, \rho_0}). \]
          \end{itemize}
    \item[(b)] With $p=0,2$ or $n$, if $D_W(\rho_0^{-1}L, A_{1, \rho_0}) \geq \rho_0^{-(p+\lambda)}
      D_W(L, A_{1, \rho_0})$, then
      \[D_W(\rho_0^{-2}L, A_{1, \rho_0}) \geq \rho_0^{-(p+\lambda)}
        D_W(\rho_0^{-1}L, A_{1, \rho_0}).\]
   \end{itemize}
\end{lemma}
\begin{proof}
  We first prove (a),  arguing by contradiction.  Suppose that $\kappa_{0,i} \to 0$,
  and we have Lagrangians $L^i$ and $W^i\in \mathcal{W}_{L^i}$ satisfying the hypotheses
  \begin{enumerate}
    \item $L^i$ is a $c_0$-graph over $W_0$ on the annulus
      $A_{1,\rho_0^{3}}$,
    \item $d(W^i, W_0) < \kappa_{0,i}$, $D_{W^i}(L^i, A_{1,\rho_0}) <
      \kappa_{0,i}$,
    \item $|\theta_{L^i} - \theta_{W^i}| \leq \kappa_{0,i} D_{W_i}(L^i,
      A_{1,\rho_0})$ on $A_{1,\rho_0^{3}}$.
    \end{enumerate}
   We will show that for suitable $c_0,  C$ the conclusions will
   hold once $i$ is sufficiently large.
   
   If $c_0$ is chosen sufficiently small, then we can write the $L^i$
   as the graphs of one-forms $\eta_i$ over $W^i$ on the annulus
   $A_{1,\rho_0^{3}}$. Let us write
   \[ d_i = D_{W^i}(\rho_0^{-1} L^i, A_{1,\rho_0}). \]

   We can assume that for all $i$ the conclusions in part (a) fail
   with the choice $W' = W^i$
   and in particular this means that
   \[ \begin{aligned}
       D_{W^i}(\rho_0^{-2}L^i, A_{1, \rho_0}) &< \rho_0^{-\lambda}
       d_i, \\
       D_{W^i}(L^i, A_{1,\rho_0}) &< \rho_0^{-\lambda} d_i.
     \end{aligned} \]

   This implies that we have an $L^2$ bound $\Vert
   \eta_i\Vert_{A_{1,\rho_0^{3}}} \leq C_1 d_i$. Using
   Lemma~\ref{lem:graphlimit} we find that along a subsequence we have
   $d_i^{-1}\eta_i \to \eta_\infty$ locally uniformly on
   $A_{1,\rho_0^3}$, and $\eta_\infty$ is an exact harmonic one-form on
   $W_0$. We can write $\eta_\infty = \eta_\infty^0 + \tilde\eta_\infty$, where
   $\eta_\infty^0$ is a homogeneous degree one harmonic one-form and $\tilde\eta_\infty$ is
   orthogonal to the space of such homogeneous forms. By the
   integrability assumption for $W_0$ we can find
   $\tilde{W}^i\in\mathcal{W}_{L^i}$ that are given by the graphs of $d_i
   \eta_\infty^0$ over $W^i$ to leading order. Writing $L^i$ as the
   graphs of $\tilde{\eta}_i$ over $\tilde{W}^i$ we find that
   \[ d_i^{-1}\tilde{\eta}_i \to \tilde{\eta}_\infty, \]
   locally uniformly. 
   Since $\tilde{\eta}_\infty$ has no degree one component,
   Lemma~\ref{lem:L23ann1}, applied with $d=1$, implies that we must have either
   \[\label{eq:etainf1} \Vert \tilde{\eta}_\infty\Vert_{A_{\rho_0^2, \rho_0^3}} \geq
     \rho_0^{1-\lambda_2+\frac{n}{2}} \Vert
     \tilde{\eta}_\infty\Vert_{A_{\rho_0, \rho_0^2}}, \]
   or
   \[ \label{eq:etainf2} \Vert \tilde{\eta}_\infty\Vert_{A_{1, \rho_0}} \geq
     \rho_0^{-1-\lambda_2-\frac{n}{2}} \Vert
     \tilde{\eta}_\infty\Vert_{A_{\rho_0, \rho_0^2}}. \]

   The local uniform convergence $d_i^{-1}\tilde\eta_i \to
   \tilde\eta_\infty$, and scaling, implies that we have
   \[ \lim_{i\to\infty} d_i^{-1}D_{\tilde{W}_i}(\rho_0^{-1} L^i, A_{1,
       \rho_0}) = \rho_0^{-1-\frac{n}{2}}
     \Vert \tilde\eta_\infty\Vert_{A_{\rho_0, \rho_0^2}}, \]
   and
   \[ \begin{aligned}
       \lim_{i\to\infty} d_i^{-1} D_{\tilde{W}_i}(L^i, A_{1,\rho_0}) &\geq 
       \Vert \tilde\eta_\infty\Vert_{A_{1, \rho_0}}, \\
       \lim_{i\to\infty} d_i^{-1} D_{\tilde{W}_i}(\rho_0^{-2} L^i, A_{1,\rho_0}) &\geq \rho_0^{-2-n}
       \Vert \tilde\eta_\infty\Vert_{A_{\rho_0^2, \rho_0^3}}.
     \end{aligned} \]
   It follows, using \eqref{eq:etainf1} or \eqref{eq:etainf2}, that we
   have either
   \[ \label{eq:etainf3} \lim_{i\to\infty} d_i^{-1} D_{\tilde{W}_i}(L^i, A_{1,\rho_0}) &\geq \rho_0^{-\lambda_2}
     \lim_{i\to\infty} d_i^{-1} D_{\tilde{W}_i}(\rho_0^{-1} L^i,
     A_{1,\rho_0}), \]
   or
   \[ \label{eq:etainf4}
     \lim_{i\to\infty} d_i^{-1} D_{\tilde{W}_i}(\rho_0^{-2} L^i, A_{1,\rho_0}) &\geq \rho_0^{-\lambda_2}
     \lim_{i\to\infty} d_i^{-1}D_{\tilde{W}_i}(\rho_0^{-1} L^i,
     A_{1,\rho_0}). \]
   If $d_i^{-1} D_{\tilde{W}_i}(\rho_0^{-1}L^i, A_{1,\rho_0}) \to 0$,
   then for large $i$ we are in the ``Decay setting'' of (a). Otherwise the
   inequalities \eqref{eq:etainf3}, \eqref{eq:etainf4} imply that if
   $\lambda < \lambda_2$, then for sufficiently large $i$, if $\rho_0$
   is chosen sufficiently small (depending on $C, \lambda$), we are in
   either the growth or decay setting of (a).

   The proof of (b) is very similar, arguing by contradiction. As
   in part (a) we assume that we have a sequence $L^i$ for which the
   conclusion fails. Writing the $L^i$ as graphs of $\eta_i$ over
   $W_i$, and normalizing by $d_i = D_{W_i}(\rho_0^{-1}L^i,
   A_{1,\rho_0})$, we can extract a limiting harmonic one-form
   $d_i^{-1}\eta_i \to \eta_\infty$ on $W_0$, which satisfies
   \[ \rho_0^{-1-\frac{n}{2}} \Vert \eta_\infty\Vert_{A_{\rho_0, \rho_0^2}} = 1, \]
   and at the same time 
   \[ \begin{aligned}
       \Vert \eta_\infty \Vert_{A_{1,\rho_0}} &\leq \rho_0^{p-\lambda-1-\frac{n}{2}} \Vert \eta_\infty\Vert_{A_{\rho_0, \rho_0^2}}, \\
       \Vert \eta_\infty \Vert_{A_{\rho_0^2, \rho_0^3}} &\leq
       \rho_0^{1+\frac{n}{2}-p-\lambda} \Vert \eta_\infty\Vert_{A_{\rho_0, \rho_0^2}}.
     \end{aligned} \]
   These contradict Lemma~\ref{lem:L23ann1} part (i), with $d=1-p$. 
\end{proof}

In order to apply this result to rescalings $\rho_0^{-k}L$, we need to
ensure that these rescalings 
remain close to $W_0$ on the annulus $A_{1,\rho_0^3}$. In the proof of
Proposition~\ref{prop:mainelliptic} we will be
applying this result to Lagrangians given by
$L = L_{-1/2}$ for a flow $L_t$ obtained from $L^0_t$ as in (\dag). The following
results shows that if $L_{-1}$
is sufficiently close to $W_0$, then $\rho_0^{-k}L_{-1/2}$ will be close to a
cone as long as its Gaussian area is sufficiently close to that of
$W_0$. This will be used below to ensure that we can keep applying
Lemma~\ref{lem:3ann2} at smaller and smaller scales. Note that to
apply the Lemma, we also need the angle $|\theta_L - \theta_W|$ to be
very small relative to the distance from $L$ to $W$. At the same time,
our goal is to control the distance of $L$ to a cone, so 
at scales where this distance is relatively small compared to the
angle bound, our task is already done. See the proof of
Proposition~\ref{prop:mainelliptic} below for the detailed argument. 

\begin{lemma}\label{lem:conescale}
  Let $c_0 > 0$ be small. Given $\epsilon > 0$,
  there are $\delta, \kappa_1 > 0$ depending on
  $c_0, \epsilon, W_0$, such that we have the following: suppose that
  $L_{-1}$ is a $\delta^2$-graph over $W_0$ on $A_{\delta^{-1}, \delta}$, $|x|
  < \delta$, and let $L=L_{-1/2}-x$. Suppose that for some $r < 1$ the rescaled surface $\tilde{L} =
  r^{-1}L$ satisfies $\Theta(\tilde{L}) >
  \Theta(W_0) - \kappa_1$. In addition suppose that $\tilde{L}$ is a
  $c_0$-graph over $W_0$ on the annulus $A_{2,1}$. Then $\tilde{L}$ is
  an $\epsilon^2$-graph over $W\in\mathcal{W}_L$ over the annulus
  $A_{\epsilon^{-1}, \epsilon}$. 
\end{lemma}
\begin{proof}
  We can argue by contradiction. Suppose that for a given small $c_0,
  \epsilon > 0$ no $\delta > 0$ works, and so we have a sequence
  $L^i_t$ satisfying that $L^i_{-1} \to W_0$, while
  \[ \int_{L^i_{-1/2}} \rho_{x_i, -1/2+r_i^2}\, d\mathcal{H}^n> \Theta(W_0) - 1/i,  \]
  for some $x_i \to 0$ and $r_i\in (0,2)$. 
  The condition $L^i_{-1}\to W_0$ implies that
  \[ \int_{L^i_{-1}} \rho_{x_i, -1/2+r_i^2} d\mathcal{H}^n&\to \Theta(W_0) \\
      \int_{L^i_{-1}} |\theta - \theta_{W_0}|^2\, \rho_{x_i, -1/2 +
        r_i^2} d\mathcal{H}^n&\to 0. \]
   The argument in Neves~\cite[Theorem A]{Neves1} together with
   Proposition~\ref{prop:N1} implies that along a subsequence the Lagrangians
   $r_i^{-1}(L^i_{-1/2} - x_i)$ converge to a special Lagrangian
   cone $C$. By assumption the $r_i^{-1}(L^i_{-1/2}-x_i)$ are $c_0$-graphs over
   $W_0$ on the annulus $A_{2,1}$, and so if $c_0$ is sufficiently
   small, then we must have $C = W\in \mathcal{W}_L$. In particular for
   large $i$, $r_i^{-1}(L^i_{-1/2} - x_i)$ is an $\epsilon^2$-graph
   over $W$ on $A_{\epsilon^{-1}, \epsilon}$ as required. 
 \end{proof}

\begin{lemma}\label{lem:growthbound2}
  Let $L_t$ be a flow obtained from $L^0_t$ by rescaling as in (\dag). Suppose that
  $L_{-1}$ is a $\delta^2$-graph over $W_0$ on $A_{\delta^{-1},
    \delta}$ for a sufficiently small $\delta > 0$. 
  Suppose that $\kappa_0$ in Lemma~\ref{lem:3ann2} is chosen
  sufficiently small, depending on $L^0_t$. 
  Given an $\epsilon > 0$ sufficiently small depending on $\kappa_0$,
  we choose $\kappa_1$,
  so that  Lemma~\ref{lem:conescale} applies.   For $|x| < 1$ define
\[ \label{eq:r01} r_0(x) = \inf\{ R\in (0,1) \,:\, \Theta( r^{-1} (L_{-1/2}-x)) >
  \Theta(W_0) - \kappa_1, \text{ for all }r\in (R, 1)\}. \]
There is a large
constant $k_1 > 0$, depending on $W_0, \kappa_1$, such that the
following holds. Suppose that for some $r \in (\rho_0^{-k_1} r_0(0),
1)$ the Lagrangian $L = r^{-1} L_{-1/2}$ satisfies the hypotheses of
Lemma~\ref{lem:3ann2}. Then 
  \begin{itemize}
    \item[(a)] The growth rate $p=n$ cannot appear in part (b) of
      Lemma~\ref{lem:3ann2} applied to $L$. 
    \item[(b)] If the link of $W_0$ is connected, then $p=2$ can also
      not appear.
    \item[(c)] Suppose that there is no good blowup sequence along the
      flow $L^0_t$ and that $r_0(0)\leq
      r_0(x)$ for all $|x| < 1$. Then the growth rate $p=0$ can also not
      appear.
    \item[(d)] Suppose that the growth rate $p=0$ occurs. Then $r_0(0) > 0$ and the 
      rescaling $r_0(0)^{-1}L_{-1/2}$ satisfies
      \[ \int_{(B_2\setminus B_1)\cap r_0(0)^{-1}L_{-1/2}} |x^\perp|^2\,
        d\mathcal{H}^n > c_1, \]
      where $c_1$ depends on $W_0$ and the choice of $\kappa_1$. In
      addition for all $r \in (r_0(0), 1)$ the scaling
      $r^{-1}L_{-1/2}$ is a $c_0$-graph over $W_0$ on
      $A_{1,1/2}$. 
   \end{itemize}
\end{lemma}
\begin{proof}
  Let us first consider the statement (a), with growth rate $p=n$. We argue
  by contradiction, so let us suppose
  that for some $r > \rho_0^{-k_1} r_0(0)$, 
  the surface $L=r^{-1}L_{-1/2}$ satisfies
\[ \label{eq:dw30} D_W(\rho_0^{-2}L, A_{1,\rho_0}) \geq
\rho_0^{-(n+\lambda)}D_W(\rho_0^{-1}L, A_{1,\rho_0}). \]
We will reach a contradiction if the constants $\kappa_0, \epsilon, \delta$ are sufficiently small, and $k_1$ is large.

We suppose that the hypotheses of Lemma~\ref{lem:3ann2} hold,
so $L$ is a $c_0$-graph over $W_0$ on $A_{1,\rho_0^3}$, $d(W,W_0) <
\kappa_0$, $D_W(L, A_{1,\rho_0}) < \kappa_0$, and $|\theta_L-\theta_W|
\leq \kappa_0 D_W(L, A_{1,\rho_0})$ on $B_1(0)$. We claim that we can
then keep iterating the conclusion of (b) in Lemma~\ref{lem:3ann2} applied to
$\rho_0^{-k}L$, as long as $\rho_0^k r > r_0(0)$. To see this we just
need to ensure that the hypotheses of the Lemma continue to hold (with
the same choice of $W$). We argue by induction, supposing that  the
hypotheses of the Lemma, together with the growth assumption (b), with
$p=n$ hold for $\rho_0^{-k}L$, and $\rho_0^k r > r_0(0)$.

First, the condition $|\theta_L - \theta_W| \leq \kappa_0
D_W(\rho_0^{-(k+1)}L, A_{1,\rho_0})$ continues to hold by the growth
  assumption. Since $\rho_0^{-k}L$ is a $c_0$-graph over $W_0$ on
  $A_{1,\rho_0^3}$, by Lemma~\ref{lem:conescale} we know that
  $\rho_0^{-k}L$ is an $\epsilon^2$-graph over some $W'\in
  \mathcal{W}$ on the annulus $A_{\epsilon^{-1}, \epsilon}$. But then
  $D_W(\rho_0^{-k}L, A_{1,\rho_0}) < \kappa_0$ implies that $d(W, W')
  < C(\kappa_0 + \epsilon^2)$ for some $C > 0$, and so also $d(W_0,
  W') < C(\kappa_0 + \epsilon^2)$ for a larger $C$. If $\kappa_0,
  \epsilon$ are sufficiently small (depending on $W_0, c_0$), then we
  can conclude that $\rho_0^{-k-1}L$ is also a $c_0$-graph over
  $W_0$ on the annulus $A_{1,\rho_0^3}$.

  It remains to argue that we must still have $D_W(\rho_0^{-k-1}L,
  A_{1,\rho_0}) < \kappa_0$. If this were not the case, then we would
  have
  \[ \begin{aligned}
      D_W(\rho_0^{-k-1}L , A_{1,\rho_0}) &\geq \kappa_0, \\
      D_W(\rho_0^{-k-1}L , A_{1,\rho_0}) &\geq \rho_0^{-(n+\lambda)}
        D_W(\rho_0^{-k}L, A_{1,\rho_0}).
      \end{aligned} \]
  If $\epsilon$ is chosen sufficiently small (i.e. $\kappa_1$ is also
  chosen small), depending on $\kappa_0$,
  then this contradicts that $\rho_0^{-k}L$ is an $\epsilon^2$-graph
  over a cone $W'$ on the annulus $A_{\epsilon^{-1}, \epsilon}$.

  It follows then that the growth condition in part (b) of
  Lemma~\ref{lem:3ann2} applies to $\rho_0^{-k} L$ as long as
  $r\rho_0^k > r_0(0)$. Write $k_2$ for the largest such $k$, so
  that in particular $k_2 \geq k_1$. Note that $k_2$ is necessarily finite, since we
  cannot have the growth condition holding for all $k$ while also $D_W(\rho_0^{-k-1}L, A_{1,\rho_0}) < \kappa_0$. Equivalently, we have $r_0(0) > 0$. 
  Consider the rescaled surface $\tilde{L} =
  (r \rho_0^{k_2})^{-1}L_{-1/2} = \rho_0^{-k_2}L$. By the definition
  of $r_0(0)$ we then have
  \[ \label{eq:Thetasmall10} \Theta(\lambda \tilde{L}) \leq \Theta(W_0) - \kappa_1, \]
  for some $\lambda\in (1, \rho_0^{-1})$. The growth condition at each scale implies that we have
  \[ D_W(\rho_0^{j}\tilde{L}, A_{1,\rho_0}) \leq \rho_0^{n+\lambda} D_W(\rho_0^{j-1}\tilde{L}, A_{1\rho_0}),  \]
  for $j=1,\ldots,k_2$. If we have a sequence of counterexamples with $\delta, \kappa_0\to 0$ and $k_2\geq k_1\to\infty$, then the corresponding rescaled surfaces $\tilde{L}$ will converge to a special Lagrangian current $L_\infty$ that is asymptotic to $W_0$ at rate $2-n-\lambda$. Note that $L_\infty \not= W_0$ because of \eqref{eq:Thetasmall10}. This contradicts Proposition~\ref{prop:slowdecay}.
  Therefore, if $\delta$ is chosen sufficiently small, then our
  original assumption \eqref{eq:dw30} leads to a contradiction. This
  shows the required result for $p=n$.

  To prove (b), for $p=2$ in the case when $W_0$ is connected, the
  argument is entirely analogous, using part (b) of Proposition~\ref{prop:slowdecay}.

  To see
  the claim (c) in the case when there is no good blowup along the
  flow $L^0_t$, we also argue in a similar way. We suppose that with a large
  $k_1$, small $\delta, \kappa_0$, and
for some $r > \rho_0^{-k_1} r_0(0)$, 
  the surface $L=r^{-1}L_{-1/2}$ satisfies
\[ \label{eq:dw30.1} D_W(\rho_0^{-2}L, A_{1,\rho_0}) \geq
  \rho_0^{-\lambda}D_W(\rho_0^{-1}L, A_{1,\rho_0}). \]
  Just as before, we can iterate the growth property in
  Lemma~\ref{lem:3ann2} part (b), applied to $\rho_0^{-k}L$, and so we
  have
  \[ \label{eq:dw31}
    D_W(\rho_0^{-k-1}L, A_{1,\rho_0}) \geq \rho_0^{-\lambda}
    D_W(\rho_0^{-k}L, A_{1,\rho_0}), 
  \]
    as long
  as $r\rho_0^k > r_0(0)$, and as a consequence we have $r_0(0) > 0$. 
  We now let $k_2\geq k_1$ be the largest value of
  $k$ so that we still have $r\rho_0^k > r_0(0)$. The rescaled surface $\tilde{L} =
  \rho_0^{-k_2}L$ now satisfies
  \[ \Theta(\tilde L) &> \Theta(W_0) - \kappa_1, \]
  but at the same time, since $r\rho_0^{k_2+1} \leq r_0(0) \leq
  r_0(x)$ for all $|x| < 1$, we have
  \[  \label{eq:dw32} \Theta(a\tilde{L} - a(r\rho_0^{k_2})^{-1}x) \leq
    \Theta(W_0) - \kappa_1, \]
  for some $a\in(1,\rho_0^{-1})$. 
  We claim that this is not possible for sufficiently small $\delta, \kappa_0$
  and large $k_2$ if there is no good blowup along $L^0_t$. Indeed,
  if we had such rescalings $\tilde{L}$ for arbitrarily small $\delta$
  and large $k_2$, then in the limit we would obtain a blowup limit
  that would be a special Lagrangian current. This current would be
  asymptotic to $W_0$ at rate $\lambda+2$ because of \eqref{eq:dw31} (and because $d(W, W_0)< \kappa_0$), but
  it would not be equal to a translate of $W_0$ because of
  \eqref{eq:dw32}.

  Finally to see the claim (d), we argue similarly to (b). As in case (b),
  if the growth case with $p=0$ occurs, then $r_0(0) > 0$ and the
  rescaled surface $\tilde{L} =\rho_0^{-k_2}L$ satisfies
  \[ \label{eq:Th20} \Theta(\tilde{L}) > \Theta(W_0) - \kappa_1 \text{ and }
    \Theta(a\tilde{L}) \leq \Theta(W_0) - \kappa_1, \]
  for some $a\in (1,\rho_0^{-1})$. Recall that from 
  Lemma~\ref{lem:conescale} we have that the rescalings
  $r^{-1}L_{-1/2}$ are $c_0$-graphs over $W_0$ on $A_{10,1/10}$, say. 

  We claim that if $\delta$ is
  sufficiently small, and $k_2$ is large, then this implies that
  \[ \label{eq:x20} \int_{\tilde{L}\cap (B_2\setminus B_1)} |x^\perp|^2\,
    d\mathcal{H}^n > c_0, \]
  for $c_0$ depending on $W_0$ and $\kappa_1$. As in (b), if we had
  such rescalings for arbitrarily small $\delta,\kappa_0$ and large $k_2$, then
  in the limit we would obtain a special Lagrangian current
  $\tilde{L}_\infty$, asymptotic to $W_0$, and  the convergence to this
  $\tilde{L}_\infty$ is smooth on $A_{10,1/10}$ because of the
  graphicality. Because we have $\Theta(\tilde{L}_\infty, 0, r) \to \Theta(W_0)$ as $r\to\infty$ and also \eqref{eq:Th20}, $\tilde{L}_\infty$ could
  not be a cone (centered at the origin), which implies that already
  on the annulus $B_2\setminus B_1$ it cannot be a cone. It follows
  that
  \eqref{eq:x20} is satisfied for $\tilde{L}_\infty$ for some $c_0 >
  0$. So $\eqref{eq:x20}$ is also satisfied for $\tilde{L}$ if
  $\delta,\kappa_0$ and $k_2^{-1}$ are sufficiently small. 
\end{proof}

We can now prove Proposition~\ref{prop:mainelliptic}. 
\begin{proof}
We choose $\kappa_0, k_1, \kappa_1$ as in Lemma~\ref{lem:growthbound2}. We will allow for $\delta$ to be chosen even smaller below. 
Let us first consider
the case when the link of $W_0$ is connected. Let us define
$L=L_{-1/2}$, and let 
\[ \label{eq:r02} r_0 = \inf\{ R\in (0,1) \,:\, \Theta( r^{-1} L) >
  \Theta(W_0) - \kappa_1, \text{ for all }r\in (R, 1)\}, \]
so that $r_0$ tends to zero as $\delta\to 0$. 
Define $\bar\theta = \sup_{L\cap B_2} |\theta - \theta_W|$. Notice
that all rescalings $r^{-1}L$ for $r < 1$ then satisfy
$\sup_{r^{-1}L\cap B_2} |\theta-\theta_W| \leq \bar \theta$. 
The hypotheses of
Proposition~\ref{prop:mainelliptic} imply that if $\delta$ is
sufficiently small, then $d_0 := D_W(L, A_{1,\rho_0}) < \kappa_0^2$
and also $\bar\theta < \kappa_0^2$.
We might not have $\bar{\theta} \leq
\kappa_0d_0$, however for all $k>0$ such that $\rho_0^{k} >
\rho_0^{-k_1}r_0$, (with the $k_1$ in 
Lemma~\ref{lem:growthbound2}), we
claim that we must have 
\begin{itemize}
  \item[either (1)] $\bar\theta > \kappa_0 D_W(\rho_0^{-k}L, A_{1,\rho_0}), $
\item[or (2)] 
$D_W(\rho_0^{-k}L, A_{1,\rho_0}) \leq
  \rho_0^{-(2+\lambda)} D_W(\rho_0^{-(k-1)}L,
  A_{1,\rho_0}).$
\end{itemize}
Indeed, suppose that  both (1) and (2) fail for some such $k$ and let
$k_0$ be the smallest of them. Using the assumption that $\rho_0^k >
\rho_0^{-k_1}r_0$ we can argue as in the proof of Lemma~\ref{lem:growthbound2} to
see that $\rho_0^{-k}L$ satisfies the hypotheses of
Lemma~\ref{lem:3ann2}, and that we are in the growth setting of part
(b) of the Lemma. We then obtain a contradiction to
Lemma~\ref{lem:growthbound2}, with $p=2$.  

It follows from this that for all $k > 0$ such that $\rho_0^{k} > \rho_0^{-k_1}r_0$, we
have
\[ D_W(\rho_0^{-k} L, A_{1,\rho_0}) \leq C(\bar\theta + d_0)
  \rho_0^{-k(2+\lambda)} \leq Cd \rho_0^{-k(2+\lambda)}, \]
where $C$ depends on $\kappa_0$. This implies that for all $r \geq
\rho_0^{-k_1-1} r_0$
we have
\[ \label{eq:DW4} D_W(r^{-1}L, A_{1,\rho_0}) \leq Cd
  r^{-2-\lambda}. \]
To conclude we claim that if $\delta$ is chosen
sufficiently small, then in fact $\rho_0^{-k_1-1}r_0 < \delta^{-1}
d^{(2+\lambda)^{-1}}$. To see this, note that by using
Proposition~\ref{prop:N1} we can assume that for all $\tilde{r} <1$
the rescaling $\tilde{r} ^{-1}L$ is as
close as we like to some special Lagrangian current (depending on $\tilde{r}$)
by choosing $\delta$ small (independently of $\tilde{r}$). Let $r_1 =
\rho_0^{-k_1-1}r_0$, and apply \eqref{eq:DW4} to $r=r_1$. If
$dr_1^{-2-\lambda} \leq \epsilon$ for sufficiently small $\epsilon$, and $\delta$ is
sufficiently small, then this 
implies that $r_1^{-1}L$ must be very close to the
cone $W$. By choosing $\epsilon$ and $\delta$ very small (depending on $W_0, \kappa_1, k_1$), the definition of $r_0$ then implies that $r_0 < \rho_0^{1+k_1} r_1$. This contradicts the definition of $r_1$, so for these choices of $\epsilon, \delta$ we have $dr_1^{-2-\lambda} >
\epsilon$. This implies that $r_1 <
\epsilon^{-(2+\lambda)^{-1}} d^{(2+\lambda)^{-1}}$. Choosing $\delta$
  even smaller if necessary, this implies $\rho_0^{-k_1-1}r_0 < \delta^{-1}
d^{(2+\lambda)^{-1}}$ as required. The case when $n=2$ is basically
the same, using the claim in Lemma~\ref{lem:growthbound2} that the
growth rate $p=n$ cannot occur. 

\bigskip
Let us now consider the case when the link of $W_0$ is not necessarily
connected, and at the same time we have no good blowup along $L^0_t$. We want to
iterate Lemma~\ref{lem:3ann2}, but first we translate the time slice
$L_{-1/2}$ so that the conicality radius is minimized at the
origin. More precisely for any $x$ we let $\tilde{L}= L_{-1/2} -x$,
and define $r_0(x)$ by the formula \eqref{eq:r01}. Let us choose $x_0$
with $|x_0| \leq 1$ which minimizes $r_0(x)$. Note that by choosing
$\delta > 0$ sufficiently small we can ensure that $L_{-1/2}$ is also
very close to the cone $W_0$, and so the minimum $r_0 := r_0(x_0)$ as well as
a minimizer $x_0$ are both small.

We now consider $\tilde{L} = L_{-1/2} - x_0$. There are two cases
depending on how large $x_0$ is relative to $d$ in the assumptions of
Proposition~\ref{prop:mainelliptic}.

Let us suppose first that $|x_0| > 4d$.
Note that $\rho_0^{-1}\tilde{L} = \rho_0^{-1}L_{-1/2} -
\rho_0^{-1}x_0$. By assumption $L_{-1/2}$ is a $d$-graph over $W$ on
$A_{2,\rho_0/2}$, and so on a slightly smaller annulus than
$A_{2,\rho_0/2}$, $\tilde{L}$ is a $d+2|x_0|$-graph over $W$. In
particular this implies that
\[ \label{eq:DW50} D_W(\tilde{L}, A_{1,\rho_0}) \leq C |x_0|, \]
for a constant $C$ depending on $W_0$. At the
same time the distance from $W$ to $\rho_0^{-1}\tilde{L}$ on $A_{9/10,
  6/10}$ is at least
\[ \frac{1}{2}\rho_0^{-1}|x_0| - \rho_0^{-1}d \geq
  \frac{1}{4}\rho_0^{-1} |x_0|, \]
which implies that
\[ \label{eq:DW51}  D_W(\rho_0^{-1}\tilde{L}, A_{1,\rho_0}) \geq
  C^{-1}\rho_0^{-1} |x_0|, \]
for a possibly larger $C$. Combining \eqref{eq:DW50} and
\eqref{eq:DW51}, we have
\[ D_W(\rho_0^{-1} \tilde{L}, A_{1,\rho_0}) \geq  C^{-2} \rho_0^{-1}
  D_W(\tilde{L}, A_{1,\rho_0}), \]
and so if $\rho_0$ is sufficiently small (and using that $\lambda$ is
small), then
\[ D_W(\rho_0^{-1} \tilde{L}, A_{1,\rho_0}) \geq \rho_0^{-\lambda}
  D_W(\tilde{L}, A_{1,\rho_0}). \]
At the same time $\sup_{\tilde{L}\cap B_1}|\theta - \theta_W| < d$. 
Therefore in this case we are in the growth setting of
Lemma~\ref{lem:3ann2}. Lemma~\ref{lem:growthbound2} implies that this
is not possible if $d$ is sufficiently small and there is no good
blowup along the flow.

We can therefore assume that $|x_0| \leq 4d$. We now want to iterate
  Lemma~\ref{lem:3ann2} applied to $\tilde{L}$, and the given $W\in \mathcal{W}$. We
  define a sequence $W_0=W, W_1, W_2, \ldots$ inductively (as long as
  possible), and let $d_i=D_{W_i}(\rho_0^{-i}L, A_{1,\rho_0})$. Let
  us define 
  \[ \overline{\theta} = \sup_{B_2 \cap L} |\theta - \theta_{W}|. \]

  At each stage we try to apply Lemma~\ref{lem:3ann2}, part (i), to
  $\rho_0^{-i}L$ and $W_i$. In order to do this we will want to
  ensure that we still have the condition $\overline{\theta} \leq
  \kappa_0 d_i$, as well as the two conditions
\begin{itemize}
\item[$(a_i)$] $d(W_i, W_0) < \kappa_0$, 
\item[$(b_i)$] $\rho_0^{-i}L$ is a $c_0$-graph over $W_0$ on
  $A_{1,\rho_0^{3}}$ and $d_i < \kappa_0$. 
\end{itemize}
  Assuming that $(a_i), (b_i)$ hold, then 
\begin{itemize}
  \item[Case 1.] if $\bar\theta > \kappa_0 d_i$ then we let
    $W_{i+1}=W_i$, 
  \item[Case 2.] if $\bar\theta \leq \kappa_0 d_i$ then we let  $W_{i+1} = W'$, given by 
    Lemma~\ref{lem:3ann2}, part (i).
  \end{itemize}

  We claim that as long as we have $\rho_0^{i} > \rho_0^{-k_1}r_0$
  (for the $k_1$ in Lemma~\ref{lem:growthbound2})  and
  $i \leq C_2^{-1} \bar\theta^{-1}$ for a
  constant $C_2$ depending on $\kappa_0, \rho_0$, this process can
  be continued, unless at some stage we are in the ``Growth setting''
  of Lemma~\ref{lem:3ann2}. At the same time the growth setting is
  ruled out by Lemma~\ref{lem:growthbound2} if there is no good blowup
  along the flow.

  To see that the process can be continued, we need to check that
  conditions $(a_{i+1}), (b_{i+1})$ remain satisfied.
  For condition $(a_{i+1})$ note that
  \[ \label{eq:Widist} d(W_{i+1}, W_0) \leq d(W_0, W_1) + \ldots + d(W_i, W_{i+1}) \leq
    C(d_0 + \ldots + d_i).  \]
  For any $j\leq i$, in Case 2 we can assume that  $d_{j+1} \leq \rho_0^{\lambda} d_j
  \leq \frac{1}{2}d_j$ since we must be in the decay setting of
  Lemma~\ref{lem:3ann2}.
  In Case 1 we have $d_j \leq
  \kappa_0^{-1}\bar\theta$, and so $d_{j+1} \leq
  \rho_0^{-\lambda}\kappa_0^{-1}\bar\theta$, since we cannot be in the 
  growth setting. Adding up the contributions we have $d(W_{i+1}, W_0)
  \leq C_1i(d_0 + \bar\theta)$, where $C_1$ depends on $\kappa_0,
  \rho_0$. It follows that $d(W_{j+1}, W_0) <
  \kappa_0$ as long as $C_1i(d_0 + \bar\theta) < \kappa_0$, i.e. if
  $C_1id < \kappa_0$. 

  Regarding condition $(b_{i+1})$ note first that we have $d_i,
  d_{i+1} \leq C_1d< \kappa_0$, if $d$ is sufficiently small. To see that $\rho_0^{-{i+1}}L$ is a
  $c_0$-graph over $W_{i+1}$ we use Lemma~\ref{lem:conescale}. Since
  $d_i < \kappa_0$ and $d(W_i, W_0) < \kappa_0$, we know that on the
  annulus $A_{2,1}$, $2\rho_0^{-i}\tilde{L}$ is a $c_0$-graph over
  $W_0$. It follows that $2\rho_0^{-i}\tilde{L}$ is an
  $\epsilon^2$-graph over some $W\in \mathcal{W}$ on
  $A_{\epsilon^{-1}, \epsilon}$. Using $d_i < \kappa_0$ and $d(W_i,
  W_0) < \kappa_0$ we must have $d(W, W_0) < C\kappa_0$ (we can assume that
  $\epsilon$ is much smaller than $\kappa_0$). Since we can also
  arrange that $\kappa_0$ is much smaller than $c_0$, it follows from
  this that $\rho_0^{-(i+1)}\tilde{L}$ is still a $c_0$-graph over
  $W_0$. So condition $(b_{i+1})$ holds. 

  The conclusion is that as long as $\rho_0^{i} > \rho_0^{-k_1}r_0$ and $i \leq
  C_2^{-1}d^{-1}$, we can construct $W_i$ as above.   Let us
  define 
\[ i_{max} = \lfloor \min\{  \log_{\rho_0} r_0 -k_1, C_2^{-1}d^{-1}\} \rfloor. \]
  For $i < i_{max}$ we have $d(W_i, W) \leq C_4id$ from
  \eqref{eq:Widist}, and at the same
  time we have $d_i \leq C_1d$.  If
  $\rho_0^{i_{max}+1} < \rho_0^{-k_1}r_0$, then $r_0^{-1}\tilde{L}$ would be an
  $O(d)$-graph over $W_{i_{max}}$ on $A_{2,1}$, but its Gaussian area would be
  at most $\Theta(W_0) - \kappa_1$. For sufficiently small $\delta$ this is
  a contradiction, since in the limit as $\delta\to 0$, the corresponding
  $r_0^{-1}\tilde{L}$ would converge to a cone in $\mathcal{W}$. 

  It follows that if $\delta$ (and so also $d$) is sufficiently small, then
  $i_{max} = \lfloor C_2^{-1}d^{-1}\rfloor$. Let $R_0 =
  \rho_0^{i_{max}} \geq \rho_0^{-k_1}r_0$. Then for $r \in (R_0,1)$ the rescaling $r^{-1}\tilde{L}$ is a
  $C_5d|\ln r|$-graph over $W$ on the annulus $A_{1,1/2}$. Recall
  that $L = \tilde{L} + x_0$, and we have $|x_0| \leq 4d$, so it
  follows that for $r > R_0$, $r^{-1}L$ is a $C_6d
  r^{-1}$-graph over $W$ on $A_{1,1/2}$. It remains to check that $R_0
  < \delta^{-1}d$ to see that $r^{-1}L$ is a
  $C_6dr^{-1}$-graph over $W$ for all $r\in (\delta^{-1}d, 1)$, as
  required. For this note that $i_{max} > C_2^{-1}d^{-1}-1$, so 
  \[ \delta d^{-1}R_0  = \delta d^{-1} \rho_0^{i_{max}} <  \delta
    d^{-1}\rho_0^{C_2^{-1}d^{-1}-1} \to 0, \text{ as }
    d\to 0, \]
   so the required bound $R_0 < \delta^{-1}d$ holds once $d$ is small
   enough.

   The proof of the final case (c) in the proposition is very similar
   to case (b). The difference is that we do not need to translate the
   time slice, and in the growth case we can directly apply part (d) of
   Lemma~\ref{lem:growthbound2}. 
\end{proof}

\section{Decay estimate along the flow}\label{section:Decayalongflow}
As before we consider a Lagrangian mean curvature flow $L^0_t$ in
$\mathbb{C}^n$ for $t\in [-1,0)$, developing a singularity at
$(0,0)$. In this section we will be working with different rescaled
mean curvature flows $M_\tau$ obtained from $L^0_t$. More precisely we
consider rescalings $L_t$ of $L^0_t$ as in (\dag), and $M_\tau$ will
denote the rescaled mean curvature flow corresponding to $L_t$,
defined by
\[ M_\tau = e^{\tau/2} L_{-e^{-\tau}}, \]
for $\tau\in [0,\infty)$. 

We will assume the following Conditions (A):
\begin{enumerate}
  \item[(A1)] $M_\tau$ has uniformly bounded area ratios and Lagrangian
angle, for all $\tau$. Note that by the monotonicity formula this follows
if we assume that $L^0_{-1}$ has uniformly bounded area ratios and
Lagrangian angles (see Neves~\cite{Neves1}). All constants that follow will depend on
this bound on the Lagrangian angle and area ratios. 
  \item[(A2)] The Lagrangian $M_0\cap B(0,1)$ is exact, 
  i.e. $\int_\gamma\lambda =0$ for any closed
    curve $\gamma$ on it. 
  \item[(A3)] One of the following three situations holds:
    \begin{itemize}
    \item[(a)] The link of $W_0$ is connected and $n > 2$,
    \item[(b)] There is no good blowup sequence along the flow $L^0_t$
      in the sense of Definition~\ref{defn:goodblowup},
    \item[(c)] The dimension $n=2$.
      \end{itemize}
  \end{enumerate}

Our goal in this section is to show that under these assumptions
the tangent flow $W_0$ is
unique. More generally, we will show that under these conditions if the Gaussian area does not change much on a time interval $\tau\in[0, T]$, and $M_0$ is close to the cone $W_0$, then so is $M_T$. This result will lead to both the uniqueness of the
tangent flow and imply that if the tangent flow is $W_0$, then
in fact there is a good blowup
sequence for the flow. This will complete the proof of
Theorem~\ref{thm:goodblowup}. 

Given any $\epsilon > 0$ let us define
\[ T_1(\epsilon) = \sup\{ T\,:\, \text{ for all $\tau\in [0, T], M_\tau$
  is an $\epsilon^2$-graph over $W_0$ on the annulus
  $A_{\epsilon^{-1}, \epsilon}$.} \} \]
The following shows that along the rescaled flow the ``good
graphicality'' region extends exponentially fast. Similar results are
used in Colding-Minicozzi~\cite{CMunique} and Chodosh-Schulze~\cite{CSunique}
for instance. 

\begin{lemma}\label{lem:T1}
  Let $c_0 > 0$. There is an $\epsilon_0 > 0$ depending on $c_0$ such
  that for all $\tau\in [0, T_1(\epsilon_0))$ and $R\in [1, e^{\tau/2}]$ we
  have that $R^{-1}M_\tau$ is a
  $c_0$-graph over $W_0$ on the annulus $A_{2,1/2}$.
\end{lemma}
\begin{proof}
  Consider first a (non-rescaled) Lagrangian mean curvature flow $L_t$
  for $t\in [-1,0)$, with uniformly bounded area ratios and Lagrangian
  angles. Let $c_0 > 0$. We claim that if $L_{-1}$ is an
  $\epsilon^2$-graph over $W_0$ on the annulus $A_{\epsilon^{-1},
    \epsilon}$ for sufficiently small $\epsilon$, then $L_t$ is a
  $c_0$-graph over $W_0$ on the annulus $A_{2,1/2}$ for $t\in
  [-1,0)$. To see this, we can argue by contradiction: if we had a
  sequence of such flows $L^i_t$, and $\epsilon_i\to 0$, then using
  Proposition~\ref{prop:N1} the flows would converge, along a
  subsequence, to the static flow given by $W_0$. This
  convergence is uniform for $t\in [-1,0)$, but only in the sense
  of measures. We can then use White's regularity theorem to deduce
  that for sufficiently large $i$ the flow $L^i_t$ is a $c_0$-graph
  over $W_0$ on $A_{2, 1/2}$ for all $t\in [-1,0)$.

  The claim in the Lemma follows by applying this result to the flows
  given by initial condition $L_{-1} = M_\tau$ for all $\tau\in [0,
  T_1(\epsilon_0))$. 
\end{proof}

The quantity $T_1(\epsilon)$ above tells us how long the flow remains
close to $W_0$ on a large annulus. 
For any $\epsilon > 0$ we also define
\[ T_\Theta(\epsilon) = \sup\{ \tau\,:\, \Theta(M_\tau) > \Theta(W_0) -
  \epsilon\}, \]
which tells us how long the (rescaled) flow remains close to a cone in
terms of the Gaussian area. 
The main result of this section is the following, stating that as
long as the Gaussian area of $M_\tau$ remains close to
that of $W_0$, the flow will remain close to $W_0$ as a graph as
well. The proof of this proposition will take up the rest of the section. 

\begin{prop}\label{prop:T1Ttheta}
  Assume Conditions (A), and let $\epsilon > 0$.
  There is a $\delta > 0$ depending on $\epsilon$
  such that if $M_0$ is a $\delta^2$-graph over $W_0$ on
  $A_{\delta^{-1}, \delta}$, then $T_1(\epsilon) \geq T_\Theta(\delta)$. 
\end{prop}
The argument is similar to the approach taken in \cite{LSSz2},
at heart relying on proving a nonlinear version of the three annulus
Lemma~\ref{lem:L23ann2} along the rescaled flow. In \cite{LSSz2} a
non-concentration estimate was shown by exploiting the link between
special Lagrangians and complex curves in two dimensions. In the
present setting we rely on the estimate in
Proposition~\ref{prop:mainelliptic} instead, which exploits the link with
minimal submanifolds. Similarly to Definition~\ref{defn:components1} we will use the following decomposition. 

\begin{definition}\label{def:components}
Let us decompose $M_0\cap B(0,1)$ into a union of smooth connected submanifolds
\[ M_0\cap B(0,1) = M_{0,1}\cup \ldots \cup M_{0,m}. \]
By Lemma~\ref{lem:T1}, if $\epsilon_0$ is sufficiently small, then for all $\tau\in [0,T_1(\epsilon_0)]$ we have corresponding connected submanifolds $M_{\tau, j} \subset B(0, e^{\tau/2})$, and 
\[ M_\tau \cap B(0,e^{\tau/2}) = M_{\tau, 1}\cup \ldots \cup M_{\tau, m}. \]
The $M_{\tau,j}$ satisfy the rescaled mean curvature flow on their domains of definition.

Corresponding to this decomposition we also have a decomposition of the cone $W_0$:
\[\label{eq:W0decomp} W_0 = \bigcup_{j=1}^m W_{0,j},\]
such that $M_{\tau, j}$ is graphical over $W_{0,j}$ on $A_{e^{\tau/2}, 1/2}$. As a reminder, the links of the $W_{0,j}$ are not necessarily connected. 

We define $\tilde{\mathcal{W}}$ to be given by Lagrangian cones
\[ W = \bigcup_{j=1}^m W_j, \]
where each $W_j$ is a special Lagrangian deformation of $W_{0,j}$, but for different $j$ the Lagrangian angles may be different. Recall that we also have $\mathcal{W}\subset \tilde{\mathcal{W}}$ consisting of special Lagrangian cones. In the notation of Definition~\ref{defn:components1} we have $\tilde{\mathcal{W}} = \mathcal{W}_{M_\tau}$ for any $\tau$. 
\end{definition}

Given the decomposition of $M_\tau\cap B(0,e^{\tau/2})$ and $W\in \tilde{\mathcal{W}}$ above, we define the angle function 
\[ \label{eq:thetaWdefn}
\theta_W : M_\tau\cap B(0,e^{\tau/2}) \to \mathbb{R} \]
by letting $\theta_W = \theta_{W_j}$ on the component $M_{\tau, j}$. So on the different connected components $\theta_W$ may take different constant values. Using this we define the following distance type function of $M_\tau$ to a cone $W\in \tilde{\mathcal{W}}$:
\[ \label{eq:EWdef} E_W(M_\tau) = \sup_{B_2\cap M_\tau} r^{\alpha}d_W +
  \left(\int_{M_\tau\cap B(0,e^{\tau/2})} (\chi d_W^2 + |\theta - \theta_W|^2) e^{-|x|^2/ 4}\,
    d\mathcal{H}^n\right)^{1/2}, \]
where $\chi: \mathbb{C}^n\to [0,1]$ vanishes in $B_{1/2}$ and equals 1
outside of $B_1$, and
where according to the three cases in Condition (A3) we have: 
\begin{itemize}
  \item[(a)]  If $W_0$ has connected link and $n > 2$, we set
    $\alpha\geq 1+\lambda_2$ for the $\lambda_2$ in
    Proposition~\ref{prop:mainelliptic}. For simplicity we can let
    $\alpha=1.1$. 
  \item[(b)] When there is no good blowup sequence we let
    $\alpha=0$,
  \item[(c)] When $n=2$ we set $\alpha = 1+\lambda_2$ with $\lambda_2$ as in (a),  but we
    choose $\lambda_2 > 0$ very small, to be chosen below in the proof
    of Proposition~\ref{prop:decay4}. Note that the function $\chi$ in
    \eqref{eq:EWdef} only plays a role in this case, since in the
    other two cases bounding $r^\alpha d_W$ on $B_2$ automatically bounds the
    integral of $d^2_W$ on $B_2$. 
  \end{itemize}

  For technical reasons later it will be convenient to define the
following variant as well:
\[ \label{eq:tildeEdefn} \tilde{E}_W(M_\tau) = \sup_{s\in [\tau-3, \tau]} E_W(M_\tau), \]
so that $\tilde{E}_W(M_\tau)$ gives bounds for $M_s$ on the interval $s\in
[\tau-3, \tau]$. 
We have the following extension result, that follows using an argument
by contradiction and Proposition~\ref{prop:N1}.

\begin{lemma}\label{lem:ext1}
  Assume Conditions (A) and let $B >0$. Given $\epsilon_0 > 0$ there is an
  $\epsilon_1 > 0$ depending on $B, \epsilon_0$
  such that if $\tau< T_1(\epsilon_0)$ and $E_W(M_\tau) < \epsilon_1$ for
  some $W\in
  \tilde{\mathcal{W}}$ with $d(W_0, W) < \epsilon_1$,  then $T_1(\epsilon_0)
  > \tau + 10B$. 
\end{lemma}
\begin{proof}
  Suppose that for a given $\epsilon_0, B$ we did not have a suitable
  $\epsilon_1$. Then we would have a sequence of rescaled flows
  $M^i_\tau$, satisfying Conditions (A), and corresponding times $\tau_i <
  T^i_1(\epsilon_0)$ (where $T^i_1$ are the $T_1$ functions
  corresponding to the flows $M^i_\tau$), such that $E_{W_i}(M^i_{\tau_i}) <
  1/i$, for $W_i$ with $d(W_0, W_i) < 1/i$, but $T^i_1(\epsilon_0)
  \leq \tau_i + 10B$.

  Define the translated flows $\tilde{M}^i_\tau = M^i_{\tau+\tau_i}$. Using
  Proposition~\ref{prop:N1} we know that up to choosing a subsequence,
  the rescaled flows $\tilde{M}^i_\tau$ converge to a static flow given
  by a special Lagrangian $L_\infty$. We necessarily must have
  $L_\infty = W_0$, and so by White's regularity theorem the
  convergence is locally smooth for $\tau > 0$. It then follows that
  $\tilde{T}^i_1(\epsilon_0) \to \infty$ (for the $T_1$-functions
  corresponding to $\tilde{M}^i_\tau$), and so we also have
  $T^i_1(\epsilon_0) \to \infty$, contradicting our assumption. 
\end{proof}

In addition we have the following non-concentration estimate for the
distance $E_W$, under our conditions. This is the key result that
allows us to pass information from the linear drift heat equation to the
flow.  
\begin{lemma}\label{lem:nonconc2}
  Assume Conditions (A). There are $\epsilon_0,
  C > 0$ and $p > 1$ depending on $L^0_t$ satisfying the following. Suppose that
  $E_W(M_{s}) < d $ for $s\in [\tau_0-1, \tau_0]$ for some $W\in \tilde{\mathcal{W}}$ 
  with $d(W, W_0) < \epsilon_0$, and $\tau_0 <
  T_1(\epsilon_0)-1$. 
  Suppose in addition that on the annulus
  $A_{2,\rho_0/2}$, $M_s$ is a $d$-graph over $W$ for $s\in [\tau_0-1, \tau_0+1]$ (in
  the sense of Definition~\ref{defn:cgraph}). Define $\bar{d} =
  \max\{d,  e^{-\tau_0}\}$, and assume that $\bar{d} < \epsilon_0$. 
  Then we have the estimate
  \[ d_W^2 \leq C e^{\frac{|x|^2}{4p}} \bar{d}^2 \max\{r^{-2\alpha}, 1\}\]
  on $M_{\tau_0+1}$. In addition we also have
  \[ \left[\int_{M_{\tau_0+1}\cap B(0,e^{\tau_0/2})} ( \chi d_W^{2p} + |\theta - \theta_W|^{2p})\,
      e^{-|x|^2/4}\, d\mathcal{H}^n \right]^{1/p}\leq C \bar{d}^2, \]
  where $\chi$ is as in \eqref{eq:EWdef}. 
\end{lemma}
\begin{proof}
  Let us first focus on the region $\mathbb{C}^n \setminus
  B(0,1)$. Here we use that by Lemma~\ref{lem:T1} we have good
  graphicality along the flow for $\tau\in [\tau_0, \tau_0+1]$, at least where
  $|x| < e^{\tau_0/2}$. The distance $d_W$ then satisfies a differential
  inequality on this region.

  More precisely, let us write the flow $M_\tau$ for $\tau\in[\tau_0, \tau_0+1]$ as a family of normal graphs over the minimal cone $W$ on the annulus $A_{e^{\tau_0/2},1/2}$. Thus, we have normal vector fields 
  \[ v_\tau : W \cap A_{e^{\tau_0/2}, 1/2} \to \mathbb{C}^n, \]
  such that a region of $M_\tau$ (in a slight perturbation of the annulus) is parametrized by the points $x + v_\tau(x)$, with $x\in W\cap A_{e^{\tau_0/2}, 1/2}$.
  The evolution equation for $v_\tau$ is of the form
  \[ \frac{\partial}{\partial \tau} v_\tau = \Delta v_\tau + A\ast A\ast v_\tau + 
  \frac{1}{2}(v_\tau - x\cdot \nabla v_\tau) + |x|^{-1} Q(|x|^{-1}x, |x|^{-1}v, \nabla v, |x| \nabla^2 v).\]
  Here $A\ast A\ast v_\tau$ is a contraction of the second fundamental form $A$ with $v_\tau$, and $Q = Q(y,p,q,r)$ is a power series with terms that are all at least quadratic in $p,q,r$, however $r$ appears at most linearly. The Laplacian $\Delta$ is the connection Laplacian $-\nabla^*\nabla$.
  
  The graphicality of $M_\tau$ over $W$ from Lemma~\ref{lem:ext1} implies that we have bounds of the form $|\nabla^i v_\tau| 
  \leq c_0 |x|^{1-i}$ for $i=0,1,2$, where $c_0$ can be as small as we
  like, if $\epsilon_0$ is chosen sufficiently small.
Using this, for sufficiently small $c_0$
  we have the differential inequality
  \[ |\partial_\tau v_\tau - \Delta v_\tau + \frac{1}{2}x\cdot \nabla v_\tau| \leq C(|v| + |\nabla v|). \]
 Here $C$ depends on the cone $W_0$. In order to absorb the
  gradient term we compute the evolution of $|v|^\alpha$ for
  $\alpha > 1$. Working with $\alpha = 3/2$ for simplicity we obtain a
  differential inequality of the form
  \[ \partial_\tau |v_\tau|^{3/2} - \Delta |v_\tau|^{3/2} + \frac{1}{2}x\cdot \nabla |v_\tau|^{3/2}\leq
    C_1|v_\tau|^{3/2}, \]
  for a constant $C_1$ depending on $W_0$. In particular $e^{-C_1(\tau - \tau_0)}
  |v_\tau|^{3/2}$ is a subsolution of the drift heat equation on $W$, at least on the
  annulus $A_{e^{\tau_0/2}, 1/2}$, for $\tau\in [\tau_0, \tau_0+1]$. Note that $|v_\tau|(x) = d_W(x + v_\tau(x))$, and under the parametrization of $M_\tau$ using $v_\tau$ the volume forms of $M_\tau$ and $W$ are uniformly equivalent. Therefore the integrals of $d_W^{2p} e^{-|x|^2/4}$ over a region in $M_\tau$ and of $|v_\tau|^{2p} e^{-|x|^2/4}$ over the corresponding region on $W$ are uniformly equivalent. 

  Let us define
  \[ f(x,\tau) = e^{-C_1(\tau-\tau_0)} |v_\tau|^{3/2}, \]
  so that $f$ is a subsolution of the drift heat equation on $W$, 
  on the annulus
  $A_{e^{\tau_0/2}, 1/2}$, for $\tau\in [\tau_0, \tau_0+1]$. At the same time, along the mean curvature flow $(\partial_s - \Delta) |x|^2 =-2n$, so
  $|x|^2$ is a subsolution of the heat
  equation. It follows that $e^{-(\tau-\tau_0)}|x|^2$ is a subsolution of the drift heat
  equation along the rescaled flow. Let us define the function
  \[ \tilde{f} = \begin{cases} e^{-(\tau-\tau_0)}|x|^2 - e^{\tau_0/2}, &\text{
        for } |x| \geq e^{\tau_0/2}, \\
      \max\{ f, e^{-(\tau-\tau_0)} |x|^2 - e^{\tau_0/2}, 2\bar{d}^{3/2}\}, &\text{
        for } 2 < |x| < e^{\tau_0/2}, \\
      2\bar{d}^{3/2}, &\text{ for }|x| \leq 2. \end{cases} \]
  Note that on any ball $B(0,R)$ we automatically have $|v_\tau| \leq R$,
  and in addition on the annulus $A_{2,1/2}$ by our assumption we have
  $|v_\tau| \leq \bar{d}$. Using this we can see that $\tilde{f}$ is continuous,
  and is a subsolution of the drift heat equation along $W$ for
  $\tau\in [\tau_0, \tau_0+1]$.

  We have
  \[ \tilde{f}^{4/3} \leq \begin{cases} C\bar{d}^2, &\text{ for }|x| \leq 2,
      \\
      C(\bar{d}^2 + |v_\tau|^2), &\text{ for }2 < |x| < e^{\tau_0/4}, \\
      C|x|^{8/3}, &\text{ for }|x| \geq e^{\tau_0/4}. \end{cases} \]
  Using that $\bar{d} \geq e^{-\tau_0}$, this implies that if $\bar{d}$ is sufficiently
  small, then
  \[ \int_{W} \tilde{f}^{4/3}\, e^{-|x|^2/4}\, d\mathcal{H}^n
    \leq C \bar{d}^2, \text{ for } \tau=\tau_0,\]
  for a suitable constant $C$ depending on $W_0$.

  Using Ecker's log-Sobolev inequality~\cite{Ecker} we find that for some $q > 1$,
  and replacing $C$ with a larger constant, we have
  \[ \int_{W} \tilde{f}^{4q/3} e^{-|x|^2/4}\, d\mathcal{H}^n
    \leq C \bar{d}^{2q}, \text{ for } \tau=\tau_0+1/2. \]
 Arguing similarly to the proof of \cite[Lemma 3.5 (2)]{LSSz2}
  we also obtain a pointwise bound of the
  form
  \[ \tilde{f}^{4/3} \leq C e^{\frac{|x|^2}{4q}} \bar{d}^2, \text{ for }\tau=\tau_0+1. \]
  In particular, using the definition of $\tilde{f}$, this implies
  that on $M_{\tau_0+1} \cap A_{e^{\tau_0/4},2}$ we have the estimate 
  \[ d_W^2 \leq Ce^{\frac{|x|^2}{4q}} \bar{d}^2. \]
  Note that this estimate is automatic for $|x| > e^{\tau_0/4}$ using
  that $d_W \leq |x|$ and $\bar{d}\geq e^{-\tau_0}$, and the same pointwise
  estimate also holds with $q$ replaced by any $p < q$. 

  The required estimate for $d_W$ inside the ball $B(0,2)$ follows
  directly  from Proposition~\ref{prop:mainelliptic}, once we obtain a pointwise bound for the Lagrangian angle. By integrating the pointwise estimate,
  we obtain the required integral estimate for $d_W^{2p}$ for any $p <
  q$. Thus it remains to estimate $|\theta - \theta_W|$. 

  Note that on the ball $B(0, e^{\tau_0/2})$ the function $\theta-\theta_W$ is a solution of the drift heat equation, and moreover we are assuming a uniform bound for $\theta$. Similarly to the definition of $\tilde{f}$ above we can define the function
  \[ \label{eq:tildetheta} \tilde{\theta} = \begin{cases} e^{-(\tau-\tau_0)}|x|^2 - e^{\tau_0/2}, &\text{ for } |x| \geq e^{\tau_0/2}, \\
  \mathrm{max}\{ |\theta - \theta_W|, e^{-(\tau-\tau_0)}|x|^2 - e^{\tau_0/2}\}, &\text{ for }|x| < e^{\tau_0/2}, \end{cases}
  \]
  on $M_\tau$ for $\tau\in [\tau_0, \tau_0+1]$. 
  The uniform bound for $|\theta-\theta_W|$ implies that if $\tau_0$ is sufficiently large, then for $|x| = e^{\tau_0/2}$ and $\tau\in [\tau_0, \tau_0+1]$ we have 
  \[ e^{-(\tau-\tau_0)}|x|^2 - e^{\tau_0/2} \geq e^{\tau_0-1} - e^{\tau_0/2} > |\theta-\theta_W|, \]
  so $\tilde\theta$ is a continuous subsolution of the drift heat equation along $M_\tau$. In addition, when $|x| < e^{\tau_0/4}$, then $\tilde\theta = |\theta-\theta_W|$. It follows, using the assumption $E_W(M_\tau), e^{-\tau_0} \leq \bar{d}$ as well, that
  \[ \int_{M_{\tau_0}} \tilde{\theta}^2\, e^{-|x|^2/4}\, d\mathcal{H}^n \leq C\bar{d}^2. \]
  Ecker's log-Sobolev inequality~\cite{Ecker} then leads to the estimate 
  \[ \left(\int_{M_{\tau}} \tilde{\theta}^{2q}\, e^{-|x|^2/4}\, d\mathcal{H}^n\right)^{1/q} \leq C'\bar{d}^2, \]
  for some $q > 1$ for $\tau\in [\tau_0+\frac{1}{2}, \tau_0+1]$. This leads to the required
  integral estimate for $|\theta - \theta_W|^{2p}$ on $M_{\tau_0+1}$, and in addition, the monotonicity formula applied to $\tilde\theta$ leads to the pointwise estimate $|\theta-\theta_W| \leq C\bar{d}$ on $M_{\tau_0+1}\cap B(0,2)$, so we can apply Proposition~\ref{prop:mainelliptic} to obtain the required pointwise distance bounds on $M_{\tau_0+1}\cap B(0,1)$. The integral estimate for $d_W^{2p}$ then follows by integrating the pointwise bound. 
\end{proof}

We will use the non-concentration estimate in the following form
several times. 
\eqref{eq:tildeEdefn}. 
\begin{cor}\label{cor:nonconc}
  Assume Conditions (A). There are
  $\epsilon_0, C > 0$ depending on 
  $L^0_t$ satisfying the following. Suppose that $E_W(M_{s}) < d$ for some
  $W\in\tilde{\mathcal{W}}$ such that $d(W,W_0) <
  \epsilon_0$, for $s\in [\tau_0-3, \tau_0]$, and $\tau_0 < T_1(\epsilon_0)-1$. Suppose that on the
  annulus $A_{2, \rho_0/2}$, $M_s$ is a $d$-graph over $W$ for $s\in
  [\tau_0-3, \tau_0+1]$. Let $\bar{d} = \max\{ d, e^{-\tau_0}\}$. 

  For any $\gamma > 0$ there exists a compact set $K_\gamma\subset
  \mathbb{C}^n\setminus\{0\}$ satisfying the following.
  Whenever on $K_\gamma\cap M_s$, for $s\in [\tau_0-2,
    \tau_0+1]$  we have 
  \[ \label{eq:est1onK} d_W^2 \leq \gamma^2 e^{\frac{|x|^2}{4q}}
    \bar{d}^2\, \text{ for } |x| > \rho_0/4 \]
  and
  \[ \label{eq:est2onK} |\theta - \theta_W|^2 \leq \begin{cases} \gamma^2 e^{\frac{|x|^2}{4q}}
      \bar{d}^2\,\max\{|x|^{-2}, 1\}, &\text{ if } n > 2, \\
\gamma^2 e^{\frac{|x|^2}{4q}}
      \bar{d}^2, &\text{ if } n =2, \end{cases}
    \]
  for some $q> 1$, then for $\bar{d}_\gamma = \max \{\gamma d,
  e^{-\tau_0}\}$ we have the improved estimate $E_W(M_{s})
  \leq C\bar{d}_\gamma$ for $s\in [\tau_0, \tau_0+1]$. 
\end{cor}
\begin{proof}
  First note that by Lemma~\ref{lem:nonconc2} we have the $L^p$
  bound
   \[ \left[\int_{M_{s}\cap B(0, e^{(\tau_0-1)/2})} ( \chi d_W^{2p} + |\theta - \theta_W|^{2p})\,
       e^{-|x|^2/4}\, d\mathcal{H}^n \right]^{1/p}\leq C \bar{d}^2,  \]
   for $s\in [\tau_0-1, \tau_0+1]$. It follows that for these values of $s$,
   for any compact set $K$, if we have the pointwise estimates
   \eqref{eq:est1onK} and \eqref{eq:est2onK}, then
   \[ \int_{M_s\cap B(0, e^{(\tau_0-1)/2})} (\chi d_W^2 + |\theta - \theta_W|^2) e^{-|x|^2/4}\,
     &d\mathcal{H}^n  = \int_{M_s\cap K} (\chi d_W^2 + |\theta - \theta_W|^2) e^{-|x|^2/4}\,
     d\mathcal{H}^n \\
   &\quad + \int_{M_s\cap B(0, e^{(\tau_0-1)/2})\setminus K} (\chi d_W^2 + |\theta - \theta_W|^2) e^{-|x|^2/4}\,
   d\mathcal{H}^n \\
 &\leq C\gamma^2\bar{d}^2 + C\bar{d}^2 \left( \int_{M_s\cap B(0,e^{(\tau_0-1)/2})\setminus K} e^{-|x|^2/4}\,d\mathcal{H}^n\right)^{1-1/p}.\]
   Using the bound on the area ratios we can choose $K$ in a suitable
   way (depending on $\gamma$), such that
   \[ \label{eq:dwthetaL2} \int_{M_s\cap B(0, e^{(\tau_0-1)/2})} (\chi d_W^2 + |\theta - \theta_W|^2) e^{-|x|^2/4}\,
     d\mathcal{H}^n \leq C\gamma^2\bar{d}^2, \]
   for a suitable $C > 0$. 

   Using the monotonicity formula applied to the function $\tilde{\theta}$ considered in the proof of Lemma~\ref{lem:nonconc2}, it follows that for $s\in
   [\tau_0, \tau_0+1]$, on $B_2(0)\cap M_{s}$
   we have the pointwise bound
   \[ |\theta - \theta_W|^2 \leq C\gamma^2\bar{d}^2. \]
   If $\epsilon_0$ is chosen sufficiently small, then we can apply
   Proposition~\ref{prop:mainelliptic} using \eqref{eq:est1onK} to see
   that the conclusions (a) and (b) of the Proposition hold with
   $d=C\gamma\bar{d}$. This implies in particular that for these $s$, in $B_1(0)\cap
   M_{s}$, we have $d_W \leq C\gamma\bar{d}
   r^{-\alpha}$. Combined with \eqref{eq:dwthetaL2}, this shows that
   $E_W(M_{s}) \leq C\gamma\bar{d}$. 
\end{proof}

Recall \eqref{eq:tildeEdefn} for the
definition of $\tilde{E}_W$. 
Using Lemma~\ref{lem:L23ann2} together with the non-concentration
estimate, we have the following. 
\begin{lemma}\label{lem:3ann3}
  Assume Conditions (A). Let $\lambda\in
  (\lambda_1, \lambda_2)$ for the constants in
  Lemma~\ref{lem:L23ann2} with $d=0$. 
  There are $\epsilon_0 > 0$ and a large  $B > 0$ depending on $L^0_t$,
  satisfying the following. Suppose that $\tau < T_1(\epsilon_0) - 2B$,
  $\tilde{E}_W(M_{\tau+B}) > e^{-\tau}$, and $W\in
  \tilde{\mathcal{W}}$ with $d(W, W_0) < 
  \epsilon_0$. Then if 
  \[ \tilde{E}_W(M_{\tau+ B}) \geq e^{B\lambda} \tilde{E}_W(M_\tau), \]
  then we have
  \[ \tilde{E}_W(M_{\tau+2B}) \geq e^{B\lambda} \tilde{E}_W(M_{\tau+B}). \]
\end{lemma}
\begin{proof}
  The proof is by contradiction, analogous to the proof of
  \cite[Proposition 6.2]{LSSz2}. Suppose that we have a sequence of
  flows $M^i_\tau$ satisfying Conditions (A), and we have times
  $\tau_i\to\infty$, $W_i\in \tilde{\mathcal{W}}$ satisfying $W_i\to W_0$, and
  $\epsilon_i\to 0$ such that $\tau_i < T_1^i(\epsilon_i) - 2B$ for large
  $B$. Here we write $T_1^i$ for the $T_1$ function corresponding to
  the flow $M^i$. Let us write $d_i = \tilde{E}_{W_i}(M^i_{\tau_i+B})$. We
  assume that the conclusion of the lemma fails, so that
  \[ \label{eq:EWbounds2} \tilde{E}_{W_i}(M^i_{\tau_i}) &\leq e^{-B\lambda} d_i, \\
    \tilde{E}_{W_i}(M^i_{\tau_i+2B}) &\leq e^{B\lambda} d_i. \]
  Let us also define
  \[ \tilde{E}_i = \sup_{s\in [0,2B]} \tilde{E}_{W_i}(M^i_{\tau_i+s}). \]
  
  We first show that we have a constant $C_4$, depending on $W_0,
  \lambda, B$ such that $\tilde{E}_i \leq C_4d_i$. From the assumptions it
  follows that for $s\in [0,2B]$ we can write $M^i_{\tau_i + s}$ as the
  graph of a one form $\eta_i(s)$ over $W_i$ on larger and larger
  subsets of $W_i\setminus\{0\}$, and up to choosing a subsequence the
  forms $\tilde{E}_i^{-1}\eta_i$ converge locally smoothly to a solution
  $\eta$ of the drift heat equation on $W_0$. This solution satisfies
  that $|x|^\alpha|\eta|$ and $d^*\eta$ are bounded on
  $B_2(0)\setminus \{0\}$, and in addition
  $\eta$ is in the Gaussian $L^2$-space on $W_0\setminus B_1(0)$. If no constant $C_4$ exists as
  claimed, then $\tilde{E}_i / d_i\to \infty$, and it follows that $\eta(s)=0$
  for $s\in [B-1,B]$. This implies that $\eta$ is identically zero for
  all $s$.
  We can use Lemma~\ref{lem:nonconc2} to show that this contradicts
  the assumption that $\tilde{E}_i = \tilde{E}_{W_i}(M^i_{\tau_i + s_i})$ for some
  $s_i\in [0,2B]$.

  To see this, note that from the definition of $\tilde{E}_W$, the bounds
  \eqref{eq:EWbounds2}, and the fact that $\tilde{E}_i/d_i\to\infty$, we can
  assume that $s_i\in [3,2B-3]$. Then $E_{W_i}(M^i_{\tau_i+s}) \leq
  \tilde{E}_i$ for $s\in [s_i-6, s_i]$. We apply Corollary~\ref{cor:nonconc} with $\tau_0 = \tau_i +
  s$ for $s\in [s_i-3, s_i]$ and $\epsilon_0 = \epsilon_i$. Using that the limit $\eta$ of
  the normalized one-forms $\tilde{E}_i^{-1}\eta_i$ vanishes, it follows that
  for any $\gamma > 0$, the estimates \eqref{eq:est1onK}, \eqref{eq:est2onK} hold once $i$
  is sufficiently large. By Corollary~\ref{cor:nonconc} we have
  \[ \tilde{E}_i = \tilde{E}_{W_i}(M^i_{\tau_i+s_i}) \leq C \max\{\gamma \tilde{E}_i,
    e^{-\tau_i}\}, \]
  but note that
  we also have $e^{-\tau_i} < \tilde{E}_{W_i}(M^i_{\tau_i+B}) = d_i$.
  For sufficiently large $i$, if $\gamma$ is
  chosen small enough, this is a contradiction. 
   Therefore, we have $\tilde{E}_i \leq C_4d_i$ for all
   $i$ for a constant $C_4$.

   We now repeat the argument above, writing $M^i_{\tau_i+s}$ as the
   graphs of $\eta_i(s)$ over $W_i$ on larger and larger regions. This
   time we consider the rescalings $d_i^{-1}\eta_i$. As above, we find
   that along a subsequence $d_i^{-1}\eta_i \to \eta_\infty$, a
   solution of the drift heat equation on $W_0$, where the convergence
   is smooth on compact subsets of $W_0\setminus\{0\}$. By
   the assumptions \eqref{eq:EWbounds2} we have
   \[ \Vert \eta_\infty(0)\Vert &\leq e^{-B\lambda},  \\
     \Vert \eta_\infty(2B)\Vert &\leq e^{B\lambda}, \]
   using the norms as in \eqref{eq:normdefn10}. 
   Using Lemma~\ref{lem:L23ann2} with $d=0$ we find that
   given any $\gamma > 0$, if $B$ is chosen sufficiently large (depending on
   $\lambda, \gamma$), then
   \[ \label{eq:gammaest10} \Vert \eta_\infty(s)\Vert &\leq \gamma, \text{ for } s\in
     [B-10,B+10]. \]
   Indeed, suppose first that $\Vert \eta_\infty(B+10)\Vert >
   \gamma$. If
   \[ \Vert \eta_\infty(i+1)\Vert \leq e^{\lambda_1} \Vert
     \eta_\infty(i)\Vert \text{ for all } i < B+10, \]
   then we have
   \[ e^{-B\lambda} \geq \Vert \eta_\infty(0) \Vert \geq e^{-(B+10)\lambda_1} \Vert
     \eta_\infty(B+10)\Vert > e^{-(B+10)\lambda_1} \gamma. \]
   If $B$ is sufficiently large, depending on $\gamma$, then this is a
   contradiction, since $\lambda > \lambda_1$. It follows that for
   some $i_0 < B+10$ we have
   \[ \Vert \eta_\infty(i_0+1)\Vert > e^{\lambda_1} \Vert
     \eta_\infty(i_0)\Vert, \]
   and this implies, by Lemma~\ref{lem:L23ann2}, that
    \[ \Vert \eta_\infty(i+1)\Vert \geq e^{\lambda_1} \Vert
       \eta_\infty(i)\Vert, \text{ for all }i \geq B+10. \]
     It follows that
    \[ e^{B\lambda} \geq \Vert \eta_\infty(2B)\Vert \geq e^{(B-10)\lambda_2} \Vert
      \eta_\infty(B+10)\Vert > e^{(B-10)\lambda_2} \gamma. \]
    Again, this is a contradiction if $B$ is large, depending on
    $\gamma$, since $\lambda_2 > \lambda$. 
   This shows that we must have $\Vert
   \eta_\infty(B+10)\Vert\leq    \gamma$, and a very similar
   argument shows that $\Vert
   \eta_\infty(s)\Vert\leq    \gamma$ for all $s\in [B-10,B+
   10]$.

   Using Lemma~\ref{lem:pointwise2}, \eqref{eq:gammaest10} implies the estimates
   \[\label{eq:pointwise3} |\eta_\infty(s)|^2(x) \leq C\gamma^2
     e^{\frac{|x|^2}{4p}} \text{ for } |x| > \rho_0/4, \]
   and
   \[ \label{eq:pointwise31}
     |d^*\eta_\infty(s)|^2(x) \leq \begin{cases}
       C\gamma^2 e^{\frac{|x|^2}{4p}} \max\{ |x|^{-2}, 1\}, &\text{ if
       } n >2, \\
       C\gamma^2 e^{\frac{|x|^2}{4p}}, &\text{ if }n=2. \end{cases}
     \]
   for some $p > 1$, and $s\in [B-9,B+10]$. On any compact
   set $K\subset W_0\setminus\{0\}$ we have
   $d_i^{-1}(\theta_{M^i_{\tau_i+s}}- \theta_{W_i}) \to d^*\eta_\infty$
   uniformly. This, together with the pointwise estimate
   \eqref{eq:pointwise31}, implies that given $K,\gamma$, for large
   enough $i$ we have
   \[ \label{eq:p30} |\theta_{M^i_{\tau_i+s}} - \theta_{W_i}|
     \leq \begin{cases}
       Cd_i \gamma e^{\frac{|x|^2}{8p}}\max\{|x|^{-1}, 1\}, &\text{ if
       }n > 2, \\
       Cd_i \gamma e^{\frac{|x|^2}{8p}}, &\text{ if } n=2, \end{cases} \]
   for $s\in [B-9, B+10]$, on $K$. 
   At the same time we have the bound $\tilde{E}_i\leq C_4d_i$ from the first
   part of the argument above. Using also \eqref{eq:pointwise3},
   Corollary~\ref{cor:nonconc} then implies
   that for any $\gamma > 0$, once $i$ is sufficiently large, we have
   \[ d_i = \tilde{E}_{W_i}(M^i_{\tau_i+B}) \leq C\max\{ \gamma C_4d_i,
     e^{-\tau_i-B}\}. \]
   Here $C_4$ depends on $B$, but $C$ does not. Recall also that
   $e^{-\tau_i} \leq d_i$. We can therefore first choose
   $B$ sufficiently large so that $Ce^{-\tau_i-B} < \frac{d_i}{2}$. This
   choice determines $C_4$, and we can choose $\gamma$ small so that
   $C\gamma C_4 d_i < \frac{d_i}{2}$. With these choices we obtain
   $d_i < \frac{d_i}{2}$, which is the required contradiction. 
\end{proof}

From now on we will consider $B$ fixed and large enough
so that this result holds. Iterating this result we obtain the following. 

\begin{cor}\label{cor:growthcase}
  Assume Conditions (A), and let $\epsilon_1 < \epsilon_0$, for the
  $\epsilon_0$ in Lemma~\ref{lem:3ann3}.  There exists $\epsilon_2 > 0$ depending on
  $L^0_t, \epsilon_1$ satisfying the following. Suppose that $M_0$ is an
  $\epsilon_2^2$-graph over $W_0$ on the annulus $A_{\epsilon_2^{-1},
    \epsilon_2}$. In addition suppose that for some $\tau
  < T_1(\epsilon_1) - 2B$ we have
  \[ \tilde E_W(M_{\tau+B}) \geq e^{\lambda B} \tilde E_W(M_\tau), \]
  where $d(W,W_0) < \epsilon_2$ and $\tilde E_W(M_{\tau+B}) > e^{-\tau}$.
  Then we have $T_1(\epsilon_1) > T_\Theta(\epsilon_2)$. 
\end{cor}
\begin{proof}
  We will argue by contradiction, assuming that no such $\epsilon_2$
  exists. Suppose that we have a sequence $\epsilon_{i,2} \to 0$ and flows $M^i_\tau$, such that
  $M^i_0$ is an $\epsilon_{i,2}^2$-graph over $W_0$ on the annulus
  $A_{\epsilon_{i,2}^{-1}, \epsilon_{i,2}}$, and that for some $\tau_i < T^i_1(\epsilon_1) - 2B$ we
  have
  \[ \label{eq:g40}\tilde E_{W_i}(M^i_{\tau_i+ B}) \geq e^{B\lambda}
    \tilde E_{W_i}(M^i_{\tau_i}), \]
  where $d(W_i,W_0) < \epsilon_{i,2}$, and $\tilde E_{W_i}(M^i_{\tau_i+B}) > e^{-\tau_i}$. We will show
  that for sufficiently large $i$ we have $T^i_1(\epsilon_1) >
  T^i_{\Theta}(\epsilon_{i,2})$. Here $T^i_1, T^i_{\Theta}$ are the
  functions corresponding to the flows $M^i$.

  First note that we can iterate Lemma~\ref{lem:3ann3}, using the
  assumption \eqref{eq:g40}, to obtain
  \[ \label{eq:g41} \tilde E_{W_i}(M^i_{\tau_i+ kB}) \geq e^{B\lambda}
    \tilde E_{W_i}(M^i_{\tau_i +
      (k-1)B}), \]
  as long as $\tau_i + kB < T^i_1(\epsilon_1)$. Indeed the condition
  $\tilde E_{W_i}(M^i_{\tau_i+(k-1)B}) > e^{-\tau_i-(k-2)B}$ will hold
  automatically along the iteration.  Let us write $k_i$ for
  the largest such integer $k$ for each $i$, so that
  \[ \tau_i + (k_i+1)B \geq T^i_1(\epsilon_1) > \tau_i + k_iB. \]
  Suppose in addition, by contradiction, that we have
  $T^i_{\Theta}(\epsilon_{i,2}) \geq T^i_1(\epsilon_1)$ for infinitely
  many 
  $i$. Using an argument similar to Lemma~\ref{lem:conescale} we find
  that along a subsequence the flows $\tilde{M}^i_s =
  M^i_{T^i_1(\epsilon_1) + s}$ for $s\in [- 10B,10B]$ converge to a static flow given by a cone
  $W'$, where necessarily $d(W', W_0) < C\epsilon_1^2$ (using the
  $\epsilon_1^2$-graphicality over $W_0$ at time
  $T^i_1(\epsilon_1)$). We cannot have $W'=W_0$, since if $\tilde{M}^i_s$ were
  to converge to $W_0$ for all $s\in [-10B, 10B]$, then for large $i$
  we would have $T^i_1(\epsilon_1) \geq T^i_1(\epsilon_1) + 10B$. 

  It follows that $\tilde E_{W_i}(M^i_{\tau_i+(k_i-1)B}) \not\to 0$, and so the
  convergence of the flows to a static flow implies
  \[ \lim_{i\to \infty} \frac{\tilde E_{W_i}(M^i_{\tau_i +
        k_iB})}{\tilde E_{W_i}(M^i_{\tau_i + (k_i-1)B})} =1.  \]
   This contradicts \eqref{eq:g41} for large $i$. 
\end{proof}

We now come to the main decay estimate along the rescaled flow. We
first focus on the setting where we assume Condition (A3(a)) or (A3(b)). The
case of Condition (A3(c)) will be treated below, using some additional
ideas from \cite{LSSz2}. 
\begin{prop}\label{prop:decay4}
  Assume Conditions (A), with either (A3(a)) or (A3(b)) satisfied. 
  There is an $\epsilon_0 > 0$ and $N, C > 0$ depending on $L^0_t$ with
  the following property. Suppose that $\tilde E_W(M_\tau) > e^{-\tau}$ for some
  $W$ with $d(W, W_0) < \epsilon_1<\epsilon_0$, and $\tau <
  T_1(\epsilon_1) - \epsilon_0^{-1}$. 
  Suppose that $T_1(\epsilon_1) \leq
  T_\Theta(\epsilon_2)$ for the $\epsilon_2$ determined by
  $\epsilon_1$ in
  Corollary~\ref{cor:growthcase}. Then there is a $W'\in \tilde{\mathcal{W}}$
  satisfying
    \begin{itemize} \item[(a)] $d(W, W') \leq C\tilde E_W(M_\tau)$, and
\item[(b)] $\tilde E_{W'}(M_{\tau+NB}) \leq
  \frac{1}{2}\tilde E_W(M_\tau)$, for the constant $B$ in Lemma~\ref{lem:3ann3}.
  \end{itemize}
\end{prop}
\begin{proof}
  We argue by contradiction. Suppose that we have a sequence
  $\epsilon_1^i\to 0$, and 
  $W_i\in \tilde{\mathcal{W}}$ such that 
  $d(W_i, W_0) < \epsilon_1^i$. In addition we have flows $M^i_\tau$
  such that $\tilde E_{W_i}(M_{\tau_i}^i) >
  e^{-\tau_i}$, and $\tau_i < T_1(\epsilon^i_1) - (\epsilon^i_0)^{-1}$. We
  are also assuming that $T_1(\epsilon_1^i) \leq T_\Theta(\epsilon^i_2)$ for the
  $\epsilon^i_2$ determined by $\epsilon^i_1$. We claim that if $N, C
  > 0$ are chosen sufficiently large, then the conclusions (a)
  or (b) will hold for all sufficiently large $i$.

  Let us write $d_i = \tilde E_{W_i}(M^i_{\tau_i})$. 
  Note first that using Corollary~\ref{cor:growthcase} we can assume
  that $\tilde E_{W_i}(M^i_{\tau_i + kB})\leq e^{kB\lambda} d_i$,
  as long as $\tau_i + kB < T_1(\epsilon_1^i) - 2B$, i.e. $k$ can be
  arbitrarily large as $i\to\infty$. 
  Because of this growth estimate 
  together with the assumption that $d_i\to 0$, we can write
  $M^i_{\tau_i+s}$ as the graph of a one-form $\eta_i(s)$ over a compact
  subset $K_i$ of $(W_i\setminus \{0\})\times [0,\infty)$, where as
  $i\to\infty$ the sets $K_i$ exhaust $(W_0\setminus \{0\})\times
  [0,\infty)$. Using Lemma~\ref{lem:graphlimit}, the sequence
  $d_i^{-1}\eta_i$ converges, along a subsequence, to a solution of
  the drift heat equation
  \[ \partial_s \eta = \Delta \eta + \frac{1}{2}(\eta - x\cdot \nabla \eta) \]
  on $(W_0\setminus \{0\}) \times [0,\infty)$, where the $\eta(s)$ are
  closed forms. Using the non-concentration estimate,
  Lemma~\ref{lem:nonconc2},
  we have that for
  all $s$, $|x|^\alpha \eta(s)$ is uniformly bounded in $B_1$, and in addition
  $\eta(s)$ is in the Gaussian $L^2$-space.

  We can decompose $\eta$ using the eigenfunction decomposition of
  $\mathcal{L}_0$ as in Lemma~\ref{lem:decomp1}. Let us write
  \[ \eta = \eta_0 + \eta_{<0} + \eta_{>0}, \]
   where $\eta_0$ is the static component, and $\eta_{<0}, \eta_{>0}$
   are the decaying and growing components respectively.

   First we claim that $\eta_{>0}=0$. To see this, note that if
   $\eta_{>0}\not=0$, then we have $\Vert \eta(s)\Vert \geq c
   e^{\lambda_2 s}$ for some $c > 0$ and all $s > 0$ (for the
   $\lambda_2$ in Lemma~\ref{lem:L23ann2}), but this contradicts the growth bound for
   $\tilde E_{W_i}(M^i_{\tau_i + s})$ that we found above.
   We can therefore write $\eta = \eta_0 + \eta_{<0}$. There are two
   cases depending on whether we are assuming (A3(a)), in which case
   the link of $W_0$ is connected, or (A3(b)), when there is no good
   blowup. 

   \bigskip
   \noindent{\bf Case 1.}
   Let us first 
   focus on the setting where the link of $W_0$ is connected. By
   Lemma~\ref{lem:static1}, we can write $\eta_0 = d(f + c|x|^2)$ for a
   degree 2 harmonic function $f$ on $W_0$ and a constant $c$. By the
   integrability assumption we can define the cones $W'_i\in\tilde{\mathcal{W}}$,
   which to leading order are given by perturbing $W_i$ by the one
   forms $d_i \eta_0$. Note that this implies $d(W_i, W_i') \leq
   Cd_i$. 

   Repeating the arguments above, we can write $M^i_{\tau_i + s}$ as the graphs of
   $\eta_i'(s)$ over larger and larger subsets of $(W_i'\setminus
   \{0\}) \times [0,\infty)$. By construction, along a subsequence we
   will have $d_i^{-1}\eta_i' \to \eta' = \eta_{<0}$. We have $\Vert
   \eta_{<0}(s)\Vert \leq C e^{-\lambda_2s}$, and from this we want to use
   Corollary~\ref{lem:nonconc2}, applied to $\tau_0=\tau_i+NB+s$ with $s\in
   [-3,0]$, to reach the contradiction
   \[ \label{eq:E20}
     \tilde E_{W_i'}(M^i_{\tau_i+NB}) \leq \frac{1}{2}\tilde E_{W_i}(M^i_{\tau_i}).
   \]
   The argument is somewhat similar to the proof of
   Lemma~\ref{lem:3ann3}. First, note that
   \[ \tilde E_{W_i'}(M^i_{\tau_i}) \leq Cd_i, \]
   and so by the growth property in Corollary~\ref{cor:growthcase} we
   have $\tilde E_{W_i'}(M^i_{\tau_i + kB}) \leq C e^{kB\lambda} d_i$ for all
   $k$, once $i$ is large. In particular it follows, as in the first
   part of the proof of Lemma~\ref{lem:3ann3}, that
   \[ E_{W_i'}(M^i_{\tau_i + NB + s}) \leq C_{NB} d_i, \text{ for } s\in
     [-10,10], \]
   where $C_{NB}$ depends on $NB$.

   Using that $\Vert \eta_{< 0}(s) \Vert\leq Ce^{-\lambda_2 s}$, we
   can argue as in the second part of the proof of
   Lemma~\ref{lem:3ann3} to show that Corollary~\ref{cor:nonconc}
   implies \eqref{eq:E20} for sufficiently large $i$, as required.

   \bigskip
   \noindent{\bf Case 2.} Suppose now that $W_0=\bigcup W_{0,j}$ according to the decomposition \eqref{eq:W0decomp} depending on the connected components of $M^i_{\tau_i}$. Up to choosing a subsequence we can assume that the decomposition does not depend on $i$. 
   We are also assuming (A3(b)), so there
   is no good blowup sequence along the flow. 
   We further
   decompose $\eta_0 = \eta_{00} + \eta_{01}$ as follows. First let us write $W_{0,j} = \bigcup_a V_{j,a}$, where each $V_{j,a}$ is a cone with connected link. 
   By Lemma~\ref{lem:static1}, on each $V_{j,a}$
   we can write $\eta_0 = d(f_{j,a} + c_{j,a}|x|^2)$ for a degree 2 harmonic
   function $f_{j,a}$ and a constant $c_{j,a}$. Let $\underline{c}_j$ be the
   average of the $c_{j,a}$, and define
   \[ \eta_{00} &= d(f_{j,a} + \underline{c}_j|x|^2)\, \text{ on }V_{j,a} \\
     \eta_{01} &= (c_{j,a}-\underline{c}_j)\, d|x|^2\, \text{ on }V_{j,a}. \]
   We then have $d^*\eta_{00} = 2n\underline{c}_j$ on $W_{0,j}$. Using the
   integrability assumption for $W_0$ we can define the cones $W_i'
   \in \tilde{\mathcal{W}}$ that to leading order correspond to perturbing
   $W_i$ by the one forms $d_i\eta_{00}$. Note that $d_i \eta_{01}$
   also corresponds to perturbing $W_i$, but the result will not be
   in $\tilde{\mathcal{W}}$, unless each $c_{j,a}=\underline{c}_j$, i.e. $\eta_{01}=0$. 

   As in Case 1, we can view the $M^i_{\tau_i+s}$ as graphs of one-forms
   $\eta_i'(s)$ over the $W_i'$ on larger and larger sets, and
   $d_i^{-1}\eta_i' \to \eta' = \eta_{01} + \eta_{<0}$. 
   We would like to argue like in Case 1, except we have the
   additional term $\eta_{01}$ that is not decaying. If $\Vert
   \eta_{01}\Vert < \kappa_1$ for a sufficiently small $\kappa_1$,
   then we can argue as in Case 1. We can therefore assume that $\Vert
   \eta_{01}\Vert \geq \kappa_1$ for a fixed small $\kappa_1 > 0$.
   In particular this means that on at least one component $W_{0, j_0}$ we have. 
   $d^*\eta_{01} > C^{-1}\kappa_1$ for some fixed $C > 0$.

   We can now argue somewhat similarly to \cite[Proposition
   7.3]{LSSz2} to get a contradiction. As in \cite{LSSz2} we  use the
   exactness of the flows (at least close to the origin), so that
   there are
   Lagrangian potentials $f_i(\tau)$ such that $e^{-s}(f_i - 2\theta_i)$
   satisfy the drift heat equation along $M^i_s$. While in
   \cite[Proposition 7.3]{LSSz2} we were only able to control the
   $f_i$ at certain times, in the current setting we will be able to
   show that the oscillation of $f_i$ can be made arbitrarily small
   for all sufficiently large times. This will imply that the
   oscillation of $\theta$ is also as small as we like, however this
   is bounded below by $\kappa_1$, leading to a contradiction. 

   As in Case 1, Corollary~\ref{cor:growthcase} implies that
   $\tilde E_{W_i'}(M^i_{\tau_i + kB}) \leq C e^{kB\lambda} d_i$, and so
   \[ \tilde E_{W_i'}(M^i_{\tau_i + NB + s}) \leq C_{NB} d_i, \text{ for } s\in
     [-10,10], \]
   for a constant $C_{NB}$ depending on $N$ (the constant $B$ is fixed). 
  We first use the nonconcentration estimate,
  Corollary~\ref{cor:nonconc},
  to improve this to a bound for $E_{W_i'}$ which is independent of the choice of $N$.

  For this note that we have the convergence $d_i^{-1}\eta_i' \to \eta'$, and $\eta'$
   has no growing component, so $\Vert \eta'(s)\Vert \leq C$ for all
   $s > 0$ for a uniform $C$. By Lemma~\ref{lem:pointwise2} we get
   pointwise estimates
   \[ \label{eq:eta'40} | \eta'(s)|^2(x) \leq C e^{|x|^2/4p}\, \text{ for } |x| >
     \rho_0/4, \]
   and
   \[ | d^*\eta'(s)|^2 \leq \begin{cases} C e^{|x|^2/4p} \max\{
       |x|^{-2},1\}, &\text{ if } n> 2, \\
       C e^{|x|^2/4p}, &\text{ if }n=2, \end{cases} \]
   for $s > 1$. 
   It follows that on any compact set $K \subset
   \mathbb{C}^n \setminus \{0\}$, for sufficiently large $i$, we have
   \[     |\theta - \theta_{W_i'}|^2 \leq \begin{cases} C e^{|x|^2/4p}
       d_i^2 \max\{|x|^{-2}, 1\}, &\text{ if }n>2, \\
       C e^{|x|^2/4p} d_i^2, &\text{ if } n=2, \end{cases}
     \]
   on $K\cap M^i_{\tau_i + s}$ for $s > 1$ and for some $p >
   1$. Together with \eqref{eq:eta'40}, Corollary~\ref{cor:nonconc} implies that 
   \[ \tilde E_{W_i'}(M^i_{\tau_i+NB+s}) \leq Cd_i, \text{ for }s\in [-9,10]. \]

   Applying Lemma~\ref{lem:nonconc2}, we have
   \[ d^2_{W_i'} \leq C_1 e^{|x|^2 / 4p_0} d_i^2\, \text{ on } M^i_{\tau_i + NB + s} \]
   and
  \[ \left[ \int_{M^i_{\tau_i + NB + s}} (d_{W_i'}^{2p_0} + |\theta -
       \theta_{W_i'}|^{2p_0}) e^{-|x|^2 / 4}\,
       d\mathcal{H}^n\right]^{1/p_0} \leq C_1 d_i^2, \]
  for $s\in [-8,10]$.
  (note that in the current setting, where we have
  no good blowups, the constant $\alpha=0$).
  In terms of the $p_0$ here, let $p_0' =
   (1+p_0)/2$, and for small $d > 0$ define $R_d$ by
   \[ e^{R_d^2 / 4p_0'} d^2 = 1. \]
   This definition implies that given $\epsilon > 0$, if $i$ is
   sufficiently large (depending on $\epsilon, C_1, p_0$), then we
   have
   \[ C_1 e^{|x|^2 / 4p_0} d_i^2 < \epsilon, \text{ for } |x| <
     R_{d_i}. \]

   Given $\gamma > 0$, we can assume that $N$ is sufficiently large so
   that
   \[ \Vert \eta_{<0}(NB + s)\Vert \leq \gamma, \text{ for }
     s\in [-10,10], \]
   which by Lemma~\ref{lem:pointwise2} implies bounds of the form
   \[ |\nabla^k \eta_{<0}(NB + s)|^2(x) \leq C_k \gamma^2 e^{|x|^2 /
       4p_k} \max\{|x|^{-2k},1\}, \text{ for } s\in [-9,10]. \]
   Note here that in the two-dimensional case there can be no $d\ln
   |x|$ terms in the decaying piece $\eta_{<0}$.  Since $d_i^{-1}\eta_i' \to \eta_{01} +
   \eta_{<0}$, and $d^*\eta_{01} > C^{-1}\kappa_1$ on a component $W_{0, j_0}$, it follows that if
   $\gamma$ is chosen sufficiently small (depending on $\kappa_1$), then on a component of the annular region
   $A_{2,1}\cap M^i_{\tau_i + NB+s}$ we have 
   \[ \label{eq:t50} |\theta - \theta_{W_i'}| >
     \frac{1}{2}C^{-1}\kappa_1 d_i, \text{
       for } s\in[-9,10]. \]

    Assumpion (A2) implies that we can find a potential $f_i$ such that
   $df_i=\lambda$ on $M^i_{\tau_i+NB+s}$ for $s\in [-10,10]$, at least on
   the ball $B_{e^{\tau_i/2}}$, by using the good graphicality over $W_0$
   provided by Lemma~\ref{lem:T1}. 
   Using $|\nabla f_i| = |x^\perp|$,
   and the fact that $\eta_{01}$ is homogeneous of degree 1, 
   this implies that on any given compact set $K\subset
   \mathbb{C}^n\setminus \{0\}$, once $i$ is
   sufficiently large, we will have
   \[ \label{eq:fi10} |\nabla f_i|^2 \leq C_1 \gamma^2 d_i^2 e^{|x|^2 /
       4p_1}, \text{ on } M^i_{\tau_i +NB + s} \cap K, \text{ for } s\in
     [-9,10]. \]
   
 In addition using Lemma~\ref{lem:T1}, we know that
   $M^i_{\tau_i+NB+s}$ is graphical over $W_i'$ at least on the annulus
   $A_{e^{\tau_i/2}, 1/10}$, and for large $i$ this contains the annulus
   $A_{R_{d_i}, 1/10}$. Indeed, we have $d_i > e^{-\tau_i}$, and so
   \[ e^{e^{\tau_i}/4p_0'} d_i^2 > e^{d_i^{-1}/4p_0'} d_i \to \infty, \]
   as $i\to \infty$. Viewing $M^i_{\tau_i + s}$ as the graphs of
   $\eta_i'$ over $W_i'$, using parabolic theory we obtain estimates
   analogous to those in Lemma~\ref{lem:pointwise2} for the $\eta_i'$
   on the annulus $A_{R_{d_i}-1, 1/5}$. In particular this implies
   that on this annulus we have
   \[\label{eq:fi11} |\nabla f_i|^2 \leq C d_i^2 e^{|x|^2 / 4p_1}, \]
   for some $p_1 > 1$, and $s\in [-8,10]$. Choosing $p_0$ smaller,
   we can arrange that $p_1 > p_0$.

   \bigskip
   \noindent {\bf Estimating $f_i$ on a small ball}.
   We claim that for sufficiently large $i$, and $s\in [-9,10]$, on each component of $M^i_{t_i+NB+s}\cap B_1$  we
   have the estimate
   \[ \mathrm{osc}_{M^i_{t_i+NB+s}\cap B_{r_\gamma}} f_i < \gamma
     d_i, \]
   where $r_\gamma\in (0,1)$ depends on $\gamma$. We write this argument
   separately since it is the only part of the proof that will need to
   be modified in Proposition~\ref{prop:decay5} below. 
   \begin{proof}
   We use the graphicality estimate from Proposition~\ref{prop:mainelliptic}
   part (b). In our current setting we know that $M^i_{\tau_i + NB + s}$
   is the graph of $\eta_i'(NB+s)$ over $W_i'$, and so it is a
   $d_i$-graph over $W_i'$ on $A_{1, \rho_0}$ for $s\in [-9,10]$ for
   large $i$, and moreover the estimate $|\theta-\theta_{W_i'}| <
   Cd_i$ on $B_2\cap M^i_{\tau_i+NB+s}$ follows from the monotonicity
   formula.  We obtain that for all $r\in (\delta^{-1}Cd_i, 1)$,
   $r^{-1}M^i_{\tau_i+NB+s}$ is a $Cd_ir^{-1}$-graph over $W_i'$ on
   $A_{1, 1/2}$, for $s\in [-9,10]$. It follows that
   \[ |\nabla f_i| \leq Cd_i, \text{ on } M^i_{\tau_i+NB+s} \cap A_{r,
       r/2}, \]
   for $r\in (\delta^{-1}Cd_i, 1)$. In addition, using that $df_i =
   \lambda$ on $M^i_{\tau_i+NB+s}$, we have
   \[ |\nabla f_i| \leq \delta^{-1}Cd_i\, \text{ on }M^i_{\tau_i+NB+s}
     \cap B_{\delta^{-1}Cd_i}. \]

   Consider a connected component $\Sigma=M^i_{\tau_i+NB+s, j}\cap B(0,2)$.
   The intrinsic diameter of $\Sigma\cap B_{\delta^{-1}Cd_i}$ is bounded
   by $\delta^{-1}C'd_i$ for a larger constant $C'$ (this uses the
   volume non-collapsing property implied by the bound on the
   Lagrangian angle, as in the proof of Lemma 7.2 in
   Neves~\cite{Neves1}). It follows from these 
   considerations that for any $\gamma > 0$ we can find an $r_\gamma >
   0$ such that
   \[ \mathrm{osc}_{\Sigma\cap B_{r_\gamma}} f_i <
     C''(\delta^{-2} d_i^2 + d_i r_\gamma)  < \gamma
     d_i, \]
   for sufficiently large $i$ and $s\in [-9,10]$. In other words we
   can find $\underline{f}_i$, depending on $i, s$, that are constant on each component of $M^i_{\tau_i+NB+s}\cap B(0,2)$, such that
   \[ |f_i - \underline{f}_i| \leq \gamma d_i\, \text{ on }
     M^i_{\tau_i+NB+s}\cap B_{r_\gamma}, \]
   for $s \in [-9,10]$ and large $i$.
   \end{proof}

   \bigskip
   We can extend the locally constant functions $\underline{f}_i$ to $M^i_{\tau_i + NB+s}\cap B(0, e^{\tau_i/2})$ using the good graphicality over $W_0$. Then,  first 
   using \eqref{eq:fi10} we can extend the above estimate to the region $K$, so that
   \[ \label{eq:fi40}|f_i - \underline{f}_i|^2 \leq C_1 \gamma^2 d_i^2 e^{|x|^2/4p_2},
     \text{ on } M^i_{\tau_i+NB+s}\cap (K \cup B_1), \]
   for $s\in [-9,10]$ and sufficiently large $i$, with some $p_2 >
   p_0$  (assuming that $K\cup
   B_{r_\gamma}$ contains the unit ball). Using the bound
   \eqref{eq:fi11}, we can further extend this to a weaker estimate
   (without the small factor $\gamma$), to the ball $B_{R_{d_i}-1}$:
   \[\label{eq:fi41} |f_i - \underline{f}_i|^2 \leq C_1  d_i^2 e^{|x|^2/4p_2},
     \text{ on } M^i_{\tau_i+NB+s}\cap B_{R_{d_i}-1}, \]
   for $s\in [-8,10]$ and $i$ sufficiently large.
   In a similar way we
   also have the analogous estimate for the Lagrangian angle: 
   \[ |\theta - \theta_{W_i'}|^2 \leq C_1  d_i^2 e^{|x|^2/4p_2},
     \text{ on } M^i_{\tau_i+NB+s}\cap B_{R_{d_i}-1}, \]
   for $s\in [-8,10]$ and $i$ sufficiently large.

   We now argue in the same way as in the proof of Proposition 7.3 in
   \cite{LSSz2}. We let $s^i_1=\tau_i + NB$ and $s^i_2 = s^i_1 + \sigma$
   with $\sigma > 0$ small, to be chosen below. Define the function
   \[ h = e^{s^i_1 - s}(f_i - \underline{f}_i - 2(\theta -
     \theta_{W_i'})) + 2(\theta - \theta_{W_i'}) \]
   on  $M^i_s\cap B_{R_{d_i}}$ for $s\in [s^i_1, s^i_2]$. Let $\chi_0:[0,\infty) \to \mathbb{R}$
   be a cutoff function such that $\chi_0(a) = 1$ for $a <
   (R_{d_i}-2)^2$, and $\chi_0(a) = 0$ for $a > (R_{d_i}-1)^2$, and so
   that $\chi_0, \chi_0', \chi_0''$ are bounded independently of
   $i$. We let $\chi(x) = \chi_0(|x|^2)$. Then $\chi^2h^2$ is globally
   defined. As in \cite[Proposition 7.3]{LSSz2} we can compute that
   \[ (\partial_s - \mathcal{L}_0) (\chi^2h^2) \leq \begin{cases} C|x|^2
       h^2\, &\text{ for } R_{d_i}-2 < |x| < R_{d_i}-1, \\
       0\, &\text{ otherwise }
       \end{cases}
     \]
     and we also have
     \[ h^2 \leq C d_i^2 e^{R_{d_i}^2/4p_2}, \text{ for } 1 < |x| <
         R_{d_i}. \]
       Define
       \[ \Theta = (e^{s^i_1-s}|x|^2 - e^{-\sigma}(R_{d_i}-2)^2)_+. \]
       We then have
       \[ (\partial_s - \mathcal{L}_0) \Theta \leq \begin{cases} -2\,
           &\text{ if } e^{s^i_1-s}|x|^2 > e^{-\sigma} (R_{d_i}-2)^2, \\
           0, \text{ otherwise.} \end{cases}. \]
       Note that we have $e^{s-s^i_1-\sigma} < 1$ for $s\in
       [s^i_1,s^i_2]$, so $(\partial_s - \mathcal{L}_0)\Theta \leq -2$
       if $|x| > R_{d_i}-2$. In particular for a suitable constant
       $C_2 > 0$ we have
       \[ \label{eq:h50} (\partial_s - \mathcal{L}_0) (\chi^2h^2 + C_2 d_i^2
         R_{d_i}^2 e^{R_{d_i}^2/ 4p_2}\Theta) \leq 0, \]
       for $s\in [s^i_1, s^i_2]$.

       At $s=s^i_1$ we have $\chi^2h^2 =
       \chi^2(f_i-\underline{f}_i)^2$. From the estimates
       \eqref{eq:fi40}, \eqref{eq:fi41},  we have that if the compact
       set $K$ is
       chosen sufficiently large (depending on $\gamma$), then 
       \[ \label{eq:h51} \int_{M^i_{s}} \chi^2(f_i - \underline{f}_i)^2\,
           e^{-|x|^2/4}\, d\mathcal{H}^n \leq C \gamma^2 d_i^2,
           \text{ for }s\in [\tau_i+NB-7, \tau_i+NB+10]. \]
       We can also estimate the Gaussian integral of $\Theta$ at
       $s=s^i_1$ in the
       same way as in \cite{LSSz2}: note that for $s=s^i_1$ we have
       $\Theta(x) =0$ if $|x| < e^{-\sigma/2}(R_{d_i}-2)$, which
       is satisfied if $|x| < e^{-\sigma}R_{d_i}$ for large $i$. At
       the same time $\Theta(x) \leq |x|^2$ for all $|x|$. It follows that
       \[ \int_{M^i_{s^i_1}} \Theta\, e^{-|x|^2/4}\, d\mathcal{H}^n
         \leq Ce^{-2\sigma} R_{d_i}^2 \exp(-e^{-2\sigma} R_{d_i}^2/4), \]
       and so
       \[ \int_{M^i_{s^i_1}}  C_2 d_i^2
         R_{d_i}^2 e^{R_{d_i}^2/ 4p_2}\Theta\, e^{-|x|^2/4}\,
         d\mathcal{H}^n \leq C d_i^2 e^{-2\sigma} R_{d_i}^4
         \exp\left(\frac{R_{d_i}^2}{4p_2} -
           \frac{R_{d_i}^2}{4e^{2\sigma}}\right). \]
       As long as $\sigma$ is chosen sufficiently small so that
       $e^{2\sigma} < p_2$, then for sufficiently large $i$ we will
       have
       \[ \label{eq:h52} \int_{M^i_{s^i_1}}  C_2 d_i^2
         R_{d_i}^2 e^{R_{d_i}^2/ 4p_2}\Theta\, e^{-|x|^2/4}\,
         d\mathcal{H}^n \leq \gamma^2 d_i^2. \]
       Using \eqref{eq:h50} in the monotonicity formula, together with
     the estimates \eqref{eq:h51} and \eqref{eq:h52}, and that $\Theta
     \geq 0$, we find that at
     $s=s^i_2$ we have
     \[ \int_{M^i_{s^i_2}} \chi^2h^2\, e^{-|x|^2/4}\,d\mathcal{H}^n
       \leq C\gamma^2 d_i^2. \]
     Using \eqref{eq:h51} at $s=s^i_2$ we find that
     \[ (2 - 2e^{\sigma})^2 \int_{M^i_{s^i_2}}\chi^2 ( \theta -
         \theta_{W_i'})^2\, e^{-|x|^2/4}\, d\mathcal{H}^n \leq C
         \gamma^2 d_i^2, \]
       so
       \[ \label{eq:thetaint40} \int_{M^i_{s^i_2}\cap B_{R_{d_i}-2}} ( \theta -
         \theta_{W_i'})^2\, e^{-|x|^2/4}\, d\mathcal{H}^n \leq C
         \gamma^2 d_i^2, \]
       for a larger $C$ (depending on $\sigma$). At the same time the
       uniform bound on the Lagrangian angle and area ratios implies
       that for any $c > 1$ we have
       \[ \label{eq:thetaint41} \int_{M^i_{s^i_2}\setminus B_{R_{d_i}-2}} ( \theta -
         \theta_{W_i'})^2\, e^{-|x|^2/4}\, d\mathcal{H}^n \leq 
         e^{-(R_{d_i}-2)^2/4c}, \]
       for sufficiently large $i$. By definition we have
       $e^{-R_{d_i}^2/4p_0'} = d_i^2$ for some $p_0' > 1$, so
       \eqref{eq:thetaint40} and \eqref{eq:thetaint41} imply
       that if $i$ is sufficiently large, we have
       \[ \int_{M^i_{s^i_2}} ( \theta -
         \theta_{W_i'})^2\, e^{-|x|^2/4}\, d\mathcal{H}^n \leq C
         \gamma^2 d_i^2. \]
       Using the monotonicity formula we find that at the later time $s^i_1
       + 1 = \tau_i + NB+1$ we have the pointwise bound
       \[ |\theta - \theta_{W_i'}| \leq C\gamma d_i\, \text{ on }
         M^i_{\tau_i+NB+1} \cap B_2. \]
       If $\gamma$ is chosen sufficiently small, then this contradicts
       the estimate \eqref{eq:t50} for the Lagrangian angle on the
       annulus $A_{2,1}$. The contradiction shows that for large $N$
       we have the decay required in the statement of the
       Proposition. 
     \end{proof}

Next we prove a decay estimate in the setting where
Condition (A3(c)) holds, i.e. in the two dimensional case, without
assuming that there is no good blowup. The argument is more involved,
using additional ideas from \cite{LSSz2}. Similarly to \cite{LSSz2} we want to define
the excess of $M\subset \mathbb{C}^2$, relative to a union of planes
$W$, however now, as above, the planes might not all have the same angle. For a rescaled flow $M_\tau$ in $\mathbb{C}^2$ satisfying Conditions (A), and for $\tau\in [0, T_1(c_0))$ for small $c_0 > 0$, we define 
\[ \mathcal{A}(M_\tau) = \Theta(M_\tau) - \Theta(W_0) + \inf_W \int_{M_\tau\cap B(0, e^{\tau/2})} |\theta - \theta_W|^2\, e^{-|x|^2/4}\, d\mathcal{H}^2. \]
Here the infimum is over all $W\in \tilde{\mathcal{W}}$ as in \eqref{eq:thetaWdefn}.

We need the following slight modification of Proposition 4.6 from
\cite{LSSz2}.
\begin{prop}\label{prop:Abound}
  Assume Conditions (A), with (A3(c)) satisfied. There are
  $\epsilon_0 >0, \beta_0 > 0$ depending on $W_0$, satisfying the
  following. Suppose that $d < \epsilon_0$ and $\tau < T_1(\epsilon_0) -
  10$ with $e^{-\tau} < d$. If $\tilde{E}_W(M_\tau) < d$ for some $W\in \tilde{\mathcal{W}}$ with
  $d(W, W_0) < \epsilon_0$, then we have
  \[ |\mathcal{A}(M_\tau)| \leq d^{1 + \beta_0}. \]
\end{prop}
\begin{proof}
  The proof is essentially the same as that of Proposition 4.6 in
  \cite{LSSz2}. As above (and also as in \cite{LSSz2}) we define $R_d$
  by
  \[ e^{R_d^2/4p_0} d^2 = 1, \]
  for some $p_0 > 1$ that is very close to 1. From
  Lemma~\ref{lem:nonconc2} we have
  \[ d_W^2 \leq C e^{|x|^2/4p} d^2\max\{ |x|^{-2\alpha}, 1\}, \]
  where recall that $\alpha = 1 + \lambda_2$ for $\lambda_2$ very
  close to 0, to be chosen. Note that by Lemma~\ref{lem:T1} the
  surface $M_\tau$ has good graphicality on $A_{e^{\tau/2}, 1/10}$, and as in
  the proof of Proposition~\ref{prop:decay4}, this contains the
  annulus $A_{R_d,1/10}$ for small $d$. Using parabolic estimates the
  distance bound above implies that on $A_{R_d-1, 1/5}$, if $M_\tau$ is the
  graph of the one-form $\eta$ over $W$, then we have
  \[ |\eta|, |\nabla \eta| \leq C e^{|x|^2/4p} d^2, \]
  perhaps for a smaller $p > 1$. We choose $p_0$ above so that $1 <
  p_0 < p$. As for inside the unit ball $B_1$ we use
  Proposition~\ref{prop:mainelliptic}, part (a). This implies that
  $M_\tau$ has good graphicality over $W$ for $|x| \in (\delta^{-1}
  d^{(2+\lambda_2)^{-1}}, 1)$. Note that in Proposition 4.6 of
    \cite{LSSz2} we had a slightly larger graphicality region, where
    $|x|\in (\delta^{-1} d^{1/2}, 1)$. Still, the same proof applies
    to estimate $\mathcal{A}(M_\tau)$. 
\end{proof}

For any $\beta > 0$ we define $\mathcal{A}_\beta(M) =
|\mathcal{A}(M)|^{\beta-1}\mathcal{A}(M)$. 
Using the monotonicity formula, in a similar way to the proof of Lemma~\ref{lem:nonconc2}  we have the following.
\begin{lemma}\label{lem:AMtaumonotone}
    There is a constant $C > 0$ depending on the area ratio and Lagrangian angle bounds, such that if $\tau_0 < \tau_1 < T_1(c_0)$, then
    \[ \mathcal{A}(M_{\tau_0}) - \mathcal{A}(M_{\tau_1}) \geq 
    \int_{\tau_0}^{\tau_1} \int_{M_\tau\cap B(0,2)} \left( \frac{1}{4\pi}\left|H + \frac{x^\perp}{2}\right|^2 + 2|H|^2 \right)\, e^{-|x|^2/4}\, d\mathcal{H}^2\, d\tau - Ce^{-\tau_0}.\]
\end{lemma}
\begin{proof}
 The Gaussian area $\Theta(M_\tau)$ satisfies
 \[ \frac{d}{d \tau}  \int_{M_\tau} e^{-|x|^2/4}\, d\mathcal{H}^2 = -\int_{M_\tau} \left| H + \frac{x^\perp}{2}\right|^2\, e^{-|x|^2/4}\, d\mathcal{H}^2, \] 
 by the monotonicity formula. For the term involving $\theta-\theta_W$, we argue as in the proof of Lemma~\ref{lem:nonconc2}, using the function $\tilde\theta$ defined in \eqref{eq:tildetheta}. We have
 \[ \int_{M_{\tau_0}} \tilde{\theta}^2 e^{-|x|^2/4}\, d\mathcal{H}^2 \leq \int_{M_{\tau_0}\cap B(0,e^{\tau/2})} |\theta - \theta_W|^2 e^{-|x|^2/4}\, d\mathcal{H}^2 + Ce^{-\tau_0}.\]
 Since $\tilde{\theta}$ is a subsolution of the drift heat equation, the monotonicity formula implies
 \[ \int_{M_{\tau_1}} \tilde{\theta}^2 e^{-|x|^2/4}\, d\mathcal{H}^2 &\leq \int_{M_{\tau_0}} \tilde{\theta}^2 e^{-|x|^2/4}\, d\mathcal{H}^2 - \int_{\tau_0}^{\tau_1} \int_{M_\tau} 2|\nabla\tilde\theta|^2\, e^{-|x|^2/4}\, d\mathcal{H}^2\, d\tau \\
 &\leq \int_{M_{\tau_0}} \tilde{\theta}^2 e^{-|x|^2/4}\, d\mathcal{H}^2 - \int_{\tau_0}^{\tau_1} \int_{M_\tau\cap B(0,e^{\tau_0/4})} 2|H|^2\, e^{-|x|^2/4}\, d\mathcal{H}^2\, d\tau. \]
 Combining these, we obtain the required result. 
\end{proof}

We have the following
version of Proposition~\ref{prop:decay4} in the setting of Condition
(A3(c)), analogous to Proposition 7.3 in \cite{LSSz2}.
\begin{prop}\label{prop:decay5}
  Assume Conditions (A), with (A3(c)) satisfied, so $W_0$ is a union
  of planes in $\mathbb{C}^2$. Suppose also that $\lambda_2$ in
Lemma~\ref{lem:L23ann1} is chosen sufficiently small. There are
  $\epsilon_0,N, C > 0$ and $\beta\in (0,1)$ depending on $W_0$, satisfying the
  following. Suppose that $\tilde{E}_W(M_\tau) > e^{-\tau}$, for $W\in \tilde{\mathcal{W}}$, where $d(W,W_0)
  < \epsilon_1 < \epsilon_0$, and $\tau < T_1(\epsilon_1) -
  \epsilon_0^{-1}$. Suppose that $T_1(\epsilon_1) \leq
  T_\Theta(\epsilon_2)$ for the $\epsilon_2$ determined by
  $\epsilon_1$ in Corollary~\ref{cor:growthcase}. Then there is a
  $W'\in \tilde{\mathcal{W}}$ satisfying $d(W,W')\leq C\tilde{E}_W(M_\tau)$
  together with one of the following properties:
  \begin{itemize}
  \item[(a)]  $\tilde{E}_{W'}(M_{t \tau+NB}) \leq
    \frac{1}{2}\tilde{E}_W(M_\tau)$, for $B$ determined in
    Lemma~\ref{lem:3ann3}
    \item[(b)] $\tilde{E}_{W'}(M_{\tau+NB}) \leq
      \mathcal{A}_\beta(M_{\tau+NB-3}) - \mathcal{A}_\beta(M_{\tau+NB})$.
    \end{itemize}
  \end{prop}
  \begin{proof}
    The proof closely follows the argument in the setting of Condition
    (A3(b)) in the proof of Proposition~\ref{prop:decay4}. The main
    difference is that in the current setting the distance 
    $d = \tilde{E}_W(M_\tau)$ does not imply a sufficiently good
    graphicality estimate for $M_\tau$ over $W$ near the origin. Indeed
    instead of the estimate $d_W \leq d$ for the distance for $W$, we
    only have $r^\alpha d_W \leq d$, where $\alpha = 1 +
    \lambda_2$ for $\lambda_2 > 0$ as close to 0 as we like.
    This is also slightly worse than the estimate $r d_W \leq d$ that
    was used in the proof of the
    analogous Proposition 7.3 in \cite{LSSz2}. That estimate was used
    in a sharp way in \cite{LSSz2}, so we need to use
    a different approach to bound the oscillation of the Lagrangian
    potential $f$ near the origin.

    The setup of the argument is exactly the same as in the proof of
    Proposition~\ref{prop:decay4}, so we will use the same notation. 
    The only part that needs to be changed is the estimate of $f_i$ on
    a small ball. The beginning of the argument is the same as in
    \cite{LSSz2}. We assume 
     that (b) in the statement of the proposition fails, so
    \[ \tilde{E}_{W'}(M^i_{\tau_i+NB}) \geq
      \mathcal{A}_\beta(M^i_{\tau_i+NB-3}) -
      \mathcal{A}_\beta(M^i_{\tau_i+NB}). \]
    Using Lemma~\ref{lem:AMtaumonotone} this implies
    \[ \int_{-3}^0 |\mathcal{A}(M^i_{\tau_i+NB+s})|^{\beta-1}
      \int_{M^i_{\tau_i+NB+s}\cap B(0,2)} (|H|^2 + |x^\perp|^2) e^{-|x|^2/4}\,
      d\mathcal{H}^2\,ds \leq \beta^{-1} Cd_i + Ce^{-\tau_i},\]
    and recall that $e^{-\tau_i} < d_i$ by assumption. 
    Using Proposition~\ref{prop:Abound} we have
    $|\mathcal{A}(M^i_{\tau_i+NB+s})|\leq d_i^{1+\beta_0}$ for $s\in
    [-3,0]$, and so we obtain
    \[ \int_{-3}^0 \int_{M^i_{\tau_i+NB+s}\cap B(0,2)} (|H|^2 + |x^\perp|^2) e^{-|x|^2/4}\,
      d\mathcal{H}^2\,ds \leq C_\beta d_i^{1+(1-\beta)(1+\beta_0)}. \]
    If $\beta$ is chosen sufficiently small, then we obtain
    \[ \int_{-3}^0 \int_{M^i_{\tau_i+NB+s}\cap B(0,2)}( |H|^2 + |x^\perp|^2)\,
      d\mathcal{H}^2\,ds \leq d_i^{2+2\beta_1}, \]
    for some $\beta_1 > 0$.

    For a small $\sigma > 0$ to be chosen later, we can find $s^i_1,
    s^i_2 \in [-2,-1]$ such that $\sigma/2 < s^i_2-s^i_1 < \sigma$,
    and for sufficiently large $i$ we have
    \[ \label{eq:xperpupper} \int_{M^i_{\tau_i + NB + s^i_j}\cap B(0,2)} (|H|^2 + |x^\perp|^2)\,
      d\mathcal{H}^2 \leq d_i^{2+\beta_1}, \text{ for } j=1,2.\]
    The argument for bounding the potentials $f_i$ on $B_1\cap
    M^i_{\tau_i + NB + s^i_j}$ now needs some other ingredients than what
    was used in  \cite{LSSz2}.

    To simplify the notation, let us fix a large $i$ and $j=1$ or $2$,
    and write $L =  M^i_{\tau_i
      +NB + s^i_j}$.  We apply Proposition~\ref{prop:mainelliptic}
    part (c) to $L$. If the first alternative in (c) holds, then we
    are in the same setting as if there were no good blowups, and we
    can estimate $f_i$ on $L \cap B_2$ exactly as in the proof of
    Proposition~\ref{prop:decay4}. Suppose that the second alternative
    holds in part (c) of Proposition~\ref{prop:mainelliptic}. We then
    have some $r_0 > 0$ such that the rescaling $r_0^{-1}L$ satisfies
    the estimate
    \[ \int_{r_0^{-1}L\cap A_{2,1}} |x^\perp|^2\, d\mathcal{H}^2 >
      c_1. \]
    Scaling back down we find that
    \[ \int_{L \cap A_{2r_0, r_0}} |x^\perp|^2\, d\mathcal{H}^2 > c_1
      r_0^4. \]
    At the same time we have the estimate \eqref{eq:xperpupper}, so
    $c_1 r_0^4 < d_i^{2+\beta_1}$ if $i$ is sufficiently large. This
    gives the upper bound
    \[ r_0 < C d_i^{\frac{1}{2} + \frac{\beta_1}{4}} \]
    for the scale at which $L$ stops being close to the cone
    $W'$. Indeed Proposition~\ref{prop:mainelliptic} says that for
    $r\in (\delta^{-1}r_0, 1)$ the rescaling $r^{-1}L$ is a $c_0$-graph
    over $W'$. If $i$ is sufficiently large, then $\delta^{-1}r_0 <
    d_i^{1/2}/2$, so $L$ has good graphicality over $W'$ on the
    annulus $A_{1, d_i^{1/2}/2}$. 

    At this point we can argue similarly to \cite[Proposition
    7.3]{LSSz2}. First we can apply \cite[Claim 7.6]{LSSz2} to bound
    the oscillation of $f_i$ on $L\cap 
    B_{d_i^{1/2}}$.  This shows that for any $\gamma > 0$ there is a constant
    $\underline{f_i}$ such that $|f_i - \underline{f_i}| < \gamma d_i$
    on $L \cap B_{d_i^{1/2}}$ for sufficiently large $i$. On the annulus 
    $B_1 \setminus B_{d_i^{1/2}/2}$ we can use the good graphicality
    of $L$ over $W'$, as in \cite[Claim 7.7]{LSSz2}, to show that for
    any $\gamma > 0$
    \[ \int_{L\cap B_1} |f_i - \underline{f_i}|^2\, d\mathcal{H}^2
      \leq \gamma^2 d_i^2, \]
    once $i$ is sufficiently large. 

    Using \eqref{eq:fi11}, we find that for sufficiently large $i$ we
    can assume that $|f_i - \underline{f_i}| \leq \gamma d_i$ on the
    annulus $B_2\setminus B_{1/2}$. At this point the argument is
    exactly the same as in the proof of
    Proposition~\ref{prop:decay4}.
  \end{proof}

We can now prove Proposition~\ref{prop:T1Ttheta}.
\begin{proof}[Proof of Proposition~\ref{prop:T1Ttheta}]
  We will give the proof in the setting of Assumption (A3(c)), applying
  Proposition~\ref{prop:decay5}. The proof in the cases (A3(a)) and
  (A3(b)) are essentially the same, except simpler since we do not
  have to deal with the situation where the alternative (b) holds in
  Proposition~\ref{prop:decay5}. The argument is quite similar to the
  proof of \cite[Proposition 7.12]{LSSz2}. 
  
  Let $\epsilon > 0$. We can assume that $\epsilon < \epsilon_0$ in
  Proposition~\ref{prop:decay5}, so we can apply that Proposition with
  $\epsilon_1=\epsilon$, which determines $\epsilon_2$.
  We will choose $\delta < \epsilon_2$ below. If $T_1(\epsilon_1) \geq
  T_\Theta(\delta)$, then we are done, so we can assume
  $T_1(\epsilon_1) < T_\Theta(\delta)$. Since $T_\Theta(\delta) \leq
  T_\Theta(\epsilon_2)$, this means that we can apply 
  Proposition~\ref{prop:decay5}.

  We define $W_i\in \tilde{\mathcal{W}}$ inductively as follows. We let
  $\delta_1 < \epsilon_1$, to be chosen below, and set $k_0 > 0$ to be
  the largest integer such that
  $k_0NB < T_1(\delta_1)$.
  For all $i < k_0$ we set $W_i=W_0$. After that we define $W_{i+1}$
  as follows:
  \begin{itemize}
  \item if $\tilde E_{W_i}(M_{iNB}) > e^{-iNB}$, then $W_{i+1}$ is the $W'$
    given by Proposition~\ref{prop:decay5}.
  \item otherwise we set $W_{i+1} = W_i$.
  \end{itemize}
  We can do this as long as we have $iNB \leq T_1(\epsilon_1)-\epsilon_0^{-1}$,
  and also $d(W_i, W_0) < \epsilon_1$. Write $i_{max}$ for the largest
  $i$ satisfying these two conditions. 
  
  Set $d_i = \tilde E_{W_i}(M_{iNB})$, so we have $d(W_i, W_{i+1}) \leq
  Cd_i$. For $k_0 \leq i \leq i_{max}$,  Proposition~\ref{prop:decay5}
  implies that either $d_i \leq e^{-iNB}$, or one of the two
  alternatives (a) or (b) hold. Note that if $d_i \leq e^{-iNB}$ then
  Corollary~\ref{cor:growthcase} also implies $d_{i+1}\leq e^{\lambda
    NB}e^{-iNB}$. In sum, we have 
  \[ \label{eq:drec} d_{i+1} &\leq \frac{1}{2}d_i + e^{(\lambda - i)NB} + \mathcal{A}_\beta(M_{(i+1)NB-3}) -
    \mathcal{A}_\beta(M_{(i+1)NB}) \\
    &= \frac{1}{2}d_i + a_{i+1},\]
  for suitable numbers $a_{i+1}$.
  Since we have
  \[ \sum_{i=k_0+1}^{i_{max}} a_i \leq Ce^{-k_0 NB} +
    \mathcal{A}_\beta(M_0) - \mathcal{A}_\beta(M_{i_{max}NB}), \]
  it follows from the recursion above that for all $i\leq i_{max}$ we have
  \[ d(W_i, W_0) \leq \sum_{i=k_0}^{i_{max}} d_i \leq 2d_{k_0} + C e^{-k_0 NB} +
    2(\mathcal{A}_\beta(M_0) - \mathcal{A}_\beta(M_{i_{max}NB})). \]

  By definition $M_{iB}$ is still a
  $\delta_1^2$-graph over $W_0$ on $A_{\delta_1^{-1},\delta_1}$ for $i
  < k_0$, so by choosing $\delta_1$ sufficiently small, we can ensure
  that $d_{k_0} < \epsilon_1/3$.
  By choosing $\delta$ sufficiently small, depending on the choice of
  $\delta_1$  (so that $M_0$ is a $\delta^2$-graph
  over $W_0$ on $A_{\delta^{-1},\delta}$), we can ensure that $k_0$ is
  large, so that $Ce^{-k_0NB} < \epsilon_1/3$.  
  At the same time since $i_{max}NB \leq T_1(\epsilon_1) <
  T_{\Theta}(\delta)$, we can also bound the change in the excess
  $\mathcal{A}_\beta$ in terms of $\delta$. 
  Therefore, by
  choosing $\delta_1$ small first and then $\delta$ small depending on
  $\delta_1$, we can ensure that
  $d(W_i, W_0) < \epsilon_1$ for all $i \leq i_{max}$. In particular the $W_i$ will be
  defined as long as $iNB \leq T_1(\epsilon_1) - \epsilon_0^{-1}$.

  The inequalities \eqref{eq:drec} also imply that if $i \leq i_{max}$, then 
  \[ d_i \leq d_{k_0} + Ce^{-k_0NB} + \mathcal{A}_\beta(M_0) -
    \mathcal{A}_\beta(M_{i_{max}NB}). \]
  By choosing $\delta_1, \delta$ even smaller, we can ensure that
  $d_{i_{max}}$ is as small as we like. Therefore for suitably small
  $\delta_1, \delta$, 
  Lemma~\ref{lem:ext1} implies that $T_1(\epsilon_1) > (i_{max}+1)NB +
  \epsilon_0^{-1}$.  This contradicts our choice of $i_{max}$, so
  we must have $T_1(\epsilon_1) \geq T_\Theta(\delta)$, as required.
\end{proof}

\section{Characterisation of the Lawlor necks}\label{section:characterisationLawlor}

The goal of this section is to extend the characterisation of Lawlor necks by Joyce-Imagi-Santos \cite{JoyceImagi}, to the setting allowing for \emph{a priori singular} special Lagrangians arising as current topology limits of smooth embedded exact Lagrangians with small Lagrangian angle. The main result is Theorem \ref{singularuniqueness}.

A brief outline is as follows. We begin with a review of the proof strategy of  \cite{JoyceImagi} in section \ref{LawlorneckJoyceImagi}, and expand on Floer theoretic ingredients in section \ref{holocurves}. The main theorem is stated in section \ref{Singularuniquenesssection} with discussions on the new difficulties. After some more analytic preparations in section \ref{Monotonictysection}, \ref{UniformcontinuityLagpotentialsection}, we embark on the key construction of a judicious perturbation of the Lawlor neck model in section \ref{PerturbationofLawlor}, by solving an auxiliary Poisson equation with precise estimates, crucially ensuring some special properties on the Lagrangian potential at the degree $0,n$ intersections. We then follow \cite{JoyceImagi} to compactify the Lagrangians inside a partial completion of $\C^n$ into a Liouville manifold, where we can appeal to the classification results in the derived Fukaya category. A robust application of holomorphic curve ingredients, together with some minimal surface generalities, finishes off the proof in section \ref{Assemblingpieces}.

\subsection{Lawlor neck and Joyce-Imagi-Santos uniqueness}\label{LawlorneckJoyceImagi}

We briefly sketch the seminal work of
 Joyce-Imagi-Santos \cite{JoyceImagi} on the uniqueness of Lawlor necks, where holomorphic curves and the algebraic structures of the Fukaya category make an interesting appearance. The symplectic backgrounds are well summarized in \cite[section 2]{JoyceImagi}, and we will follow their conventions.

\subsubsection*{Lawlor necks}

We use the standard Euclidean structure on $\C^n$, with
\[
\omega= \frac{\sqrt{-1}}{2} \sum dz_i\wedge d\bar{z}_i, \quad \Omega= dz_1\wedge\ldots dz_n, \quad \lambda= \frac{1}{2}\sum (\text{Re} z_j d\text{Im} z_j- \text{Im} z_j d\text{Re} z_j).
\]
We denote the Lagrangian potential of an exact Lagrangian $L$ as $f_L$, namely $df_L=\lambda|_L$.

The Lawlor necks \cite{Lawlor}\cite{JoyceLeeTsui} are non-compact embedded exact special Lagrangians $L_{\phi,A}$ inside $\C^n$, asymptotic at infinity to the union of two planes
\[
\Pi_0= \R^n, \quad \Pi_\phi=(e^{i\phi_1}, \ldots e^{i\phi_n})\R^n, \quad 0<\phi_i<\pi, \quad \sum \phi_i=\pi. 
\]
Symplectic topologically, they can be viewed as a realisation of the Lagrangian handle $S^{n-1}\times \R$ that appears in the Lagrangian connected sum construction. This motivates the ansatz 
\begin{equation}
L_{\phi,A}= \{   (z_1(y)x_1, . . . , z_n(y)x_n) : y \in \R, x_k \in \R, x^2_1
+ \ldots+ x_n^2 =1      \}.
\end{equation}
The special Lagrangian condition translates into an ODE system on the functions $z_1(y),\ldots z_n(y)$, which can be solved exactly as follows.

Let $n > 2$ and $a_1, . . . , a_n > 0$, and define polynomials $p, P$ by
\[
p(x) = (1 + a_1x^2) \ldots (1 + a_nx^2) -1 , \quad  P(x) =
\frac{p(x)}{x^2}.
\]
Define real numbers $\phi_1, . . . , \phi_n$ and $A$ by
\[
\phi_k = a_k
\int_{-\infty}^{\infty}
\frac{dx}{
	(1 + a_kx^2)\sqrt{P(x)}}
, \quad
A =
\int_{-\infty}^{\infty}
\frac{dx}{
	2\sqrt{P(x)}}
\]
Clearly $\phi_k,A > 0$, and elementary integration shows $\sum \phi_i=\pi$. This yields a 1-1 correspondence between $n$-tuples $(a_1,\ldots, a_n)$ with $a_k > 0$,
and $(n+1)$-tuples $(\phi_1, \ldots , \phi_n,A)$ with $\phi_k\in (0,\pi)$, $\sum \phi_k=\pi$ and $A>0$. Setting
\[
z_k(y)= e^{i\psi_k(y)}
\sqrt{a_k^{-1}+y^2 },  \quad \text{where } \psi_k(y) = a_k
\int_{-\infty}^y \frac{dx}{
	(1 + a_kx^2)\sqrt{P(x)}},
\]
yields the solution $(z_1(y),\ldots, z_n(y))$, hence the \textbf{Lawlor necks} $L_{\phi,A}$.

For fixed asymptotic planes $\Pi_0, \Pi_\phi$, the Lawlor necks arise in a 1-parameter family, related to each other by the rescaling $\vec{z}\to \lambda\vec{z}$ in $\C^n$, and $A$ behaves like 2-dimensional area $A\to \lambda^2A$ under this scaling. One also observes that asymptotically near infinity, the Lawlor necks are graphs over $\Pi_0$ (resp. $\Pi_\phi$) of the differential $df$, where
\[
|f|=O( |\vec{x}|^{2-n }), \quad |\partial^k f|=O( |\vec{x}|^{2-n-k}  ).
\]
We say the Lawlor neck has asymptotic decay rate $2-n$. The upshot is that it approaches $\Pi_0\cup \Pi_\phi$ sufficiently fast.

\subsubsection*{Joyce-Imagi-Santos uniqueness theorem}

\begin{thm}
	\cite{JoyceImagi} Let $n\geq 3$.  Assume $L$ is a smooth embedded exact special Lagrangian of phase zero, asymptotic at rate $<0$ to the union of the two planes $\Pi_0\cup \Pi_\phi$ with $\sum \phi_i=\pi$, then $L$ is  $L_{\phi,A}$ for some $A>0$.

\end{thm}

Here is a sketch of their arguments:
\begin{enumerate}
	\item  Using the asymptotic  assumption on the exact Lagrangian $L$, one can assign an analytic invariant $A(L)$ to $L$ as follows. Let $f_L:L\to \R$ be a primitive of the Liouville form $\lambda$, namely $df_L=\lambda|_L$, then $f_L$ converges to constants $c_0, c_\phi$ at the two asymptotic ends along $\Pi_0, \Pi_\phi$ respectively. Then one defines $A(L)=c_\phi-c_0$. If $L$ coincides with the Lawlor neck $L_{\phi,A}$, then $A(L)=A$.

	\item  Partially compactify $\C^n$ into a Liouville manifold identified as the plumbing $M$ of two cotangent bundles $T^*S^n$ with $T^*S^n$. Here the two copies of $S^n$ arise topologically as one-point compactifications of $\Pi_0$ and $\Pi_\phi$ by adding the points at infinity $\infty_0$ and $\infty_\phi$, and topologically $M$ is the union of $\C^n$ and the two cotangent fibres over $\infty_0$ and $\infty_\phi$ respectively. Under suitably fast decay condition at infinity, the unknown special Lagrangian $L$ can be compactified into an exact graded embedded Lagrangian $\bar{L}$ inside $M$. (More details will be recalled in section \ref{partialcompactification}). One would like to compare this to the Lagrangian $\bar{L}_{\phi,A}$ obtained by the compactification of the standard Lawlor necks $L_{\phi,A}$ inside $M$.
	
	\item  
	By analyzing the intersection pattern with the two cotangent fibres at infinity, and using the classification results of Abouzaid and Smith \cite{AbouzaidSmith}, one shows that inside $D^bFuk(M)$, the Lagrangian object $\bar{L}$ is isomorphic to one of the two Lagrangian connected sums of the two $S^n$ with suitable gradings, and in fact the assumption on Floer degrees $\sum \phi_i=\pi$ singles out $\bar{L}\simeq \bar{L}_{\phi, A}\in D^bFuk(M)$, the opposite surgery  corresponding to $\sum \phi_i=(n-1)\pi$.
	This step needs $n\geq 3$. 
	For contradiction, we assume $\bar{L}$ does not coincide with $\bar{L}_{\phi,A}$ for any choice of parameter $A>0$.

	\item
	By a modification of the Thomas-Yau argument, one shows that after a small Hamiltation perturbation $\bar{L}''$ of $\bar{L}_{\phi,A}$, we can ensure $\bar{L}''$ is transverse to $\bar{L}$, there is no degree $0, n$ intersection points in $\bar{L}''\cap \bar{L}$ inside $\C^n$, and there is precisely one intersection point $p$ and $q$ in $\bar{L}''\cap \bar{L}$ on each of the two cotangent fibres at infinity respectively. Moreover, the $D^b Fuk(M)$ class and the analytic invariants of $\bar{L}''$ agree with that of $\bar{L}_{\phi,A}$.

	\begin{rmk}
		The subtlety at infinity prevents one from removing degree $0,n$ intersections outside the $\C^n$ region, so one does not reach an immediate contradiction to the nontriviality of Floer cohomology. This technical failure is necessary, because the Lawlor necks with fixed asymptotic planes are not unique, but do arise in a 1-parameter family. It is in overcoming this technical problem that holomorphic curves appear in \cite{JoyceImagi}.	
	\end{rmk}

	\item Now suppose the Lawlor neck is chosen with the parameter $A=A(L)$, which presumes $A(L)>0$.	Floer theory produces a $J$-holomorphic strip $\Sigma$ with boundary on $\bar{L}$ and $\bar{L}''$ and two corners at some degree $0,n$ intersections $p',q'$. (\cf Prop. \ref{holomorphicstrip} below for more expositions). The only degree $0,n$ intersections must be the two points at infinity $\{ p, q\}$. 
	Now the area of the $J$-holomorphic curve can be computed cohomologically. Using the choice of parameter $A$,
	\[
	\int_{\Sigma}\omega=\int_{\partial \Sigma}\lambda= \int_{p'\to q'} df_{\bar{L}}+ \int_{q'\to p'} df_{\bar{L}''}= \pm( A(L)-A(L'')  ) =\pm(A-A)=0.
	\]
	This contradicts the positivity of area of the holomorphic curve, which proves $L$ must coincide with $L_{\phi,A}$.

	\item Finally one needs to a priori justify $A(L)>0$. This relies on a slightly more complicated holomorphic polygon counting argument, and the main upshot is that one can produce a nontrivial holomorphic triangle from a distinguished triangle in $D^b Fuk(M)$, with the three edges on $\bar{L}$, $\Pi_0\cup\{\infty_0\}$ and $\Pi_\phi\cup \{ \infty_\phi \}$. Then one shows $A(L)$ has the interpretation as its area, so must be positive.

\end{enumerate}

\subsection{Holomorphic curves}\label{holocurves}

We begin with more exposition on Floer theory. The context and notations of this subsection are not restricted to the Lawlor neck setting. Our Floer degree convention agrees with Joyce-Imagi-Santos \cite{JoyceImagi}, but differs with Seidel \cite{Seidel}. Given graded Lagrangians $L,L'$ with transverse intersection point $p$, we put 
\begin{equation}
\mu_{L,L'}(p)= \frac{1}{\pi}( \sum_1^n \alpha_i + \theta_L(p)- \theta_{L'}(p) ) ,  
\end{equation}
where inside the tangent space at $p$ for the ambient symplectic Calabi-Yau manifold,
\[
T_pL\simeq \R^n \subset \C^n ,\quad T_p L'\simeq (e^{i\alpha_1},\ldots e^{i\alpha_n})\R^n,\quad 0<\alpha_i<\pi.
\]
As observed in \cite[(2.7)]{JoyceImagi},

\begin{lem}\label{Floerdegreeinequality}
$\frac{1}{\pi}(\theta_L(p) - \theta_{L'}(p))<  \mu_{L,L'}(p) < \frac{1}{\pi}(\theta_L(p) - \theta_{L'}(p))+n.$
\end{lem}

\begin{prop}\label{holomorphicstrip}\cite[Theorem 2.15]{JoyceImagi}
	Let $(M,\omega), \lambda$ be a symplectic Calabi-Yau Liouville manifold of dimension $2n$, and consider the derived Fukaya category $D^bFuk(M)$ of smooth, embedded, compact, exact, graded Lagrangians in $(M,\omega)$ with $\Z_2$ coefficients. Let $L, L'$ be transversely intersecting Lagrangians in $M$ which are isomorphic as objects of $D^bFuk(M)$. Let $J$ be an almost complex structure on $M$ compatible with $\omega$ and convex at infinity.

	Then there exists $p,q\in L\cap L'$ with $\mu_{L,L'}(p)=0$ and $\mu_{L',L}(q)=n$ and a (possibly broken) $J$-holomorphic strip $\Sigma$ with boundary in $L\cup L'$ and corners at $p,q$. Moreover, we can require the strip to pass through any prescribed point $r\in L$ (resp. $L'$). The area of the strip is 
	\[
	\int_{\Sigma}\omega= (f_L- f_{L'})(q)- (f_{L}- f_{L'})(p),
	\]
	where $f_L, f_{L'}$ denote the Lagrangian potentials.
\end{prop}

\begin{proof}
	Consider the Floer cup product with mod 2 coefficients
	\[
	HF^0( L', L )\otimes HF^0(L, L')\to HF^0(L,L  ),
	\]
	which as $L\simeq L'\in D^bFuk(M)$ can be identified as the cup product
	\[
	H^0( L,L )\otimes H^0(L,L)\to H^0(L,L  ),\quad 1_{L}\cup 1_{L}= 1_{L},
	\]
	and thus must be nontrivial. At chain level this Floer product comes from the $A_\infty$ operation
	\[
	m_2:  CF^0(L', L) \otimes  CF^0( L, L'  )\to CF^0(L,L  )
	\]
	which must be nontrivial.

	If $J$ is a \emph{generic} domain dependent almost complex structure, and $r$ is chosen generically, then this Floer product is defined in terms of counting $J$-holomorphic strips $u: \R_s\times [0,1]_t\to M$, with boundary condition
	\begin{equation}\label{holostripconvention}
	u(s,0)\in L, \quad  u(s,1)\in L', \quad \lim_{s\to -\infty} u= p, \quad \lim_{s\to +\infty} u=q,
	\end{equation}
	for $p\in CF^0(L,L')$, $q\in CF^n(L,L')\simeq CF^0(L',L)$, and $u(0,0)=r$. The nontriviality of the Floer product implies the moduli space is nonempty. Once the genericity assumption is dropped, then transversality requirements in Floer theory may break down. However, in the exact setting, no sphere bubbling or disc bubbling is possible, so the compactness theory of holomorphic discs implies that only strip breaking can happen.

	The area is computed by
	\[
	\int_{\Sigma}\omega= \int_{\partial\Sigma} \lambda= \int_{p \to q} df_L + \int_{q \to p}df_{L'}= (f_L- f_{L'})(q)- (f_{L}- f_{L'})(p).
	\]
\end{proof}

\begin{rmk}
The prescribed point constraint $r$ is not emphasized in \cite{JoyceImagi}, but will play an important role here. Forgetting the point constraint, the moduli space of strips between the degree $0, n$ intersection points, modulo the $\R$-translation action, has virtual dimension $n-1$, and yields interesting information beyond the counting of zero dimensional moduli spaces (\cf \cite{Li}). We will not use higher dimensional moduli space structure in this paper.
\end{rmk}

\begin{rmk}
The role of the degree $0,n$ intersection points are \emph{asymmetric}. If we replace the $s\to-\infty$ limit by the intersection in $CF^n(L,L')$, and the $s\to +\infty$ limit by the intersection in $CF^0(L,L')$, then the virtual dimension will go wrong. 
\end{rmk}

We will need a version of the area monotonicity formula for the portion of the image of $J$-holomorphic curves inside a geodesic ball of the target space, in the presence of Lagrangian boundary.

\begin{lem}\label{areamonotonicitywithboundary}
Let $L$ be a smooth embedded Lagrangian, passing through a point $p$. Let $R\leq 1$. Suppose $\Sigma$ is a nonconstant J-holomorphic curve passing through $p$, whose only boundary in the ambient ball $B(p,R)$ lies on $L$, then its area has a lower bound $\text{Area}(\Sigma\cap B(p,R)) \geq C^{-1}R^2$. The constant depends only on $\omega, J$, and the $C^{1,\alpha}$ regularity of $L$.
\end{lem}

\begin{proof}
Without loss $R$ is small, so that $L$ is locally represented by a $C^{1,\alpha}$-graph.
Consider local Darboux coordinates $\{x_i, y_i\}$ such that $L$ locally coincides with the plane $\R^n_{x_i}$, and the almost complex structure $J_0$ at the origin acts by $J\partial_{x_i}=\partial_{y_i}$. Inside the small coordinate ball, the deviation of $J$ from $J_0$ is of order $O(R^\alpha)\ll 1$. Denote $g_0=\omega(\cdot, J_0\cdot)$ as the constant Euclidean metric in the coordinates, which is uniformly equivalent to the original metric $g=\omega(\cdot, J\cdot)$.

We define the reflection map $\iota: (x_i,y_i)\mapsto (x_i,-y_i)$, which is locally almost an isometry up to $O(R)$ error. The current $\tilde{\Sigma}=\Sigma - \iota_*\Sigma$ then has no boundary inside $B_{g_0}(p, C^{-1}R)\subset B_g(p,R)$, and satisfies for $r\leq C^{-1} R$ that
\[
\text{Mass}(\tilde{\Sigma}\cap B_{g_0}(r)) \leq C\int_{\tilde{\Sigma}\cap B_{g_0}(r)  } \omega .
\]
By the isoperimetric inequality, we can find an integral current $Q$ contained in the ball $B_{g_0}(r)$ with $\partial Q= \partial (\tilde{\Sigma}\cap B_{g_0}(r)  )$ and
\[
\int_{Q\cap B_{g_0}(r)  } \omega \leq \text{Mass}(Q) \leq C \text{Mass}( \partial Q  )^2.
\]
Using Stokes formula
\[
\int_{\tilde{\Sigma}\cap B_{g_0}(r)  } \omega=  \int_{Q\cap B_{g_0}(r)  } \omega.
\]
Combining the above,
\[
\text{Mass}(\tilde{\Sigma}\cap B_{g_0}(r)) \leq C \text{Mass}( \tilde{\Sigma}\cap \partial B_{g_0}(r)  )^2.
\]
Using the coarea formula, the increasing function $f(r)= \text{Mass}(\tilde{\Sigma}\cap B_{g_0}(r))$ satisfies $f\leq C(f')^2$ for a.e $r\lesssim R$, whence
\[
(f^{1/2})' \geq C^{-1}.
\]
Since the holomorphic curve passes through the origin, we have $f(r)>0$ for any $r>0$. 
Integrating the differential inequality gives the claim.
\end{proof}

\subsection{Singular version of Lawlor neck uniqueness}\label{Singularuniquenesssection}

We now set the scene of our improved version. 

\begin{itemize}

\item (\textbf{Asymptotic near infinity}) Let $L$ be a multiplicity one special Lagrangian closed integral current inside $\C^n$, with \textbf{tangent cone at infinity}  being the union of two tranverse planes $\Pi_0\cup \Pi_\phi$, where 
\[
\Pi_0=\R^n, \quad \Pi_\phi= (e^{i\phi_1},\ldots e^{i\phi_n})\R^n, \quad 0<\phi_i<\pi.
\]
In particular by the monotonicity formula for minimal surfaces, the volume ratio $\frac{\text{Vol}(L\cap B(r))}{\omega_n r^n}$ is bounded by two.

These assumptions are known to imply the following improved asymptotic description up to translation, which we will use henceforth:

\begin{lem}
Up to translating $L$ by a vector in $\C^n$, 
there is some $R_0>0$, such that $L\setminus B(R_0)$ is a smooth graph of a normal vector field over the planes $\Pi_0$ and $\Pi_\phi$, with decay bounds $|\partial^k v|=O( r^{\rho-k-1}) $ for any $2-n< \rho<0$. 
\end{lem}

\begin{proof}
Due to the integrability of the tangent cone at infinity, 
Lambert-Lotay-Schulze \cite[section 5]{LambertLotaySchulze} shows by an optimal Lojasiewicz inequality that $L\setminus B(R_0)$ is graphical with $|\partial^k v|= O(r^{\alpha-k})$ for some $\alpha<1$. The improvement of decay rate is discussed in \cite[Prop 4.6]{JoyceImagi}. The idea is that the obstruction of crossing the critical rates comes from homogeneous harmonic functions on the tangent cone. Degree one harmonics can be accounted for by a translation, while degree zero harmonics cause no obstruction since the link of the tangent cone has first Betti number zero.
\end{proof}

\item (\textbf{Smooth approximations})
We assume there exist smooth, embedded, exact, graded Lagrangians $L_i$ defined on large Euclidean balls $B(R_i)$ with $R_i\to +\infty$, such that the Lagrangian angles $\theta_{L_i}$ converges to zero in $C^0_{loc}$ topology, and $L_i\to L$ as integral currents on any fixed ball. Moreover, we assume that on any fixed annulus $B(R)\setminus B(R_0)$, the $L_i$ converge to $L$ in the $C^1$-topology.

Since the link of the tangent cone has first Betti number zero, we can glue $L_i$ with the two ends of $L$ to obtain an exact Lagrangian, which has all the properties of $L_i$. Thus we can without loss assume $L_i$ coincides with $L$ outside some fixed ball  independent of $i\gg 1$ (eg. $B(2R_0)$). According to the classification of exact, graded, embedded, asymptotically conical Lagrangians in $\C^n$ converging to the tangent cone $\Pi_0\cup \Pi_\phi$ at negative decay rate $\rho<0$, we have

\begin{lem}\label{JoyceImagiclassification}
\cite[Cor. 2.18]{JoyceImagi} The Lagrangian $L_i$ has $H^*(L_i,\Z_2)\simeq H^*(S^{n-1}\times \R, \Z_2)$, and in particular is connected. Moreover, either $\sum \phi_i=\pi$ or $\sum \phi_i=(n-1)\pi$. 
\end{lem}

By the symmetry $\phi\leftrightarrow \pi-\phi$, we will henceforth consider $\sum \phi_i=\pi$.

\end{itemize}

\begin{rmk}\label{notC1}
We emphasize that $L$ is \emph{not} a priori assumed to be smooth inside the fixed Euclidean ball $B(R_0)$.  Moreover, inside $B(R_0)$ we do not have quantitative control on $L_i$ beyond the Lagrangian angle, and in particular no assumption is made on the mean curvature. The convergence $L_i\to L$ needs \emph{not} be in $C^1$ topology, even near the smooth locus of $L\cap B(R_0)$.	
\end{rmk}

The main goal of this section is

\begin{thm}(Singular version of Lawlor neck uniqueness)\label{singularuniqueness}
Under the setting, either $L$ coincides with the Lawlor neck $L_{\phi,A}$ for some $A>0$, or $L=\Pi_0\cup \Pi_\phi$. 
\end{thm}

The new difficulties are as follows. The strategy in \cite{JoyceImagi} crucially relies on the \emph{real analyticity} of the special Lagrangian, to remove all degree $0,n$ intersection points inside $\C^n$. As the current technology of Floer theory only applies to smooth Lagrangians, there is the delicate question of uniform control on the smooth approximations of the special Lagrangian. To overcome such problems, we recently introduced the technique of perturbation with prescribed Lagrangian angle \cite{LiquantitativeTY}.  
This leads to quantitative Thomas-Yau uniqueness results inside compact almost Calabi-Yau manifolds. However, in the current setting, the presence of degree $0,n$ intersections at infinity prevents the na\"ive implementation of this technique, and forces us to use holomorphic curve ingredients more prominently than in \cite{LiquantitativeTY}. A technically interesting feature, is that we are able to circumvent the most complicated part of \cite{JoyceImagi} concerning real analyticity and Hamiltonian perturbations, and instead appeal to a more robust holomorphic curve argument.

From henceforth in this section, we will assume that $L$ is not equal to $L_{\phi,A}$ or $\Pi_0\cup \Pi_\phi$, and aim to derive a contradiction. We shall also freely pass to subsequences of $L_i$.

\subsection{Monotonicity inequality}\label{Monotonictysection}

The Lagrangians $L_i$ are only assumed to have $|\theta|$ small, but there is no quantitative bound on the mean curvature. In this setting, some weaker consequences of the monotonicity formula for minimal surfaces continue to hold. The following is a simpler version of \cite[Prop. 2.6]{LiquantitativeTY}, which is based on Almgren's isoperimetric inequality \cite{Almgrenisoperimetric}.

\begin{prop}\label{monotonicity1}
	Let $L'\subset \C^n$ be a smooth Lagrangian with $|\theta|\leq \epsilon \ll 1$, and $p\in L'$. Then there is a lower bound on the volume of $L'\cap B_{Eucl}(p,r)$,
	\[
	\text{Vol}(L'\cap B_{Eucl}(p, r)) \geq \omega_n r^n (1-C\epsilon^2).
	\]
	The same bound holds for the intrinsic ball.

\end{prop}

\begin{proof}
For a.e. $0<r$, the level set $ L'\cap \partial B_{Eucl}(r)$ is smooth, and the sharp isoperimetric inequality allows us to find an integral current $Q$ with $\partial Q=L'\cap \partial B_{Eucl}(r)$ and mass bound
\[
Mass(Q) \leq 
n^{ -\frac{n}{n-1}  } \omega_n^{\frac{-1}{n-1}}
\mathcal{H}^{n-1}(L'\cap B_{Eucl}(r))^{\frac{n}{n-1}}.
\]
Since $\partial (L'\cap B_{Eucl}(r))= \partial Q$, the form $\text{Re}\Omega$ is closed, and $\text{Re}\Omega |_Q \leq dvol_Q$,
\[
\int_{ L'\cap B_{Eucl}(r)} \text{Re} \Omega = \int_{ Q} \text{Re}\Omega \leq  Mass(Q).
\]
Since $|\theta|\leq \epsilon$, we have 
$ \text{Vol}(L'\cap B_{eucl}(r))\leq \frac{1}{\cos \epsilon} \int_{ L'\cap B_{Eucl}(r)}  \text{Re} \Omega .$
Combining the above, and applying the coarea formula, 
the function $f(r)= \text{Vol}(L'\cap  B_{Eucl}(r)  )$ satisfies
\[
f(r)\leq n^{-\frac{n}{n-1} }\omega_n^{-\frac{1}{n-1}} (1+C\epsilon^2) f'(r)^{  \frac{n}{n-1} }.
\]
For $r>0$, the function $f$ is positive, increasing, and
\[
(f^{1/n})' \geq \omega_n^{1/n}(1-C\epsilon^2),
\]
whence the claim. For the intrinsic ball analogue, we simply replace the function $r$ with a regularized version of the intrinsic distance to $p$.
\end{proof}

\begin{prop}\label{monotonicity2}
Let $L'\subset \C^n$ be a smooth Lagrangian with $|\theta|\leq \epsilon \ll 1$, and $p\in L'$. Then the function
\[
\frac{\text{Vol}(L'\cap  B_{Eucl}(p, r)  ) }{\omega_n r^{n\cos \epsilon} }
\]
is nondecreasing in $r$.
\end{prop}

\begin{proof}
In the above proof, we can replace the integral current $Q$ by the cone over $L'\cap \partial B_{eucl}(r)$, whose mass is 
\[
\text{Mass}(Q)= \frac{r}{n} \mathcal{H}^{n-1} ( L'\cap \partial B_{eucl}(r)   ).
\]
The same comparison argument shows
\[
\cos \epsilon\, \text{Vol}(L'\cap B_{eucl}(r)) \leq 
\int_{ L'\cap B_{Eucl}(r)} \text{Re} \Omega = \int_{ Q} \text{Re}\Omega \leq  Mass(Q).
\]
Combining the above, the function $f(r)= \text{Vol}(L'\cap  B_{Eucl}(r)  )$ satisfies
\[
f'\geq n\cos \epsilon \frac{f}{r},
\]
whence $\log f- n\cos \epsilon \log r$ is a nondecreasing function.
\end{proof}

The monotonicity inequalities have some standard consequences on our sequence $L_i$ converging in the current topology to $L$. Recall that $\norm{\theta_{L_i}}_{C^0}\to 0$ by assumption, so we may assume $|\theta_{L_i}|\leq \epsilon$.

\begin{cor}\label{Hausdorffconvergence1}
Any point $p\in \C^n$ lies on the support of $L$ if and only if there is a sequence $p_i\in L_i$ with $p_i\to p$. 
\end{cor}

\begin{proof}
If $p\in \text{supp}(L)$, then since $L$ is a special Lagrangian,
 \[
\int_{B(p,r)\cap L} \text{Re}\Omega= \text{Vol}(L\cap B(p,r)) \geq \omega_n r^n,\quad \forall r>0.
\]
But the definition of current convergence implies that for any compactly supported cutoff function $\chi$,
\[
\int_{L_i } \chi \text{Re}\Omega \to \int_{L } \chi \text{Re}\Omega. 
\]
so by taking very concentrated $\chi$, the nontriviality of $\int_{L_i } \chi \text{Re}\Omega$ allows us to find the nearby $p_i$ converging to $p$.

Suppose $p_i\in L_i$ converges to $p$, then 
Prop \ref{monotonicity1} implies 
\[
\int_{B(p_i,r)\cap L_i} \text{Re}\Omega\geq \cos \epsilon \text{Vol}(L_i\cap B(p_i,r)) \geq \omega_n (1- C\epsilon^2) r^n,\quad \forall r>0.
\]
Whence
\[
\int_{B(p,r)\cap L_i} \text{Re}\Omega\geq  \omega_n (1- C\epsilon^2) (r-|p_i-p|)^n,\quad \forall r>0.
\]
For fixed cutoff functions $\chi$, this provides a uniform lower bound on $\int_{L_i}\chi \text{Re}\Omega$, so $\int_{L}\chi \text{Re}\Omega>0$ in the limit, and $p\in \text{supp}(L)$.
\end{proof}

By a standard compactness argument,

\begin{cor}\label{Hausdorffconvergence2}
The sequence $L_i$ converges to $L$ in the Hausdorff distance.
\end{cor}

\subsection{Uniform continuity of the Lagrangian potential}\label{UniformcontinuityLagpotentialsection}

We now investigate the Lagrangian potentials $f_{L_i}$ on $L_i$, with the goal of establishing subsequential uniform convergence to some Lagrangian potential on the singular Lagrangian $L$, with uniform estimates on the modulus of continuity with respect to the \emph{extrinsic distance}. This information will eventually be used to control the area of holomorphic curves. The technical aspects of this subsection are designed to overcome the difficulties in Remark \ref{notC1}.

Since $L_i$ is asymptotically conical with rate $\rho<0$, it makes sense to impose the normalization $\lim_{r\to \infty \text{ along } \Pi_0} f_{L_i}=0 $. Likewise the limit 
\[
A(L_i)= \lim_{  r\to \infty \text{along } \Pi_\phi} f_{L_i}
\]
is well defined.

\begin{lem}
There are uniform bounds
\[
\text{Vol}(L_i\cap B(R_0)) \leq C, \quad \norm{f_{L_i}}_{L^\infty}\leq C.
\]
\end{lem}

\begin{proof}
On the two ends outside $B(R_0)$, the oscillation of $f_{L_i}$ is uniformly bounded, due to the fixed asymptote with rate $\rho<0$. Since $|\theta|$ is small, the volume of $L_i$ inside $B(R_0)$ is almost the same as the integral of $\text{Re}\Omega$, which can be computed as a boundary integral by Stokes formula. Since the boundary term is $C^1$ controlled, we have $\text{Vol}(L_i\cap B(R_0)) \leq C$.

By the lower bound on the volume of intrinsic balls in Prop. \ref{monotonicity1}, this implies that for any $p\in L_i\cap B(R_0)$, an intrinsic ball of uniformly bounded radius must intersect $L_i\setminus B(R_0)$. Since $df_{L_i}=\lambda_{L_i}|_{L_i}$ is $C^0$ bounded, the potential $f_{L_i}$ is Lipschitz with respect to the intrinsic distance. Consequently there exists $q\in L_i\cap \partial B(R_0)$ with $|f_{L_i}(p)-f_{L_i}(q)|\leq C$, whence either $|f_{L_i}(p)|\leq C$ or $|f_{L_i}(p)- A(L_i)|\leq C$.

To conclude, we claim that $A(L_i)$ is uniformly bounded. This follows from the intermediate value theorem as $L_i$ is connected (\cf Lemma \ref{JoyceImagiclassification}).  
\end{proof}

\begin{lem}\label{definitevolumeratiodrop}
(Definite volume ratio drop) There is a small constant $c_0>0$ depending on $L$, such that the volume ratio
\[
\frac{Vol(L\cap B(p,r)) }{\omega_n r^n} \leq 2- 3c_0, \quad \forall p\in B(R_0),\forall r\leq 3R_0.
\]
Moreover, for large enough $i$, we have 
\[
\frac{Vol(L_i\cap B(p,3R_0)) }{\omega_n (3R_0)^n} \leq 2- 2c_0, \quad \forall p\in B(R_0).
\]
\end{lem}

\begin{proof}
For the first statement, notice by volume monotonicity for the special Lagrangian $L$, we only need to consider $r=3R_0$. If $c_0$ does not exist, then the rigidity of the monotonicity formula would imply that $L$ is conical for some central point $p$. Since $L$ is asymptotic to $\Pi_0\cup \Pi_\phi$ with negative decay rate, this would force $L=\Pi_0\cup \Pi_\phi$, which violates our contradiction hypothesis.

For large $i$, the Lagrangian angle $\theta_{L_i}$ can be assumed to be arbitrarily small, so we may replace $\text{Vol}(L_i\cap B(p,3R_0))$ by $\int_{ L_i\cap B(p,3R_0)} \text{Re}\Omega$, which can be computed as a boundary integral by the Stokes formula. The boundary $L_i\cap \partial B(p,3R_0)$ lies in the exterior of $B(2R_0)$, where $L_i$ coincides with $L$, so the second statement reduces to the first. 
\end{proof}

\begin{rmk}
As a caveat, it is \emph{not} obvious that for finite but large enough $i$, we can \emph{simultaneously} achieve the infinite number of conditions
\[
\frac{Vol(L_i\cap B(p,r)) }{\omega_n r^n} \leq 2- 2c_0, \quad \forall p\in B(R_0), \forall r\leq 3R_0.
\]	

\end{rmk}

\begin{prop}\label{Equicontinuity}
(Equicontinuity) Let $\epsilon_i= \norm{\theta_{L_i}}_{C^0}$. Then for any $p,q\in L_i$, we have
\[
|f_{L_i}(p)- f_{L_i} (q)|\leq C |p-q|+ Ce^{-C^{-1}\epsilon_i^{-2} },
\]
for sufficiently large $i$, where the constant depends on $c_0$, but is independent of $p,q$ and $\epsilon_i$.  Here $|p-q|$ is the extrinsic distance. 
\end{prop}

\begin{proof}
The argument hinges on volume monotonicity. The basic heuristic is that $f_{L_i}$ is $C^1$ with respect to the intrinsic distance, so if its modulus of continuity is large with respect to the extrinsic distance, then $L_i$ should almost have a self intersection point, which would work against the volume ratio upper bound.

Since we already have uniform $L^\infty$ bounds on $f_{L_i}$, and since the asymptotic geometry is well controlled, without loss 
$p,q$ lies inside $B(R_0)$, and denote the extrinsic distance as $d=|p-q|$, and we only need to worry about the case $d\ll 1$. Since $|df_{L_i}|$ is bounded, we see that the intrinsic distance between $p, q$ is at least $2C_1^{-1}|f_{L_i}(p)- f_{L_i}(q)|$ for some constant $C_1$. Denote $r_0= C_1^{-1}|f_{L_i}(p)- f_{L_i}(q)|$, then the intrinsic balls of radius $r_0$ centred at $p,q$ are disjoint. By Prop \ref{monotonicity1}, the volumes of both balls are at least $\omega_n r_0^n(1-C\epsilon_i^2)$. In particular, the volume of the extrinsic ball
\[
\text{Vol}(L_i\cap B(p, r_0+d) ) \geq 2\omega_n r_0^n(1-C\epsilon_i^2),
\]
whence the volume ratio
\[
\frac{\text{Vol}(L_i\cap B(p, r_0+d) )}{\omega_n (r_0+d)^n} \geq 2 (1-C\epsilon_i^2)(1- C \frac{d}{r_0}).
\]
We now apply Prop \ref{monotonicity2}, to deduce
\[
\frac{\text{Vol}(L_i\cap B(p, r_0+d) )}{\omega_n (r_0+d)^n}
\leq  \frac{\text{Vol}(L_i\cap B(p, 3R_0) )}{\omega_n (3R_0)^n} (\frac{r_0+d}{R_0})^{n\cos \epsilon_i-n}.
\]
Combined with Lemma \ref{definitevolumeratiodrop}, we arrive at 
\[
2(1-C\epsilon_i^2)(1- C \frac{d}{r_0}) \leq (\frac{r_0+d}{R_0})^{n\cos \epsilon_i-n} 2(1-c_0).
\]
For $i\gg 1$, we may assume $\epsilon_i^2\ll c_0$, so the $1-C\epsilon_i^2$ factor is negligible. Thus either $Cr_0^{-1}d \geq c_0/2$, or $r_0\gg d$ and $r_0^{ n\cos \epsilon_i-n  } \geq 1+C^{-1} c_0$. The first case unravels to the Lipschitz bound
\[
|f_{L_i}(p)- f_{L_i}(q)|\leq C |p-q|
\]
for some new constant $C$ depending on $c_0$. The second case implies $r_0\leq e^{-C^{-1}\epsilon_i^{-2}}$ for some large constant depending on $c_0$, whence
$|f_{L_i}(p)- f_{L_i}(q)|\leq C_1e^{-C^{-1}\epsilon_i^{-2}}$. 
\end{proof}

We now extract a limiting function $f_L$ on the support of $L$ by an Arzela-Ascoli style procedure as follows. Pick a countable dense subset of points $p^j$ on the support of $L$. By Cor. \ref{Hausdorffconvergence1}, for each such $p^j$, we can find a sequence $p_i^j\in L_i$ converging to $p$. By a diagonal argument, by passing to subsequence of $L_i$, we can assume $f_{L_i}(p_i^j)$ converges to some value, which we denote as $f_L(p^j)$. By passing Prop. \ref{Equicontinuity} to the limit $i\to +\infty$, we see that $f_L$ is Lipschitz with respect to the extrinsic distance, hence extends to a Lipschitz continuous function $f_L$ on $\text{supp}(L)$. We can pick an auxiliary extension of $f_L$ as a Lipschitz function on $\C^n$, so that we can compare $f_L$ with $f_{L_i}$ directly.

\begin{cor}\label{potentialC0convergence}
Up to subsequence $\norm{f_{L_i}-f_L}_{C^0(L_i)}\to 0$ as $i\to +\infty$.
\end{cor}

\begin{proof}
It suffices to ensure $\norm{f_{L_i}-f_L}_{C^0(L_i\cap B(R_0))}\to 0$. This follows from Prop. \ref{Equicontinuity} and the Hausdorff convergence of $L_i$ to $L$, using finite $\epsilon$-net approximations.
\end{proof}

\begin{cor}
(Potential on the special Lagrangian)
The function $f_L$ satisfies $df_L=\lambda|_L$ on the smooth locus of the special Lagrangian $L$.	In particular $f_L$ is a harmonic function on $L$. Moreover,
\[
\lim_{r\to \infty \text{along } \Pi_0 } f_L=0, \quad  A(L):=\lim_{r\to +\infty \text{along } \Pi_\phi} f_L=\lim_{i\to +\infty} A(L_i).
\]
\end{cor}

\begin{proof}
The weak formulation of $df_{L_i}=\lambda|_{L_i}$ is that for any compactly supported test $(n-1)$-form $\chi$,
\[
\int_{L_i} f_{L_i} d\chi =- \int_{L_i} \lambda\wedge \chi.
\]
Now 
\[
\int_{L_i} \lambda\wedge \chi\to \int_L \lambda\wedge \chi, \quad \int_{L_i} f_L d\chi\to  \int_{L} f_L d\chi
\]
 by the definition of current topology convergence, and
\[
|\int_{L_i} f_{L_i} d\chi - \int_{L_i} f_L d\chi | \leq C(\chi) \norm{f_{L_i}-f_L}_{C^0(L_i)}\to 0.
\]
Thus the weak formulation of $df_L=\lambda|_L$ holds:
\[
\int_{L} f_{L} d\chi =- \int_{L} \lambda\wedge \chi.
\]
Almgren's regularity theorem says that the special Lagrangian $L$ is smooth away from a Hausdorff codimension two subset. On the smooth locus, we infer $df_L=\lambda|_L$. Since $f_L$ is the Hamiltonian function generating the rescaling deformation of $L$, it is well known to be harmonic. The asymptotic values of $f_L$ easily follow from the $C^0$-convergence of $f_{L_i}$.
\end{proof}

\begin{cor}
The invariant $A(L)>0$. 
\end{cor}

\begin{proof}
Recall we are in the case $\sum \phi_i=\pi$ (\cf Lemma \ref{JoyceImagiclassification}), where Joyce-Imagi-Santos \cite[Prop 4.19]{JoyceImagi} proves that $A(L_i)>0$ by interpreting it as the area of a $J$-holomorphic curve. 
Passing to the limit, we have $A(L)\geq 0$.

Suppose $A(L)=0$. Since $f_L^2$ is subharmonic on $L$ (\cf \cite[Lemma 2.13]{CollinsLi}), the monotonicity formula shows
\[
f_L^2 (p)\leq \frac{1}{\omega_n r^n} \int_{L\cap B(p,r)} f_L^2,\quad  \forall r>0.
\]
Notice $A(L)=0$ implies $f_L\to 0$ at both infinity ends, so taking the $r\to +\infty$ limit, the RHS is zero, and $f_L=0$ on $L$ identically. Thus $L$ is conical, so must be equal to $\Pi_0\cup \Pi_\phi$, contradicting our hypothesis.
\end{proof}

\subsection{Perturbation of the Lawlor neck}\label{PerturbationofLawlor}

The basic versions of the Thomas-Yau uniqueness argument depends on the ability to remove degree $0, n$ intersections by perturbing the Lagrangians. When the Lagrangians are close to being singular, removing intersections in a uniform fashion is quite challenging, and instead we will  choose judicious perturbations to ensure certain `Novikov exponents' at Lagrangian intersections almost have the appropriate signs.

We start with the Lawlor neck $L_{\phi,A}$ with $A=A(L)>0$, which is asymptotic to $\Pi_0\cup \Pi_\phi$ at decay rate $2-n$, and in particular we can construct a Weinstein $\epsilon_0$-neighbourhood for some $\epsilon_0>0$ depending on $L_{\phi,A}$. Inside this neighbourhood, the function $f_L$ (extended to $\C^n$) is Lipschitz, and we can arrange
\[
|\nabla^k(f_L- f_{L_{\phi,A}})|= O( |\vec{x}|^{\rho-k} ),\quad \forall k\geq 0 
\]
 for any given $2-n<\rho<0$ along the $\Pi_0, \Pi_\phi$ ends outside $B(R_0)$.

We introduce the standard weighted H\"older spaces $C^{k,\alpha}_{\gamma}$ with growth rate $\gamma$ on the asymptotically conical $L_{\phi,A}$, with norm
\[
\sum_{j=0}^k |\nabla^j u| (1+|\vec{x}|)^{j-\gamma} + \sup_{|p|\sim |q|} \frac{|\nabla^ku(p)- \nabla^ku(q)|}{ |p-q|^\alpha } (1+|\vec{x}|)^{ k+\alpha- \gamma  }.
\]
The analytic Laplacian maps $\Lap: C^{k+2,\alpha}_\gamma\to C^{k,\alpha}_{\gamma-2}$. 
We solve the Poisson equation on $L_{\phi,A}$ with fast decaying RHS in $C^{0,\alpha}_{\rho-2}$:
\begin{equation}\label{Poissonauxiliary}
\Lap u= (f_L- f_{L_{\phi,A}})(1+|\vec{x}|)^{-2}. 
\end{equation}
The solution $u$ exists in $C^{2,\alpha}_{\rho}$, and is $C^{k,\alpha}_\rho$ for any $k$ outside of $B(2R_0)$. We can find smooth approximation $\tilde{u}$ which is arbitrarily close to $u$ in $C^{2,\alpha}_{\rho}$ norm, and coincides with $u$ near infinity.

Given a very small parameter $0<\delta\ll \epsilon_0$,
we now construct the Hamiltonian perturbed Lagrangian $L_\delta$ as the graph of $\delta d\tilde{u}$ in the Weinstein neighbourhood of $L_{\phi,A}$.

\begin{prop}\label{perturbationLagangle}
The Lagrangian $L_\delta$ is smooth embedded, exact, asymptotically conical with ends $\Pi_0\cup \Pi_\phi$ at decay rate $\rho<0$. The limiting values of its Lagrangian potential along $\Pi_0, \Pi_\phi$ are $0,A(L)$ respectively.
Moreover, the Lagrangian angle satisfies
\[
\theta_{L_\delta}= \delta(f_L- f_{L_\delta})(1+|\vec{x}|)^{-2}+ \delta^2 O( (1+|\vec{x}|)^{\rho-2} ).
\]

\end{prop}

\begin{proof}
All claims are immediate except the Lagrangian angle. Recall $L_{\phi,A}$ is a special Lagrangian, and for small deformations associated to $C^2$ Hamiltonian functions
\[
\theta_{L_\delta}= \theta_{L_{\phi,A}  } + \delta \Lap \tilde{u} + \delta^2 O( |\nabla \tilde{u} |^2 +  |\nabla^2 \tilde{u}|^2 ). 
\]
The difference between $u$ and $\tilde{u}$ can be ignored. From the weighted H\"older norm estimate,
\[
|\nabla u|=O( (1+|\vec{x}|)^{\rho-1}  ), \quad |\nabla^2 u|= O( (1+|\vec{x}|)^{\rho-2}  ),
\]
hence
\[
\theta_{L_\delta}= \delta \Lap u + \delta^2 O( (1+|\vec{x}|)^{2\rho-2} )= \delta(f_L- f_{L_{\phi,A}})(1+|\vec{x}|)^{-2}+ \delta^2 O( (1+|\vec{x}|)^{2\rho-2} ).
\]

There is a small caveat: in the line above $f_L$ refers to its restriction on $L_{\phi,A}$. Its deviation from the value of $f_L$ on $L_\delta$ is bounded by
\[
O ( \delta |d\tilde{u}| |df_L|  )= O(\delta (1+|\vec{x}|)^{\rho-1} (1+|\vec{x}|)^{\rho-1}   )= \delta O( (1+|\vec{x}|)^{2\rho-2}   ).
\] 
Multiplied by the $\delta$ in front of $f_L$, the upshot is that this ambiguity does not matter.

Finally, the deviation between $f_{L_{\phi,A}}$ and $f_{L_\delta}$ is bounded by $\delta O((1+ |\vec{x}|)^\rho)$. Replacing $f_{L_{\phi,A}}$ by $f_{L_\delta}$ introduces a term in $\theta_{L_\delta}$ bounded by $\delta^2 O((1+ |\vec{x}|)^{\rho-2})$.
\end{proof}

By the flexibility of arbitrarily small perturbations on $\tilde{u}$, 
without loss the intersection of $L_\delta$ with $L_i$ is transverse. The following is a key step in our strategy, which achieves control on the Lagrangian potential difference (`Novikov exponents') at the degree $0,n$ intersection points.

\begin{cor}\label{degreepotential1}
For $i\gg 1$ and fixed small $\delta$, 
at any degree zero intersection $p\in CF^0(L_\delta, L_i)$, we have $(f_{L_i} - f_{L_\delta})(p) \leq C\delta$, while at any degree $n$ intersection $q\in CF^n(L_\delta, L_i)$, we have $(f_{L_\delta} - f_{L_i})(q) \leq C\delta$.
\end{cor}

\begin{proof}
In our Floer degree convention, Lemma \ref{Floerdegreeinequality} implies
\[
\frac{1}{\pi}( \theta_{L_\delta}- \theta_{L_i}  ) <
\mu_{L_\delta, L_i}(p) < \frac{1}{\pi}( \theta_{L_\delta}- \theta_{L_i}  )+n. 
\]
By Prop \ref{perturbationLagangle},
\begin{equation}
\theta_{L_\delta}- \theta_{L_i} \geq \delta(f_{L}- f_{L_\delta})(1+|\vec{x}|)^{-2}- C\delta^2 (1+|\vec{x}|)^{\rho-2} - \norm{\theta_{L_i}}_{C^0}.
\end{equation}
Thus at a degree zero intersection $p\in CF^0(L_\delta, L_i)$, 
\[
(f_{L} - f_{L_\delta})(p) <   C\delta (1+|\vec{x}|)^\rho + \delta^{-1}\norm{\theta}_{L_i} (1+|\vec{x}|^2).
\]
On the other hand, from the asymptotic control outside $B(R_0)$, we know
\[
|f_L- f_{L_\delta}| \leq C (1+|\vec{x}|)^\rho, 
\]
where $\rho<0$. 
By separating the cases with $|\vec{x}|\leq \delta^{1/\rho}$ and $|\vec{x}|\geq \delta^{1/\rho}$, we deduce
\[
(f_{L_i} - f_{L_\delta})(p) <   C\delta + C\delta^{-1+2/\rho}\norm{\theta}_{L_i} +\norm{f_L- f_{L_i}}_{C^0} .
\]

We are given $\norm{\theta_{L_i}}_{C^0}\to 0$, and $\norm{f_L- f_{L_i}}_{C^0}\to 0$ (\cf Cor. \ref{potentialC0convergence}). Consequently for $i\gg 1$, we can assume \[
\norm{\theta_{L_i}}_{C^0}\ll \delta^{2-2/\rho}, \quad 
\norm{f_L- f_{L_i}}_{C^0} \ll \delta. 
\]
Thus at $p\in CF^0(L_\delta, L_i)$, we have
\[
(f_{L_i} - f_{L_\delta})(p) \leq C\delta. 
\]
The proof for the degree $n$ intersection is completely analogous. 
\end{proof}

\begin{rmk}
In the constructions above, the precise choice of the Poisson equation (\ref{Poissonauxiliary}) is not essential. What really matters is that the forcing term \emph{detects the sign} of the potential difference, and has sufficiently fast decay at infinity.
\end{rmk}

\subsection{Partial compactification}\label{partialcompactification}

We now recall from Imagi-Joyce-Santos \cite[section 2.4]{JoyceImagi} the partial compactification construction of $\C^n$ and the asymptotically conical Lagrangians with planar ends therein.

We introduce the real coordinates $x_1,\ldots x_n, y_1,\ldots y_n$ on $\C^n$ by
\[
x_j= \text{Re} z_j-\cot \phi_j \text{Im}z_j,\quad y_j= \text{Im} z_j, \quad j=1,2,\ldots n.
\]
In these coordinates $\omega= \sum dx_j\wedge dy_j$, and
\[
\Pi_0= \{ (x_1,\ldots x_n, 0,\ldots 0): x_j\in \R    \}, \quad \Pi_\phi= \{ (0,\ldots, 0,y_1,\ldots y_n): y_j\in \R  \}.
\]
Identify $(\C^n, \omega)\simeq T^*\Pi_0\simeq T^*\Pi_\phi$ as symplectic manifolds by identifying
\[
(x_1,\ldots x_n, y_1,\ldots y_n)\leftrightarrow \sum y_jdx_j \leftrightarrow -\sum x_jdy_j \in T^*\Pi_\phi. 
\]
Write $\mathcal{S}_0= \Pi_0\cup \{\infty_0 \}$ and $\mathcal{S}_\phi=\Pi_\phi\cup \{\infty_\phi  \}$ for the one-point compactifications of $\Pi_0$ and $\Pi_\phi$. Define $F(r)= \frac{1}{ \log (1+r^2) }$, and define coordinates $(\tilde{x}_1,\ldots \tilde{x}_n)\in \mathcal{S}_0\setminus \{ 0\}$ and $ (\tilde{y}_1,\ldots \tilde{y}_n)\in \mathcal{S}_\phi\setminus \{ 0\}$ by
\[
\tilde{x}_j= \begin{cases}
0, \quad & p=\infty_0,
\\
\frac{F( \sqrt{\sum x_j^2}  )x_j}{\sqrt{\sum x_j^2}},\quad & p= (\tilde{x}_1,\ldots \tilde{x}_n)\in \mathcal{S}_0\setminus \{ 0\},
\end{cases}
\]
and $\tilde{y}_j$ similarly. Then $\mathcal{S}_0, \mathcal{S}_\phi$ have  smooth structures diffeomorphic to the standard $n$-sphere.

On the level of sets, the partial compactification is
\[
M= \C^n \sqcup T^*_{\infty_0} \mathcal{S}_0\sqcup T^*_{\infty_\phi} \mathcal{S}_\phi,
\]
which is the plumbing of the cotangent bundles $T^*\mathcal{S}_0$ and $T^* \mathcal{S}_\phi$ along the common open set $\C^n$. The manifold $M$ inherits the natural symplectic structure $\omega$. The usual Liouville form on $\C^n$ does not extend to $M$. Instead, let $T\gg 1$ be a parameter, and let $\eta:\R\to [-1,1]$ be smooth with $|\eta'|=O(T^{-1})$, and
\[
\eta(t)= \begin{cases}
-1, \quad t\leq -2T,
\\
0, \quad -T\leq t\leq T,
\\
1,\quad t\geq 2T.
\end{cases}
\]
Define a smooth 1-form $\tilde{\lambda}$ on $\C^n$ by
\[
\tilde{\lambda}=\lambda+dh,\quad \lambda= \frac{1}{2} (\sum x_jdy_j-\sum y_jdx_j),\quad h= -\frac{1}{2} \eta( \sum x_j^2-\sum y_j^2   )(\sum x_jy_j). 
\]
Then $d\tilde{\lambda}=\omega$ and agrees with the natural Liouville forms on the cotangent bundles asymptotically, so extends to a smooth Liouville form on $M$.

Similarly $J,\Omega$ do not extend smoothly to $M$, but we can define modified versions $\tilde{J}, \tilde{\Omega}$ which do, so that $(M,\omega)$ is symplectically Calabi-Yau, albeit no longer complex integrable. We can arrange $\tilde{J}=J$ and $\tilde{\Omega}=\Omega$ when $-T\leq \sum x_j^2-y_j^2\leq T$ for $T\gg 1$. 
We choose $\tilde{J}$ so that $\tilde{J}(T_{\infty_0})=T_{\infty_0}(T^*_{\infty_0}\mathcal{S}_0)$ and $\tilde{J}(T_{\infty_\phi})=T_{\infty_\phi}(T^*_{\infty_\phi}\mathcal{S}_\phi)$, and $\tilde{\Omega}$ so that $\mathcal{S}_0, \mathcal{S}_\phi, T^*_{\infty_0}\mathcal{S}_0, T^*_{\infty_\phi}\mathcal{S}_\phi$ have constant phase, with gradings 
\[
\theta_{\mathcal{S}_0}=0,\quad \theta_{\mathcal{S}_\phi}=\sum \phi_j,\quad  \theta_{ T^*_{\infty_0}\mathcal{S}_0   }= n\pi/2, \quad \theta_{ T^*_{\infty_\phi}\mathcal{S}_\phi  }=\sum \phi_j+ n\pi/2. 
\]
We remark there is really a \emph{family} of such structures parametrized by $T$, and taking $T$ large is akin to neck stretching, which can be leveraged to our advantage.

Now suppose $L'$ is an asymptotically conical Lagrangian in $\C^n$ with cone $\Pi_0\cup \Pi_\phi$ and rate $\rho<0$. Define $\bar{L}'= L'\sqcup \{ \infty_0, \infty_\phi \}$. Then \cite[section 2.4]{JoyceImagi} shows that $\bar{L}'$ is a smooth compact Lagrangian in $(M,\omega)$, intersecting the cotangent fibres at infinity transversely at $\infty_0, \infty_\phi$. Assume that $L'$ is exact and connected, with potential $f_{L'}$, such that $f_{L'}$ tends to zero along $\Pi_0$ and to $A(L')$ along $\Pi_\phi$. Then $\bar{L}'$ is exact in $(M,\omega)$ with potential
\begin{equation}\label{Lagpotentialcompletion}
f_{\bar{L}'} (p)= \begin{cases}
f_{L'} (p)+ h(p), \quad & p\in L',
\\
0, \quad & p=\infty_0,
\\
A(L'),\quad & p=\infty_\phi,
\end{cases}
\end{equation}
where $\tilde{\lambda}=\lambda+ dh$. Moreover, if $L'$ is graded then so is $\bar{L}'$. Based on Abouzaid-Smith \cite{AbouzaidSmith} and by analyzing the intersection pattern of $\bar{L}'$ with the cotangent fibres at infinity, Joyce-Imagi-Santos classified all possibilities of  $\bar{L}'$ as an object of $D^bFuk(M,\omega)$, characterized by distinguished triangles \cite[Thm. 2.17]{JoyceImagi}. The additional information $\theta_{L'}\to 0$ at infinity, and $\sum_1^n \phi_i=\pi$, pins down $\bar{L}'$ uniquely up to isomorphism in $D^bFuk(M)$ \cite[Cor. 2.18]{JoyceImagi}.

We apply the construction to $L_i$ and $L_\delta$ (\cf section \ref{PerturbationofLawlor}), to obtain Lagrangians $\bar{L}_i$ and $\bar{L}_\delta$. We also choose the auxiliary parameter $T=T_i\to +\infty$ very fast. 
The upshot is

\begin{prop}
Both $\bar{L}_i$ and $\bar{L}_\delta$ are isomorphic to the compactification of the Lawlor neck $L_{\phi,A}$, as objects of $D^bFuk(M,\omega)$.
\end{prop}

\subsection{Assembling the pieces}\label{Assemblingpieces}

After arbitrarily small Hamiltonian perturbation of $\bar{L}_\delta$, supported in arbitrarily small neighbourhoods of the cotangent fibres at infinity, we can achieve transversal intersection between $\bar{L}_\delta$ and $\bar{L}_i$.

\begin{cor}\label{degreepotential2}
For $i\gg 1$ and fixed small $\delta$, 
at any degree zero intersection $p\in CF^0(\bar{L}_\delta, \bar{L}_i)$, we have $(f_{\bar{L}_i} - f_{\bar{L}_\delta})(p) \leq C\delta$, while at any degree $n$ intersection $q\in CF^n(\bar{L}_\delta, \bar{L}_i)$, we have $(f_{\bar{L}_\delta} - f_{\bar{L}_i})(q) \leq C\delta$.
\end{cor}

\begin{proof}
Inside $\C^n$, 
observe that the Floer degree of intersections is independent of the choice of complex volume form $\Omega$ or $\tilde{\Omega}$ (since it is intrinsic characterized in terms of the Lagrangian Grassmannian bundle \cite{Seideldegree}), and by (\ref{Lagpotentialcompletion}) the Novikov exponents at intersections is insensitive to the modification of Liouville form:
\[
f_{\bar{L}_i} (p)- f_{\bar{L}_\delta} (p)= f_{{L}_i} (p)- f_{{L}_\delta} (p).
\]
Thus before the Hamiltonian perturbation of $\bar{L}_\delta$, the conclusion in Cor. \ref{degreepotential1} transfers to the  compactified Lagrangians, for intersections inside $\C^n$.

At the infinity \[
f_{\bar{L}_i} (\infty_0)= f_{\bar{L}_\delta} (\infty_0)= 0,\quad f_{\bar{L}_i} (\infty_\phi)= A(L_i), \quad  f_{\bar{L}_\delta} (\infty_\phi)=A(L).
\]
Since $A(L_i)\to A(L)$, we know for $i\gg 1$ that $|f_{\bar{L}_i}- f_{\bar{L}_\delta} |< \delta$ at the infinity intersections, and a fortiori on the very small neighbourhood near infinity where the last Hamiltonian perturbation is supported. 
The claim follows.
\end{proof}

Applying Prop. \ref{holomorphicstrip} to the Lagrangians $\bar{L}_\delta, \bar{L}_i$, we observe

\begin{cor}
Through any given point $P\in \bar{L}_\delta$, there exists a (possibly broken) holomorphic strip $\Sigma$ with boundary on $\bar{L}_\delta\cup \bar{L}_i$, and $\int_{\Sigma}\omega \leq C\delta$.	
\end{cor}

\begin{cor}
For any $p\in L_{\phi,A}\cap B_{eucl}(0, \frac{1}{2}T_i)$, there exists a point $p_i\in L_i$ with $|p-p_i|\leq C\delta^{1/2}$.
\end{cor}

\begin{proof}
Recall that inside the large Euclidean ball $B(0, T_i)$, the $J,\omega$ are by construction the standard Euclidean mode, and 
we can arrange $\bar{L}_\delta$ to coincide with $L_\delta$ inside $B(0,T_i)$. The displacement from $L_{\phi,A}$ to $L_\delta$ has distance bounded by $C\delta\ll 1$, so we can find $P\in \bar{L}_\delta$ with $|P-p|\leq C\delta$. Our goal is to find a point $p_i\in L_i$ close to $P$.

 We denote $R= \min(\frac{1}{2}\text{dist}(P, L_i), 1 )$, so that $B(P,R)$ is disjoint from $L_i$, and $R\leq 1$ implies in particular that $B(P,R)$ is contained inside $B(0,T_i)$. The possibility $R=0$ means simply that $P\in L_i$, and we are done, so without loss $R>0$.

Since $L_{\phi,A}$ is a fixed smooth Lagrangian, and $L_\delta$ is a small $C^{1,\alpha}$-graph over $L_{\phi,A}$ due to the $C^{2,\alpha}$-regularity for the solution $u$ of the Poisson equation (\ref{Poissonauxiliary}), the $C^{1,\alpha}$ regularity of $L_\delta$ is controlled. We then apply the area monotonicity formula with boundary (\cf Lemma \ref{areamonotonicitywithboundary}), to deduce
\[
C^{-1}R^2\leq \int_{ \Sigma \cap B(P,R)} \omega \leq \int_{ \Sigma} \omega \leq C\delta,
\]
whence $R\leq C\delta^{1/2}$. As $\delta\ll 1$, this shows $\text{dist}(P, L_i) \leq C\delta^{1/2}$. Thus there is some $p_i\in L_i$ with
\[
|p-p_i|\leq |p-P|+ |P- p_i|\leq C\delta+ C\delta^{1/2} \leq C\delta^{1/2}.
\]
\end{proof}

\begin{cor}
For any $p\in L_{\phi,A}$, there is a point $p'\in \text{supp}(L)$ with $|p-p'|\leq C\delta^{1/2}$. 
\end{cor}

\begin{proof}
As $T_i\to +\infty$, 
any $p\in L_{\phi,A}$ is contained in $ B_{eucl}(0, \frac{1}{2}T_i)$ for large enough $i$. We can then take a subsequential limit $p'\in B(p, C\delta^{1/2})$ for $p_i\in L_i$. By Cor. \ref{Hausdorffconvergence1}, we have $p'\in \text{supp}(L)$.
\end{proof}

\begin{cor}
$L=L_{\phi,A}$ as integral currents. 
\end{cor}

\begin{proof}
Since $\delta $ can be chosen arbitrarily small, the previous statement implies that $L_{\phi,A}$ is contained in the support of the special Lagrangian $L$. Since $L$ and $L_{\phi,A}$ both have asymptotic area ratio two, $L-L_{\phi,A}$ must be empty.
\end{proof}

We have arrived at the promised contradiction, so Theorem \ref{singularuniqueness} is proved.

\section{Proof of the main results}\label{section:main}
Given the main estimate, Proposition~\ref{prop:T1Ttheta}, from Section~\ref{section:Decayalongflow} the proof of the main results follows established methods, 
similar to the work in \cite{LSSz2} for instance. For the readers' convenience we give the main steps in the proofs.

The main result on the existence of good blowups is the following. Recall the standing assumptions stated before Theorem~\ref{thm:goodblowup} for the flow $L^0_t$.
\begin{thm}\label{mainthmproof}
  There exists some type II blowup along the flow $L^0_t$ given by an exact special Lagrangian
  current $L_\infty$ which has tangent cone $W_0$ at infinity, but
  which is not equal to a translate of $W_0$.
\end{thm}
\begin{proof}
  In the terminology of Definition~\ref{defn:goodblowup} our goal is to show that there exists a good blowup along the flow $L^0_t$. Suppose that this is not the case. Then we are in the setting of Section~\ref{section:Decayalongflow}, with the Conditions (A) holding for any suitable rescaled flow obtained from $L^0_t$. In particular Proposition~\ref{prop:T1Ttheta} applies. We can now argue as follows, to show that in fact a good blowup does exist. 

  First we show that $W_0$ is the unique tangent flow at $(0,0)$. Let us define the rescaled flow $M_\tau= e^{\tau/2} L^0_{-e^{-\tau}}$, so that along a sequence $\tau_i\to\infty$ we have $M_{\tau_i} \to W_0$. Consider the time translated rescaled flows
  \[ M^i_\tau = M_{\tau+\tau_i}. \]
 Let $\epsilon > 0$, and let $\delta > 0$ be determined by Proposition~\ref{prop:T1Ttheta}. Note that for sufficiently large $i$ we have that $M^i_0$ is a $\delta^2$-graph over $W_0$ on the annulus $A_{\delta^{-1}, \delta}$. Proposition~\ref{prop:T1Ttheta} implies that for the flow $M^i$ we have $T_1(\epsilon) \geq T_{\Theta}(\delta)$, but at the same time since $M_{\tau_i}\to W_0$, the monotonicity formula implies that $T_{\Theta}(\delta) = \infty$. It follows that for large enough $i$ we have that $M_{\tau + \tau_i}$ is an $\epsilon^2$-graph over $W_0$ on the annulus $A_{\epsilon^{-1}, \epsilon}$, for all $\tau \geq 0$. Since $\epsilon$ was arbitrary, this shows that $W_0$ is the unique limit of $M_\tau$ as $\tau\to\infty$. 

 Next, let $\kappa > 0$ be small. For any $\tau\geq 0$ and $x\in B(0,1)$ define the radius
 \[ r(\tau, x) = \inf\{ R > 0\, :\, \Theta(M_\tau, x, r) \geq \Theta(W_0) - \kappa, \text{ for all } r\in (R, 1)\},\]
 where we define $r(\tau, x) = 1$ if the set above is empty. 

 Note that by the convergence $M_\tau \to W_0$ as $\tau\to \infty$ we have that for all sufficiently small $x$ and large $\tau$ (depending on $\kappa$), we have $r(\tau, x) < 1$. Moreover we have $r(\tau,0) \to 0$ as $\tau\to \infty$, and also if $\kappa$ is chosen sufficiently small, then we have $r(\tau, x)=1$ for large $\tau$ and $x\in A_{1,1/2}$. 

 It follows that for all sufficiently large $\tau$ we can pick a point $x_\tau \in B(0,1/2)$ such that 
 \[ r(\tau, x_\tau) \leq r(\tau, x), \text{ for all }x\in B(0,1). \]
 Let us define $r_\tau = r(\tau, x_\tau)$, and note that $r_\tau\to 0$ as $\tau\to\infty$. Let us assume that for all sufficiently large $\tau$ we have $r_\tau > 0$, which is the case for instance if the flow $M_\tau$ is embedded, since then $\Theta(M_\tau,x,r)\to 1$ as $r\to 0$ for all $x\in M_\tau$.  For a sequence $\tau_i\to\infty$, we then consider the points $(y_i, t_i)$ along the flow $L^0_t$, which correspond to $(x_{\tau_i}, \tau_i)$ along the rescaled flow. In other words
 \[ t_i = -e^{-\tau_i}, \qquad y_i = e^{-\tau_i/2} x_{\tau_i} = \sqrt{-t_i} x_{\tau_i}. \]
  We let $r_i = e^{-\tau_i/2} r(\tau_i, x_{\tau_i})$ denote the corresponding rescaled radii. We claim that the points $(y_i, t_i)$ together with the scales $r_i$ almost define a good blowup sequence in the sense of Definition~\ref{defn:goodblowup}, except the resulting special Lagrangian current will be asymptotic to a cone $W$ close to $W_0$.

  Since $r(\tau_i, x_{\tau_i})\to 0$, we have $r_i < \sqrt{-t_i}$ for large $i$. It is also clear that $(y_i, t_i)\to (0,0)$. Consider the Lagrangians $L_i = r_i^{-1}(L^0_{t_i} - y_i)$. We claim that for any given $\lambda \leq 1$, for large $i$ the scaling $\lambda L_i$ is close to $W_0$ in a sense made precise below. To see this, note that $\lambda L_i$ can be viewed as the time 0 slices $M^i_0$ of the rescaled flows $M^i_\tau$ based at the points $(y_i, t_i + \lambda^{-2}r_i^2)$. We have
  \[ M^i_\tau = e^{\tau/2} \lambda r_i^{-1}\Big(L^0_{t_i+\lambda^{-2}r_i^2-e^{-\tau}\lambda^{-2}r_i^2} - y_i\Big), \]
  and so at $\tau_i=\ln (\lambda^{-2} r_i^2)$ we have
  \[ M^i_{\tau_i} = L^0_{t_i+\lambda^{-2}r_i^2-1} - y_i.\]
  Since $t_i + \lambda^{-2}r_i^2\to 0$ and $y_i\to 0$ as $i\to\infty$, it follows that $M^i_{\tau_i} \to W_0$. At the same time, by our construction, we have 
  \[ \Theta(M^i_0) \geq \Theta(W_0) - \kappa, \]
  for large $i$. 
  Therefore for any $\epsilon > 0$ if we choose $\kappa$ sufficiently small, then Proposition~\ref{prop:T1Ttheta} (applied to the time translated flow $\tau\mapsto M^i_{\tau + \tau_i}$) implies that $M^i_0$ is an $\epsilon^2$-graph over $W_0$ on the annulus $A_{\epsilon^{-1}, \epsilon}$.

  Consider the sequence $L_i$. Proposition~\ref{prop:N1} implies that up to choosing a subsequence, we have $L_i\to L_\infty$, where $L_\infty$ is a special Lagrangian current. Moreover our construction implies that $\Theta(L_\infty) = \Theta(W_0) - \kappa$, while $\Theta(L_\infty - x) \leq \Theta(W_0) - \kappa$ for all $x$. The discussion above also shows that for all $\lambda \leq 1$ the rescaling $\lambda L_\infty$ is an $\epsilon^2$-graph over $W_0$ on $A_{\epsilon^{-1}, \epsilon}$ (where recall that $\kappa$ depends on the choice of $\epsilon$). In sum this implies that $L_\infty$ is a special Lagrangian current, asymptotic to a cone $W$ at infinity that is $\epsilon$-close to $W_0$, while at the same time $L_\infty$ is not the translation of a cone. By choosing $\kappa$ smaller, we can arrange that the tangent cone at infinity of $L_\infty$ is as close to $W_0$ as we like. 

  It remains to argue that by choosing a diagonal sequence we can find a blowup limit given by a special Lagrangian current $L_\infty$ with tangent cone $W_0$ at infinity, such that $L_\infty$ is not a translate of $W_0$. To see this, suppose that $L^i_\infty$ are special Lagrangian currents, with tangent cone $W_i$ at infinity, and $W_i\to W_0$, such that the $L^i_\infty$ are not translates of cones. By translating and scaling, we can arrange that for some fixed small $\kappa > 0$ we have
  \[ \label{eq:s60} \Theta(L^i_\infty) &= \Theta(W_0) - \kappa, \\
    \Theta(L^i_\infty - x) &\leq \Theta(W_0) - \kappa, \text{ for all } x. \]
  To see this, note that from the proof of the uniqueness of tangent cones by Simon~\cite{SimonAsymptotics}, we have the following. Given $\epsilon >0$ (depending on $W_0$) there is a $\kappa >0$ such that if a special Lagrangian $L$ satisfies that $\Theta(r_1^{-1}L) > \Theta(r_2^{-1}L) - \kappa$ for some $r_1 < \kappa r_2$, and $r_2^{-1}L$ is  a $\kappa$-graph over $W_0$ on the annulus $B_2\setminus B_1$, then $r_1^{-1}L$ is an $\epsilon$-graph over $W_0$ on $B_2\setminus B_1$. Using this one can see that if $\kappa$ is chosen sufficiently small, and $i$ is large, then at any point $p$ of $L^i_\infty$ with density at least $\Theta(W_0) - \kappa$, the tangent cone of $L^i_\infty$ would have to be a small perturbation of $W_i$. In particular the density at $p$ would equal $\Theta(W_0)$, and so $L^i_\infty$ would be a cone centered at $p$. It follows, using that the $L^i_\infty$ are not translates of cones, that if $\kappa$ is chosen sufficiently small, then for large enough $i$ the density of $L^i_\infty$ at every point is strictly less than $\Theta(W_0) - \kappa$. The existence of suitable translates and rescalings of the $L^i_\infty$ satisfying \eqref{eq:s60} follows. 
  
  Up to choosing a subsequence we can assume that $L^i_\infty \to L_\infty$ for a special Lagrangian current, which is not a translation of a cone, and it remains to show that the tangent cone of $L_\infty$ is $W_0$ (i.e. the limit of the tangent cones of the $L^i_\infty$). To see this we use that
  the proof of the uniqueness of tangent cones of minimal surfaces in Simon~\cite{SimonAsymptotics} implies that under the condition that $\Theta(L^i_\infty) \geq \Theta(W_0) - \kappa$, for small enough $\kappa$, we have a uniform rate of convergence of the $L^i_\infty$ to their tangent cones. More precisely, for any small $\epsilon > 0$ there is some large $R > 0$ such that $r^{-1}L^i_\infty$ is an $\epsilon$-graph over $W_i$ on the annulus $A_{2,1}$ for all $r > R$ and all $i$. It follows from this that $r^{-1}L_\infty$ is an $\epsilon$-graph over $W_0$ on $A_{2,1}$, and since $\epsilon$ was arbitrary, this shows that the tangent cone at infinity of $L_\infty$ is $W_0$. It follows that we have a good blowup sequence along the flow as required.

  It remains to deal with the case when in the construction above we have $r_{\tau_i} = 0$ for a sequence $\tau_i \to \infty$. We will see that in this case no singularity forms at $(0,0)$. Note that if $r(\tau_i, x_i)=0$ (and $\kappa$ is sufficiently small), then the flow $M_\tau$ has density greater than 1 at $(\tau_i, x_i)$. Since we are dealing with a smoothly immersed flow, this is only possible if at $(\tau_i, x_i)$ several smooth local components intersect. The argument above, invoking Proposition~\ref{prop:T1Ttheta}, shows that for any $\Lambda > 0$, if $i$ is sufficiently large, then for all $\lambda\in (0, \Lambda)$ the rescaling $\lambda^{-1}(M_{\tau_i}-x_i)$ is graphical over $W_0$ on the annulus $A_{2,1}$. In particular, for any $\Lambda > 0$, once $i$ is large, we have that $M_{\tau_i}\cap B(0,\Lambda)$ is a Lipschitz graph over the translate $W_0 + x_i$ of the cone $W_0$. Since the flow is smoothly immersed, this means that $W_0$ must be a union of transverse subspaces, and $M_{\tau_i}\cap B(0, \Lambda)$ is a union of graphs over these subspaces. Using White's regularity theorem~\cite{White} as in Neves~\cite[Corollary 4.3]{Neves}, it follows that actually no singularity is forming along $L^0_t$ at $(0,0)$. 
\end{proof}

Theorem~\ref{thm:Lawlorblowup} on the existence of blowup limits given by a Lawlor neck follows immediately from Theorem~\ref{thm:goodblowup} and Theorem~\ref{singularuniqueness}, characterizing the possible special Lagrangians asymptotic to the union of two transverse subspaces. Similarly, Corollary~\ref{cor:nostable} is an immediate consequence of Theorem~\ref{thm:goodblowup} and Proposition~\ref{prop:noexactAC}.
Finally, Theorems~\ref{thm:connected2} and \ref{thm:2dim} follow in the same way as the beginning of the proof of Theorem~\ref{mainthmproof}, since in these two cases we are assuming that Conditions (A) hold, with either (A3(a)) or (A3(c)).

\end{document}